\documentclass{amsart}
\usepackage{amsmath,amsthm,amsfonts,amssymb,amscd,overpic,mathrsfs}
\usepackage{dsfont}
\usepackage{amsthm}
\usepackage{amssymb}
\usepackage{graphicx}
\usepackage{hyperref}
\usepackage{enumerate}
\usepackage{float}
\usepackage{epstopdf}
\usepackage{xcolor}
\usepackage{amstext,amsthm,amsfonts,amsmath}
\usepackage{enumerate}
\usepackage[utf8]{inputenc}
\usepackage[font=scriptsize]{caption}
\usepackage{indentfirst}
\usepackage{setspace}

\usepackage{subfigure} 

\newcommand{\norm}[1]{\left\Vert#1\right\Vert}
\newcommand{\set}[1]{\left\{#1\right\}}
\newcommand{\exset}[2]{S_{e}\left(#1,#2\right)}

\newcommand{\inpr}[2]{\left\langle#1,#2\right\rangle}

\newcommand{\R}{\mathbb{R}}

\newcommand{\ri}[1]{\textnormal{ri}\left(#1\right)}
\newcommand{\range}[1]{\textnormal{range}\,#1}
\newcommand{\dom}[1]{\textnormal{dom}\,#1}

\newtheorem*{algorithm}{Algorithm}
\newtheorem{theorem}{Theorem}[section]
\newtheorem{lemma}{Lemma}[section]
\newtheorem{proposition}{Proposition}[section]
\newtheorem{definition}{Definition}

\title{Projective method of multipliers for linearly constrained convex minimization}
\author{Majela Pent\'on Machado}
\address{IMPA, Instituto de Matem\'atica Pura e Aplicada, 22460-320 Rio de Janeiro, RJ, Brazil}
\email{majela@impa.br}
\thanks{This work is part of the author's Ph.D. thesis, written under the supervision of Benar Fux 
Svaiter at IMPA, and supported by CAPES and FAPERJ}

\begin{document}

\subjclass[2010]{49M29, 90C25, 65K05, 68Q25}

\keywords{Constrained optimization, convex programming, complexity, total variation denoising, compressed sensing.}

\begin{abstract}
We present a method for solving linearly constrained convex optimization problems, which is based on the application of known algorithms for finding zeros of the sum of two monotone operators (presented by Eckstein and Svaiter) to the dual problem. We establish convergence rates for the new method, and we present applications to TV denoising and compressed sensing problems.
\end{abstract}

\maketitle

\section{Introduction}
\label{sec:introduction}
A broad class of problems of recent interest in image science and signal processing can 
be posed in the framework of convex optimization. Examples include the TV denoising model \cite{rudin_tv} for image 
processing and basis pursuit, which is well known for playing a central role in the theory of compressed sensing. A general subclass of such programming problems is:
\begin{equation}
\label{eq:convex opt prob}
\begin{split}
\min_{u\in\R^{m_1},v\in\R^{m_2}}\left\lbrace f(u)+g(v)\,:\, Mu+Cv=d \right\rbrace.
\end{split}
\end{equation}
Here $f:\R^{m_{1}}\rightarrow(-\infty,\infty]$ and $g:\R^{m_{2}}\rightarrow(-\infty,\infty]$ are proper closed convex functions, $M:\R^{m_{1}}\rightarrow\R^n$ and $C:\R^{m_2}\rightarrow\R^{n}$ are 
linear operators, and $d\in\R^n$. 

A well-known iterative method for solving optimization problems that have a separable structure as \eqref{eq:convex opt prob} does, is the \emph{Alternating Direction Method of Multipliers} (ADMM), which goes back to the works of Glowinski 
and Marrocco \cite{glowinski_marroco}, and of Gabay and Mercier \cite{gabay1976dual}. ADMM solves the coupled problem 
\eqref{eq:convex opt prob} performing a sequences of steps that decouple functions $f$ and $g$, making it possible to 
exploit the individual structure of these functions. It can be interpreted in terms of alternating minimization, with 
respect to $u$ and $v$, of the augmented Lagrangian function associated with problem \eqref{eq:convex opt prob}. ADMM can 
also be viewed as an instance of the method called \emph{Douglas-Rachford splitting} applied to the dual problem of 
\eqref{eq:convex opt prob}, as was shown by Gabay in \cite{gabay1983chapter}.

Other splitting schemes have been effectively applied to the dual problem of \eqref{eq:convex opt prob}, which is a special case of the problem of finding a zero of the sum of two maximal monotone operators. For example, the \emph{Proximal Forward Backward splitting} method, developed by Lions and Mercier \cite{lions_mercier}, and Passty \cite{passty_1979_FB}, corresponds to the well-known Tseng's \cite{tseng_1991} \emph{Alternating Minimization Algorithm} (AMA) for solving \eqref{eq:convex opt prob}. This method has simpler steps than ADMM, in the former one of the minimizations of the augmented Lagrangian is 
replaced by the minimization of the Lagrangian itself; however, it requires strong convexity of one of the objective functions. 

The goal of our work is to construct an optimization scheme for solving \eqref{eq:convex opt prob} applying a splitting method to its dual problem. Specifically we are interested in the family of splitting-projective methods 
proposed in \cite{eck_sv_08} by Eckstein and Svaiter
to address inclusion problems given by the sum of two maximal monotone operators.
We will apply a specific instance of these algorithms to solve a reformulation of the dual problem
of \eqref{eq:convex opt prob} as the problem of finding a zero of the sum of two maximal monotone operators,
which allows us to obtain a new algorithm for solving this problem.
This iterative method will be referred to as the \emph{Projective Method of Multipliers} (PMM). The 
convergence properties of the PMM will be obtained using the convergence results already established in \cite{eck_sv_08}. 
In contrast to \cite{eck_sv_08}, which only studies the global convergence of the family of splitting-projective methods, we also establish in this work the iteration complexity of the PMM.
Using the Karush-Kuhn-Tucker (KKT) conditions for problem \eqref{eq:convex opt prob} we give convergence rate for the PMM measured by the pointwise and ergodic iteration-complexities. 

The remainder of this paper is organized as follows. Section \ref{sec:preliminaries} reviews some definitions and facts 
on convex functions that will be used in our subsequent presentation. It also briefly discusses Lagrangian duality theory for convex optimization, for more details in this subject we refer the reader to \cite{rock_conv_anal}. Section \ref{sec:pmm} presents the Projective Method of Multipliers (PMM) for solving the class of linearly constrained optimization problems \eqref{eq:convex opt prob}. This section also presents global convergence of the PMM using the convergence analysis presented in \cite{eck_sv_08}. Section \ref{sec:complexity} derives iteration-complexity results for the PMM. Finally, section \ref{sec:application} presents some applications in image restoration and compressed sensing. This section also exhibits numerical results demonstrating the effectiveness of the PMM in solving these problems.

\subsection{Notation}
Throughout this paper, we let $\R^n$ denote an $n$-dimensional space with
inner product and induced norm denoted by $\inpr{\cdot}{\cdot}$ and $\norm{\cdot}$, respectively. For a matrix $A$, $A^T$ indicates its transpose and $\norm{A}_F=\sqrt{trace(AA^T)}$ its Frobenius norm. Given a linear operator $M$, we denote by $M^\ast$ its adjoint operator. If $C$ is a convex set we indicate by $\ri{C}$ its relative interior.

\section{Preliminaries}
\label{sec:preliminaries}
In this section we describe some basic definitions
and facts on convex analysis that will be needed along this work. We also discuss the Lagrangian formulation and dual problem of \eqref{eq:convex opt prob}. This approach will play an important role in the design of the PMM for problem \eqref{eq:convex opt prob}.

\subsection{Generalities on convex functions}

Given an extended real valued convex function $f:\R^n\rightarrow(-\infty,\infty]$, the domain of $f$ is the set 
\begin{equation*}
 \dom{f}=\set{x\in\R^n\,:\,f(x)<\infty}.
\end{equation*}
Since $f$ is a convex function, it is obvious that $\dom{f}$ is convex. We say that function $f$ is \emph{proper} if $\dom{f}\neq\emptyset$. Furthermore, we say that $f$ is \emph{closed} if it is a lower semicontinuous function. 

\begin{definition}
\label{def:subgrad}
Given a convex function $f:\R^n\rightarrow(-\infty,\infty]$ a vector $v\in\R^n$ is called a \emph{subgradient} of $f$ at $x\in\R^n$ if 
\begin{equation*}
f(x') \geq f(x)+\inpr{v}{x'-x}\qquad\qquad\forall x'\in\R^n.
\end{equation*}
The set of all subgradients of $f$ at $x$ is denoted by $\partial f(x)$. The operator $\partial f$, which maps each $x$ to $\partial f(x)$, is called the \emph{subdifferential} map associated with $f$.
\end{definition}

It can be seen immediately from the definition that $x^\ast$ is a global minimizer of $f$ in $\R^n$ if and only if $0\in\partial f(x^\ast)$. If $f$ is differentiable at $x$, then $\partial f(x)$ is the singleton set $\{\nabla f(x)\}$.

The subdifferential mapping of a convex function $f$ has the following \emph{monotonicity} property: for any $x$, $x'$, $v$ and $v'\in\R^n$ such that $v\in\partial f(x)$ and $v'\in\partial f(x')$, it follows that
\begin{equation}
\label{eq:monotonicity}
\inpr{x-x'}{v-v'}\geq 0.
\end{equation}
In addition, if $f$ is a proper closed convex function, then $\partial f$ is a \emph{maximal monotone} operator \cite{rock_1970}. This is to say that if $x,\,v\in\R^n$ are such that inequality \eqref{eq:monotonicity} holds for all $x'\in\R^n$ and $v'\in\partial f(x')$, then $x\in\dom{f}$ and $v\in\partial f(x)$.

Given $\lambda>0$, the \emph{resolvent mapping} (or \emph{proximal mapping}) \cite{moreau1965} associated with $\partial f$ is defined as
\begin{equation*}
(I+\lambda\partial f)^{-1}(z) := \arg\min_{x\in\R^n} \lambda f(x) + \frac{1}{2}\norm{x-z}^2,\qquad\qquad\forall z\in\R^n.
\end{equation*}
The fact that $(I+\lambda\partial f)^{-1}$ is an everywhere well defined function, if $f$ is proper, closed and convex, is a consequence of a fundamental result due to Minty \cite{minty}. For example, if $f(x)=\mu\norm{x}_1=\mu\sum|x_i|$ where $\mu>0$, then
\begin{equation*}
(I+\partial f)^{-1}(z) = \textbf{shrink}(z,\mu),
\end{equation*}
where 
\begin{equation}
\label{eq:shrink}
\textbf{shrink}(z,\mu)_i := \max\{|z_i| - \mu,0\}sign(z_i).
\end{equation}

The \emph{Fenchel-Legendre conjugate} of a convex function $f$, denoted by $f^\ast:\R^n\rightarrow(-\infty,\infty]$, is defined as
\begin{equation*}
 f^\ast(v)=\sup_{x\in\R^n} \inpr{v}{x}-f(x)\,,\qquad\qquad\forall v\in\R^n.
\end{equation*}
It is simple to see that $f^\ast$ is a convex closed function.
Furthermore, if $f$ is proper, closed and convex, then $f^\ast$ is a proper function \cite{brezis_func_anal}.

\begin{definition}
Given any convex function $f:\R^n\to(-\infty,\infty]$ and $\epsilon\geq0$, a vector $v\in\R^n$ is called an $\epsilon$-\emph{subgradient} of $f$ at $x\in\R^n$ if 
\begin{equation*}
f(x')\geq f(x)+\inpr{v}{x'-x}-\epsilon\qquad\qquad\forall x'\in\R^n.
\end{equation*}
The set of all $\epsilon$-subgradients of $f$ at $x$ is denoted by $\partial_\epsilon f(x)$, and $\partial_\epsilon f$ is called the $\epsilon$\emph{-subdifferential} mapping.
\end{definition}

It is trivial to verify that $\partial_0f(x)=\partial f(x)$, and $\partial f(x)\subseteq\partial_\epsilon f(x)$ for every $x\in\R^n$ and $\epsilon\geq0$. The proposition below lists some useful properties of the $\epsilon$-subdifferential that will be needed in our presentation.
\begin{proposition}
\label{Prop:subdifferential}
If $f:\R^n\to(-\infty,\infty]$ is a proper closed convex function, $g:\R^n\to\R$ is a convex differentiable function in $\R^n$, and $M:\R^m\to\R^n$ is a linear transformation, then the following statements hold:
\begin{itemize}
\item[(a)] $v\in\partial_\epsilon f(x)$ if and only if $x\in\partial_\epsilon f^\ast(v)$ for all $\epsilon\geq0$;
\item[(b)] $\partial(f+g)(x)=\partial f(x)+\nabla g(x)$ for all $x\in\R^n$;
\item[(c)] $\partial(f\circ M)(x) \supseteq M^\ast\partial f(Mx)$ for all $x\in\R^m$. In addition, if $\ri{\dom{f}}\cap\range{M}\neq\emptyset$, then $\partial (f\circ M)(x)=M^\ast\partial f(Mx)$ for every $x\in\R^m$;
\item[(d)] if $x_{i},v_{i}\in\mathbb{R}^{n}$ and $\epsilon_{i},\alpha_{i}\in\mathbb{R}_{+}$, for $i=1,\ldots,k$, are such that
\begin{align*}
v_{i}\in \partial_{\epsilon_i} f(x_{i}),\quad i=1,\ldots,k, \qquad\qquad\sum_{i=1}^{k}\alpha_{i}=1,
\end{align*}
and we define 
\begin{equation*}
\overline{x} = \sum_{i=1}^{k}\alpha_{i}x_{i}, \qquad\quad \overline{v}=\sum_{i=1}^{k}\alpha_{i}v_{i}, \qquad\quad
\overline{\epsilon}=\sum_{i=1}^{k}\alpha_{i}(\epsilon_{i}+\inpr{x_{i}-\overline{x}}{v_{i}});
\end{equation*}
then, we have $\overline{\epsilon}\geq0$ and $\overline{v}\in \partial_{\overline{\epsilon}}f(\overline{x})$.
\end{itemize}
\end{proposition}
\begin{proof}
Statements (a)-(c) are classical results which can be found, for example, in \cite{lemarechal_2} and \cite{rock_conv_anal}. For a proof of item (d) see \cite{bur_sag_sv_99} and references therein.
\end{proof}

\subsection{Lagrangian duality}

The \emph{Lagrangian function} $L:\R^{m_1}\times\R^{m_2}\times\R^n\rightarrow(-\infty,\infty]$ for problem \eqref{eq:convex opt prob} is defined as
\begin{equation}
\label{eq:lagrangian}
L(u,v,z) = f(u)+g(v)+\inpr{Mu+Cv-d}{z}.
\end{equation}
The \emph{dual function} is the concave function $\varphi:\R^n\to[-\infty,\infty)$ defined by
\begin{equation*}
\begin{split}
\varphi(z)&=\inf_{(u,v)\in\R^{m_1}\times\R^{m_2}}L(u,v,z),
\end{split}
\end{equation*}
and the \emph{dual problem} to \eqref{eq:convex opt prob} is
\begin{equation}
\label{eq:dual problem}
\max_{z\in\R^n}\varphi(z).
\end{equation}
Problem \eqref{eq:convex opt prob} will be called the \emph{primal problem}. Straightforward calculations show that weak duality holds, i.e. $\varphi^\ast\leq p^\ast$, where $p^\ast$ and $\varphi^\ast$ are the optimal values of 
(\ref{eq:convex opt prob}) and (\ref{eq:dual problem}), respectively. 

A vector $(u^\ast,v^\ast,z^\ast)$ such that $L(u^{\ast},v^{\ast},z^{\ast})$ is finite and it satisfies
\begin{equation}
\label{saddle point}
\min_{(u,v)\in\R^{m_1}\times\R^{m_2}}L(u,v,z^\ast)=L(u^\ast,v^\ast,z^\ast)
=\max_{z\in\R^n}L(u^\ast,v^\ast,z)
\end{equation}
is called a \emph{saddle point} of the Lagrangian function $L$.
Finding optimal solutions of problems \eqref{eq:convex opt prob} and \eqref{eq:dual problem} is equivalent to finding saddle points of $L$ (see \cite{rock_conv_anal}). That is, $(u^\ast,v^\ast)$ is an optimal primal
solution and $z^\ast$ is an optimal dual solution if and only if $(u^\ast,v^\ast,z^\ast)$ is a saddle point. Furthermore, if a saddle point 
of $L$ exists then $p^\ast=\varphi^\ast$, i.e. there is no duality gap \cite{rock_conv_anal}.

Notice that, if $(u^\ast,v^\ast,z^\ast)$ is a saddle point, from the definition of $L$ in \eqref{eq:lagrangian} and equalities \eqref{saddle point} we deduce that
\begin{equation*}
 f(u) + g(v) + \inpr{Mu+Cv-d}{z^\ast}\geq L(u^\ast,v^\ast,z^\ast)\geq f(u^\ast) + g(v^\ast) + \inpr{Mu^\ast+Cv^\ast-d}{z}
\end{equation*}
for all $u\in\R^{m_1}$, $v\in\R^{m_2}$, $z\in\R^n$. From these relations we can directly derive the \emph{Karush-Kuhn-Tucker} (KKT) conditions
\begin{equation}
\label{Kuhn Tucker}
 \begin{split}
  & 0= Mu^\ast+Cv^\ast -d,\\
  & 0\in\partial f (u^\ast) + M^\ast z^\ast,\\
  & 0\in\partial g (v^\ast) + C^\ast z^\ast,
 \end{split}
\end{equation}
which describe  an optimal solution of problem \eqref{eq:convex opt prob}.
Observe that the equality in \eqref{Kuhn Tucker} implies that the primal variables $(u^\ast,v^\ast)$ must be feasible. The inclusions in \eqref{Kuhn Tucker} are known as the dual feasibility conditions. We also have that the KKT conditions hold if and only if $(u^\ast,v^\ast,z^\ast)$ is a saddle point of $L$.

Observe that the dual function $\varphi$ can be written in terms of the Fenchel-Legendre conjugates of the functions $f$ and
$g$. Specifically,
\begin{align*}
 \varphi(z) = &\inf_{(u,v)\in\R^{m_1}\times\R^{m_2}} f(u)+g(v)+\inpr{Mu+Cv-d}{z}\\
 = &\inf_{u\in\R^{m_1}} f(u) + \inpr{Mu}{z} + \inf_{v\in\R^{m_2}} g(v) + \inpr{Cv}{z} - \inpr{d}{z}\\
 = & - f^\ast(-M^\ast z) - g^\ast(-C^\ast z) - \inpr{d}{z}.
\end{align*}
Hence, if we define the functions $h_1(z)=\left(f^\ast\circ-M^\ast\right)(z)$ and $h_2(z)=\left(g^\ast\circ-C^\ast\right) (z) + \inpr{d}{z}$, we have that the dual problem (\ref{eq:dual problem}) is equivalent to minimizing $h_1+h_2$ over $\R^n$. Furthermore, since $f^\ast$ and $g^\ast$ are convex and closed, and $M^\ast$ and $C^\ast$ are linear operators, it follows that $h_1$ and $h_2$ are convex closed functions \cite{rock_conv_anal}. Therefore, $z^\ast$ is a solution of \eqref{eq:dual problem} if and only if 
\begin{equation}
\label{eq:problem h1_h2}
 0 \in \partial (h_1+h_2)(z^\ast).
\end{equation}

Throughout this work, we assume that
\begin{enumerate}
\item[(A.1)] there exists $(u^{\ast},v^{\ast},z^{\ast})$ a saddle point of $L$.
\end{enumerate} 

Since condition A.1 implies that the KKT conditions hold, we have from the first inclusion in \eqref{Kuhn Tucker} and Proposition \ref{Prop:subdifferential}(a),(c) that $z^\ast\in\dom{(f^\ast\circ-M^\ast)}$, which implies that $h_1$ is a proper function. A similar argument shows that $h_2$ is also a proper function. Therefore, under hypothesis A.1, we have that the subdifferentials $\partial h_1$ and $\partial h_2$ are maximal monotone operators.

\section{The Projective Method of Multipliers}
\label{sec:pmm}
Our proposal in this work is to apply the splitting-projective methods developed in \cite{eck_sv_08}, by Eckstein and Svaiter, to find a solution of problem
\begin{equation*}
0\in\partial h_1(z) + \partial h_2(z),
\end{equation*}
and as a consequence a solution of the dual problem \eqref{eq:dual problem}, since the following inclusion holds
\begin{equation*}
\partial h_1(z) + \partial h_2(z) \subseteq \partial(h_1+h_2)(z)\qquad\quad\forall z\in\R^n
\end{equation*}
(see equation \eqref{eq:problem h1_h2} and the comments above).

The framework presented in \cite{eck_sv_08} reformulates the problem of finding a zero of the sum of two maximal monotone operators in terms of a convex feasibility problem, which is defined by a certain closed convex ``extended" solution set. To solve the feasibility problem, the authors introduced successive projection algorithms that use, on each iteration, independent calculations involving each operator.

Specifically, if we consider the subdifferential mappings $\partial h_1$ and $\partial h_2$, then the associated extended solution set, defined as in \cite{eck_sv_08}, is
 \begin{equation}
 \label{eq:ext sol set}
 S_e(\partial h_1,\partial h_2):=\set{(z,w)\in\R^n\times\R^n\,:\,-w\in\partial h_1(z)\,,\,w\in\partial h_2(z)}.
 \end{equation}
Since $\partial h_1$ and $\partial h_2$ are maximal monotone operators it can be proven that $S_e(\partial h_1,\partial h_2)$ is a closed convex set in $\R^n\times\R^n$, see \cite{eck_sv_08}. It is also easy to verify that if $(z^\ast,w^\ast)$ is a point in $S_e(\partial h_1,\partial h_2)$ then $z^\ast$ satisfies inclusion \eqref{eq:problem h1_h2} and consequently it is a solution of the dual problem. Furthermore, the following lemma holds.

\begin{lemma}
\label{Lem:saddle}
If $(u^\ast,v^\ast,z^\ast)$ is a saddle point of $L$, then $$(z^\ast,d-Cv^\ast)\in\exset{\partial h_1}{\partial h_2}.$$
Moreover, if we assume the following conditions 
\begin{enumerate}
\item[(A.2)] $\ri{\dom{f^\ast}}\cap \range{M^\ast}\neq\emptyset$;
\item[(A.3)] $\ri{\dom{g^\ast}}\cap \range{C^\ast}\neq\emptyset$.
\end{enumerate}
Then, for all $(z^\ast,w^\ast)\in S_e(\partial h_1,\partial h_2)$ there exist $u^\ast,v^\ast\in\R^n$ such that $w^\ast=d-Cv^\ast$, $w^\ast=Mu^\ast$ and $(u^\ast,v^\ast,z^\ast)$ is a saddle point of the Lagrangian function $L$.
\end{lemma}

\begin{proof}
If $(u^\ast,v^\ast,z^\ast)$ is a saddle point of the Lagrangian function, then the KKT optimality conditions hold, and the inclusions in \eqref{Kuhn Tucker}, together with
Proposition \ref{Prop:subdifferential}(a), imply that
\begin{equation*}
 u^\ast\in\partial f^\ast(-M^\ast z^\ast) \qquad\text{and}\qquad v^\ast\in\partial g^\ast(-C^\ast z^\ast).
\end{equation*}
Thus, we have
\begin{align}
\label{eq:f conj}
 &-Mu^\ast\in-M\partial f^\ast(-M^\ast z^\ast)\subseteq\partial(f^\ast\circ-M^\ast)(z^\ast) = \partial h_1(z^\ast)
 \end{align}
 and
 \begin{align}
 \label{eq:g conj}
 &-Cv^\ast\in-C\partial g^\ast(-C^\ast z^\ast)\subseteq\partial(g^\ast\circ-C^\ast)(z^\ast);
\end{align}
where the second inclusions in \eqref{eq:f conj} and \eqref{eq:g conj} follow from Proposition \ref{Prop:subdifferential}(c). Adding $d$ to both sides of \eqref{eq:g conj} and
using the definition of $h_2$ and Proposition \ref{Prop:subdifferential}(b) we have $d-Cv^\ast\in\partial h_2(z^\ast)$. Now, adding this last inclusion to \eqref{eq:f conj} we conclude that 
$$-Mu^\ast+d-Cv^\ast\in\partial h_1(z^\ast) + \partial h_2(z^\ast).$$
The first assertion of the lemma follows combining the relation above with the equality in \eqref{Kuhn Tucker} and the
definition of $\exset{\partial h_1}{\partial h_2}$.

By \eqref{eq:ext sol set} we have that if $(z^\ast,w^\ast)\in S_e(\partial h_1,\partial h_2)$ then $w^\ast\in\partial h_2(z^\ast)=-C\partial g^\ast(-C^\ast z^\ast) + d$, where the equality follows from condition A.3 and Proposition \ref{Prop:subdifferential}(b),(c). Thus, there exists $v^\ast\in\partial g^\ast(-C^\ast z^\ast)$ such that $w^\ast=-Cv^\ast+d$, and applying Proposition \ref{Prop:subdifferential}(a) we obtain that $-C^\ast z^\ast\in\partial g(v^\ast)$.

Equivalently, using $-w^\ast\in\partial h_1(z^\ast)$, hypothesis A.2 and Proposition \ref{Prop:subdifferential}(a),(c), we deduce that there is a $u^\ast$ such that  $-w^\ast=-Mu^\ast$ and $-M^\ast z^\ast\in\partial f(u^\ast)$. All these conditions put together imply that $(u^\ast,v^\ast,z^\ast)$ is a saddle point of $L$.
\end{proof}

According to Lemma \ref{Lem:saddle}, we can attempt to find a saddle point of the Lagrangian function \eqref{eq:lagrangian}, by seeking a point in the extended solution set $S_e(\partial h_1,\partial h_2)$.

In order to solve the feasibility problem defined by $S_e(\partial h_1,\partial h_2)$, by successive orthogonal projection 
methods, the authors of \cite{eck_sv_08} used the resolvent mappings associated with the operators to construct affine 
separating hyperplanes.

In our setting the family of algorithms in \cite{eck_sv_08} follows the set of recursions
\begin{align}
\label{eq:project-split-alg}
&\lambda_kb_k + x_k = z_{k-1} + \lambda_kw_{k-1},&& b_k\in\partial h_2(x_k);\\
\label{eq:project-split-alg 1}
& \mu_ka_k + y_k = (1-\alpha_k)z_{k-1} + \alpha_kx_k - \mu_kw_{k-1},&& a_k \in\partial h_1(y_k);\\
\label{eq:project-split-alg 2}
& \gamma_{k}=\frac{\inpr{z_{k-1}-x_{k}}{b_{k}-w_{k-1}} + \inpr{z_{k-1}-y_{k}}{a_{k}+w_{k-1}}}{\norm{a_{k}+b_{k}}^{2} + \norm{x_{k}-y_{k}}^{2}},\\
\label{eq:project-split-alg 3}
& z_{k} = z_{k-1} - \rho_{k}\gamma_{k}(a_{k}+b_{k}),\\
\label{eq:project-split-alg 4}
& w_{k} = w_{k-1} - \rho_{k}\gamma_{k}(x_{k}-y_{k}),
\end{align}
where $\lambda_k$, $\mu_k>0$ and $\alpha_k\in\R$ are such that $(\mu_{k}/\lambda_{k}-(\alpha_{k}/2)^{2})>0$, and $\rho_k\in(0,2)$.

We observe that relations in \eqref{eq:project-split-alg} and the definition of the resolvent mapping yield that $x_k=(I+\lambda_k\partial h_2)^{-1}(z_{k-1} + \lambda_kw_{k-1})$ and $b_k=\frac{1}{\lambda_k}(z_{k-1}-x_k)+w_{k-1}$. Similarly, \eqref{eq:project-split-alg 1} implies that $y_k=(I+\mu_k\partial h_1)^{-1}((1-\alpha_k)z_{k-1} + \alpha_kx_k - \mu_kw_{k-1}))$ and $a_k=\frac{1}{\mu_k}((1-\alpha_k)z_{k-1}+\alpha_k x_k - y_k) - w_{k-1}$. Hence, steps \eqref{eq:project-split-alg} and \eqref{eq:project-split-alg 1} are evaluations of the proximal mappings.

With the view to see that iterations \eqref{eq:project-split-alg}-\eqref{eq:project-split-alg 4} truly are successive (relaxed) projection methods for the convex feasibility problem of finding a point in $S_e(\partial h_1,\partial h_2)$, we define, for all integer $k\geq1$, the affine function  $\phi_k(z,w):\R^n\times\R^n\to\R$ as
\begin{equation}
\label{eq:phi_k}
\phi_k(z,w) = \inpr{z-x_{k}}{b_{k}-w} + \inpr{z-y_{k}}{a_{k}+w},
\end{equation}
and its non-positive level set
  \begin{equation*}
      H_{\phi_k}=\set{(z,w)\;:\;\phi_k(z,w) \leq 0}.
  \end{equation*}
Thus, by the monotonicity of the subdifferential mappings we have that $S_e(\partial h_1,\partial h_2) \subseteq H_{\phi_k}$ and it is also easy to verify that the following relations hold 
\begin{align}
\label{eq:gradient ph_k}
&\nabla\phi_k = (a_k+b_k,x_k-y_k),\\
\label{eq:gamma_k}
&\gamma_k=\frac{\phi_k(z_{k-1},w_{k-1})}{\norm{\nabla\phi_k}^2}\quad\text{ and }\quad\gamma_k\geq0,
\end{align}
for all integer $k\geq1$. Therefore, we conclude that if $\rho_k=1$ the point $(z_k,w_k)$, calculated by the update rule given by \eqref{eq:project-split-alg 3}-\eqref{eq:project-split-alg 4}, is the orthogonal projection of $(z_{k-1},w_{k-1})$ onto $H_{\phi_k}$. Besides, if $\rho_k\neq1$ we have that $(z_k,w_k)$ is an under relaxed projection of $(z_{k-1},w_{k-1})$.

As was observed in the paragraph after \eqref{eq:project-split-alg 4}, in order to apply algorithm \eqref{eq:project-split-alg}-\eqref{eq:project-split-alg 4} it is necessary to calculate the resolvent mappings associated with $\partial h_1$ and $\partial h_2$. The next result shows how we can invert operators $I+\lambda\partial h_1$ and $I+\lambda\partial h_2$ for any $\lambda>0$.

\begin{lemma}
\label{Lem:resolvent h_1 h_2}
Consider $c\in\R^n$, $\theta:\R^m\to(-\infty,\infty]$ a proper closed convex function and $A:\R^m\to\R^n$ a linear operator such that $\dom{\theta^\ast}\cap\range{A^\ast}\neq\emptyset$. Let $z\in\R^n$ and $\lambda>0$. Then, if $\tilde{\nu}\in\R^{m}$ is a solution of problem
\begin{equation}
\label{sub problem g}
\min_{\nu\in\R^{m}} \theta(\nu)+\inpr{z}{A\nu-c}+\frac{\lambda}{2}\norm{A\nu-c}^{2}
\end{equation}
it holds that $c-A\tilde{\nu}\in\partial h(\hat{z})$ where $h(\cdot)=(\theta^\ast\circ-A^\ast)(\cdot)+\inpr{c}{\cdot}$ and $\hat{z}=z+\lambda(A\tilde{\nu}-c)$. Hence, $\hat{z}=(I+\lambda\partial h)^{-1}(z)$.
Furthermore, the set of optimal solutions of (\ref{sub problem g}) is nonempty.
\end{lemma}

\begin{proof}
If $\tilde{\nu}\in\R^{m}$ is a solution of (\ref{sub problem g}), deriving the optimality condition of this minimization problem, we have
\begin{equation*}
\begin{split}
0\in \partial \theta(\tilde{\nu}) + A^\ast z + \lambda A^\ast (A\tilde{\nu}-c) = \partial \theta(\tilde{\nu}) + A^\ast(z+\lambda (A\tilde{\nu} - c)).
\end{split}
\end{equation*}
From the definition of $\hat{z}$ and the identity above it follows that 
\begin{equation*}
0\in \partial \theta(\tilde{\nu}) + A^\ast\hat{z}.
\end{equation*}
Now, by equation above and Proposition \ref{Prop:subdifferential}(a),(c) we have
\begin{equation}
\label{eq:subd-g}
-A\tilde{\nu}\in\partial(\theta^\ast\circ-A^\ast)(\hat{z}).
\end{equation}
 
Since we are assuming that $\dom{\theta^\ast}\cap\range{A^\ast}\neq\emptyset$, the definition of $h$ and Proposition \ref{Prop:subdifferential}(b),(c) yield
\begin{equation}
\label{subdf h_2}
\partial h(z) = \partial(\theta^\ast\circ-A^\ast)(z) + c =-A\partial \theta^\ast(-A^\ast z) +c, \qquad\quad\forall z\in\R^n.
\end{equation}
Therefore, adding $c$ to both sides of \eqref{eq:subd-g} and combining with the equation above we deduce that $c-A\tilde{\nu}\in\partial h(\hat{z})$. 
The assertion that $\hat{z}=(I+\lambda\partial h)^{-1}(z)$ is a direct consequence of this last inclusion and the definition of $\hat{z}$.

Next, we notice that, since $\partial h$ is maximal monotone, Minty's theorem \cite{minty} asserts that for all $z\in\R^n$ and $\lambda>0$ there exist $\tilde{z}$, $w\in\R^n$ such that 
\begin{equation}
\label{proxy h_1}
\begin{cases}
&w\in\partial h(\tilde{z}),\\
&\lambda w+\tilde{z}=z.
\end{cases}
\end{equation}
Therefore, the inclusion above, together with equation \eqref{subdf h_2}, implies that there exits $\overline{\nu}\in\partial \theta^\ast(-A^\ast\tilde{z})$ such that $w=-A\overline{\nu}+c$. This last inclusion yields $-A^\ast\tilde{z}\in\partial \theta(\overline{\nu})$, from which we deduce that
\begin{equation*}
\begin{split}
0\in \partial \theta(\overline{\nu}) + A^\ast\tilde{z} = \partial \theta(\overline{\nu}) + A^\ast(z-\lambda w),
\end{split}
\end{equation*}
where the equality above follows from the equality in \eqref{proxy h_1}.
Finally, replacing $w$ by $c-A\overline{\nu}$ in the equation above, we obtain 
\begin{equation*}
0\in \partial \theta(\overline{\nu}) + A^\ast(z+\lambda(A\overline{\nu}-c)),
\end{equation*}
from which follows that $\overline{\nu}$ is an optimal solution of problem (\ref{sub problem g}).
\end{proof}

In what follows we assume that conditions A.2 and A.3 are satisfied. We can now introduce the Projective Method of Multipliers.

\begin{algorithm}[\textbf{PMM}]
\label{algh:pmp}
Let $(z_0,w_0)\in\R^n\times\R^n$, $\lambda>0$ and $\overline{\rho}\in[0,1)$ be given. For $k=1,2,\ldots$.
\begin{itemize}
\item[1.] Compute $v_k\in\R^{m_2}$ as
\begin{equation}
\label{eq:problem g_k}
v_k=\arg \min_{v\in\R^{m_2}} g(v)+\inpr{z_{k-1}+\lambda w_{k-1}}{Cv-d}+\frac{\lambda}{2}\norm{Cv-d}^{2},
\end{equation}
and $u_k\in\R^{m_1}$ as
\begin{equation}
\label{eq:problem f_k}
u_k=\arg \min_{u\in\R^{m_1}} f(u)+\inpr{z_{k-1}+\lambda(Cv_k-d)}{Mu}+\frac{\lambda}{2}\norm{Mu}^{2}.
\end{equation}
\item[2.] If $\norm{Mu_k+Cv_k-d}+\norm{Mu_k-w_{k-1}}=0$ stop. Otherwise, set
\begin{equation*}
\gamma_k = \frac{\lambda\norm{Cv_k-d+w_{k-1}}^2+\lambda\inpr{d-Cv_k-Mu_k}{w_{k-1}-Mu_k}}{\norm{Mu_k+Cv_k-d}^2+\lambda^2\norm{Mu_k-w_{k-1}}^2}.
\end{equation*}
\item[3.] Choose $\rho_k\in[1-\overline{\rho},1+\overline{\rho}]$ and set
\begin{align*}
& z_{k} = z_{k-1} + \rho_k\gamma_k(Mu_k+Cv_k-d),\\
& w_{k} = w_{k-1} - \rho_k\gamma_k\lambda(w_{k-1}-Mu_k).
\end{align*}
\end{itemize}
\end{algorithm}

\begin{proposition}
\label{Prop:conv}
The PMM is a special instance of algorithm \eqref{eq:project-split-alg}-\eqref{eq:project-split-alg 4} where
\begin{equation}
\label{eq:definition lambda mu}
\lambda_k=\mu_k=\lambda,\qquad\qquad \alpha_k=1,
\end{equation}
and
\begin{equation}
\label{def abxy conv}
\begin{aligned}
& x_k=z_{k-1}+\lambda w_{k-1} + \lambda(Cv_k-d), &&\qquad\quad b_k=d-Cv_k\in\partial h_2(x_k),\\
& y_k=x_k-\lambda(w_{k-1}-Mu_k), &&\qquad\quad a_k=-Mu_k\in\partial h_1(y_k),
\end{aligned}
\end{equation}
for every integer $k\geq1$.
\end{proposition}

\begin{proof}
First we notice that \eqref{eq:definition lambda mu} implies
\begin{equation}
\label{eq:convergence condition}
\frac{\lambda_k}{\mu_k}-\left(\frac{\alpha_k}{2}\right)^2=\frac{\lambda}{\lambda}-\left(\frac{1}{2}\right)^2=\dfrac34,
\end{equation}
for all integer $k\geq1$. Next, applying Lemma \ref{Lem:resolvent h_1 h_2} with $\theta=g$, $A=C$, $c=d$, $z=z_{k-1}+\lambda w_{k-1}$ and $\tilde{\nu}=v_k$ we have that $x_k$ and $b_k$, defined as in \eqref{def abxy conv}, satisfy $b_k\in\partial h_2(x_k)$ and $x_k=(I+\lambda\partial h_2)^{-1}(z_{k-1}+\lambda w_{k-1})$. Therefore, the pair $(x_k,b_k)$ satisfies the relations in \eqref{eq:project-split-alg} with $\lambda_k=\lambda$. 

Similarly, applying Lemma \ref{Lem:resolvent h_1 h_2} with $\theta=f$, $A=M$, $c=0$, $z=x_k-\lambda w_{k-1}$ and $\tilde{\nu}=u_k$ we have that the points $y_k$ and $a_k$, given in \eqref{def abxy conv}, satisfy \eqref{eq:project-split-alg 1} with $\mu_k=\lambda$, $\alpha_k=1$ and $x_k$ defined in \eqref{def abxy conv}. 

Moreover, identities in (\ref{def abxy conv}) yield
\begin{equation}
\label{gradient phi conv}
b_k + a_k = d-Cv_k -Mu_k,\qquad\quad x_k-y_k = \lambda(w_{k-1}-Mu_k),
\end{equation}
and
\begin{equation}
\label{eq:z-y}
z_{k-1}-y_k= \lambda(d-Mu_k-Cv_k).
\end{equation}
Using \eqref{gradient phi conv}, \eqref{eq:z-y} and the definitions of $x_k$, $b_k$, $y_k$ and $a_k$ in \eqref{def abxy conv}, we can rewrite $\gamma_k$ in step 2 of the PMM as
\begin{equation*}
\gamma_{k}=\frac{\inpr{z_{k-1}-x_k}{b_k-w_{k-1}}+\inpr{z_{k-1}-y_k}{a_k+w_{k-1}}}{\norm{a_k+b_k}^2+\norm{x_k-y_k}^2},
\end{equation*}
which is exactly equation \eqref{eq:project-split-alg 2}. Finally, (\ref{gradient phi conv}) and the update rule in step 3 of the PMM imply that
\begin{equation*}
\begin{split}
& z_{k} = z_{k-1} - \rho_k\gamma_k(a_k+b_k),\\
& w_{k} = w_{k-1} - \rho_k\gamma_k(x_k-y_k).
\end{split}
\end{equation*}
Thus, the proposition is proven.
\end{proof}

From Proposition \ref{Prop:conv} and equalities in \eqref{gradient phi conv} it follows that if for some $k$ the stopping criterion in step 2 of the PMM holds, then 
\begin{equation}
\label{eq:stopping step 2-1}
Mu_k+Cv_k-d=0 \qquad\text{and}\qquad x_k-y_k = 0.
\end{equation}
Furthermore, by the definitions of $x_k$ and $y_k$ in \eqref{def abxy conv}, and the optimality conditions of problems
\eqref{eq:problem g_k} and \eqref{eq:problem f_k}, we have
\begin{equation}
\label{eq:stopping step 2-2}
0\in\partial g(v_k) + C^\ast x_k \qquad\text{and}\qquad 0\in\partial f(u_k) + M^\ast y_k,
\end{equation}
for all integer $k\geq1$. Combining \eqref{eq:stopping step 2-1} with \eqref{eq:stopping step 2-2} we may conclude that if the PMM stops in step 2, then $(u_k,v_k,x_k)$ satisfies the KKT conditions, and consequently it is a saddle point of $L$.

Otherwise, if the PMM generates an infinite sequence, in view of Proposition \ref{Prop:conv}, we are able to establish its global convergence using the convergence results presented in \cite{eck_sv_08}.

\begin{theorem}
Consider the sequences $\{(u_k,v_k)\}$, $\{(z_k,w_k)\}$, $\{\gamma_k\}$  and $\{\rho_k\}$ generated by the PMM. Consider also the sequences $\{x_k\}$, $\{b_k\}$, $\{y_k\}$ and $\{a_k\}$ defined in (\ref{def abxy conv}). Then,  the following statements hold.
\begin{itemize}
\item[(a)] There exist $z^\ast$ a solution of the dual problem \eqref{eq:dual problem} and $w^\ast\in\R^n$ such that $-w^\ast\in\partial h_1(z^\ast)$, $w^\ast\in\partial h_2(z^\ast)$, and $(x_k,b_k)\to(z^\ast,w^\ast)$, $(y_k,-a_k)\to(z^\ast,w^\ast)$ and $(z_k,w_k)\to(z^\ast,w^\ast)$.
\item[(b)] $Mu_k+Cv_k-d\to0$ and $x_k-y_k\to0$.
\item[(c)] $\lim\limits_{k\to\infty}f(u_k)+g(v_k)=p^\ast$.
\end{itemize}
\end{theorem}

\begin{proof}
\item[(a)] According to Proposition \ref{Prop:conv} the PMM is an instance of the algorithms in \cite{eck_sv_08} applied to the subdifferential 
operators $\partial h_1$ and $\partial h_2$, and with generated sequences $\{(z_k,w_k)\}$, calculated by step 3 of the PMM, and $\{(x_k,b_k)\}$, $\{(y_k,a_k)\}$, which are defined in \eqref{def abxy conv}. From assumption A.1 and equation \eqref{eq:convergence condition} it follows that the hypotheses of \cite[Proposition 3]{eck_sv_08} are satisfied. Thus, invoking this 
proposition we have that 
there exists $(z^\ast,w^\ast)\in S_e(\partial h_1,\partial h_2)$ such that
\begin{equation}
\label{eq:convergence limit}
(z_k,w_k)\to(z^\ast,w^\ast),\qquad(x_k,b_k)\to(z^\ast,w^\ast)\qquad\text{and}\qquad (y_k,-a_k)\to(z^\ast,w^\ast).
\end{equation}
Moreover, since $(z^\ast,w^\ast)\in S_e(\partial h_1,\partial h_2)$ we have that $-w^\ast\in\partial h_1(z^\ast)$, $w^\ast\in\partial h_2(z^\ast)$ and $z^\ast$ is a solution of the dual problem \eqref{eq:dual problem}. 

\item[(b)] By \eqref{eq:convergence limit} it trivially follows that $x_k-y_k\to0$ and $a_k+b_k\to0$. Hence, using the definition of $a_k$ and $b_k$ we deduce that $Mu_k+Cv_k-d\to0$.

\item[(c)] Let $(u^\ast,v^\ast,z^\ast)$ be a KKT point of $L$, which exists from hypothesis A.1, then from the first equality in \eqref{saddle point} we have 
\begin{equation*}
L(u^\ast,v^\ast,z^\ast) \leq L(u_k,v_k,z^\ast), \qquad\quad\text{for }\, k=1,2,\ldots.
\end{equation*}
From equation above, the definition of the Lagrangian function in \eqref{eq:lagrangian} and the KKT conditions \eqref{Kuhn Tucker} it follows that 
\begin{equation*}
f(u^\ast) + g(v^\ast) \leq f(u_k) + g(v_k) + \inpr{Mu_k+Cv_k-d}{z^\ast}.
\end{equation*}
Since $p^\ast=f(u^\ast)+g(v^\ast)$, combining inequality above with item (b) we deduce that
\begin{equation}
\label{eq:liminf-f+g}
p^\ast\leq\liminf_{k\to\infty} f(u_k) + g(v_k).
\end{equation}

Now, we observe that the first inclusion in \eqref{eq:stopping step 2-2}, together with Definition \ref{def:subgrad}, implies
\begin{equation*}
g(v^\ast) \geq g(v_k) - \inpr{C^\ast x_k}{v^\ast-v_k}.
\end{equation*}
Equivalently, from the second inclusion in \eqref{eq:stopping step 2-2} and Definition \ref{def:subgrad} it follows that
\begin{equation*}
f(u^\ast) \geq f(u_k) - \inpr{M^\ast y_k}{u^\ast-u_k}.
\end{equation*}
Adding the two equations above we obtain
\begin{equation*}
\begin{split}
p^\ast \geq &\, f(u_k) + g(v_k) - \inpr{C^\ast x_k}{v^\ast-v_k} - \inpr{M^\ast y_k}{u^\ast-u_ k}\\
= &\, f(u_k) + g(v_k) - \inpr{x_k}{Cv^\ast-Cv_k} - \inpr{y_k}{Mu^\ast-Mu_k}\\
= &\, f(u_k) + g(v_k) - \inpr{x_k-y_k}{Cv^\ast-Cv_k} - \inpr{y_k}{d-Mu_k-Cv_k},
\end{split}
\end{equation*}
where the last equality follows from a simple manipulation and the equality in \eqref{Kuhn Tucker}. Since $\{b_k=d-Cv_k\}$ and $\{y_k\}$ are convergent sequences, therefore bounded sequences, equation above, together with item (b), yields
\begin{equation*}
p^\ast \geq \limsup_{k\to\infty} f(u_k) + g(v_k).
\end{equation*}
Combining inequality above with \eqref{eq:liminf-f+g} we conclude the proof.
\end{proof}

\section{Complexity results}
\label{sec:complexity}

Our goal in this section is to study the iteration complexity of the PMM for solving problem \eqref{eq:convex opt prob}. In order to develop global convergence bounds for the method we will examine how well its iterates satisfy the KKT conditions. 
Observe that the inclusions in \eqref{eq:stopping step 2-2} indicate that the quantities $\norm{Mu_k+Cv_k-d}$ and $\norm{x_k-y_k}$ can be used to measure the accuracy of an iterate $(u_k,v_k,x_k)$ to a saddle point of the Lagrangian function. More specifically, if we define the primal and dual residuals, associated with $(u_k,v_k,x_k)$, by
\begin{equation*}
\begin{split}
& r^p_k = Mu_k+Cv_k-d,\\
& r^d_k = x_k-y_k;
\end{split}
\end{equation*}
then, from the inclusions in \eqref{eq:stopping step 2-2} and the KKT conditions it follows that when $\norm{r_k^p}=\norm{r_k^d}=0$, the triplet $(u_k,v_k,x_k)$ is a saddle point of $L$. 
Therefore, the size of these residuals indicates how far the iterates are from a saddle point, and it can be viewed as an error measurement of the PMM. It is thus reasonable to seek upper bounds for these quantities for the purpose of investigating the convergence rate of the PMM.

The \,theorem \,below estimates the quality of the best iterate among $(u_1,v_1,x_1),\dots,(u_k,v_k,x_k)$, in terms of the error measurement given by the primal and dual residuals. We refer to these estimates as \emph{pointwise} complexity bounds for the PMM.

\begin{theorem}
\label{Teo:complexity-point}
Consider the sequences $\{(u_k,v_k)\}$, $\{(z_k,w_k)\}$, $\{\gamma_k\}$  and $\{\rho_k\}$ generated by the PMM. Consider also the sequences $\{x_k\}$, $\{b_k\}$, $\{y_k\}$ and $\{a_k\}$
defined in (\ref{def abxy conv}). If $d_{0}$ is the distance of $(z_0,w_0)$ to the set $\exset{\partial h_1}{\partial h_2}$, then for all $k=1,2,\dots,$ we have
\begin{equation}
\label{eq:inclusions fg}
0\in\partial g(v_k) + C^\ast x_k,\qquad\qquad 0\in\partial f(u_k)+M^\ast y_k,
\end{equation}
and there exists and index $1\leq i\leq k$ such that
\begin{equation}
\label{estima it conv}
\norm{Mu_i+Cv_i-d}\leq \frac{2d_0}{(1-\overline{\rho})\tau\sqrt{k}},\qquad\qquad\quad \norm{x_i-y_i}\leq\frac{2d_0}{(1-\overline{\rho})\tau\sqrt{k}};
\end{equation}
where $\tau=\min\left\lbrace\lambda,\dfrac{1}{\lambda}\right\rbrace$. 
\end{theorem}

\begin{proof}
Inclusions \eqref{eq:inclusions fg} were established in \eqref{eq:stopping step 2-2}. Therefore, what is left is to show the bounds in \eqref{estima it conv}.
Since for all integer $k\geq1$ the point $(z_k,w_k)$ is a relaxed projection of $(z_{k-1},w_{k-1})$ onto the set $H_{\phi_k}$ and $S_e(\partial h_1,\partial h_2)\subseteq H_{\phi_k}$, we take an arbitrary $(z^\ast,w^\ast)\in S_e(\partial h_1,\partial h_2)$ and use well-known properties of the orthogonal projection to obtain
\begin{equation*}
\begin{split}
\norm{(z_k,w_k)-(z^\ast,w^\ast)}^2 \leq  &\norm{(z_{k-1},w_{k-1})-(z^\ast,w^\ast)}^2 + \left(1-\dfrac{2}{\rho_k}\right)\norm{(z_k,w_k)-(z_{k-1},w_{k-1})}^2\\
 = & \norm{(z_{k-1},w_{k-1})-(z^\ast,w^\ast)}^2\\
& \qquad\qquad - \rho_k(2-\rho_k)\gamma_k^2\norm{(Mu_k+Cv_k-d,\lambda(w_{k-1}-Mu_k))}^2,
\end{split}
\end{equation*}
for $k=1,2,\dots$. Thus, applying the inequality above recursively, we have
\begin{equation}
\label{eq:est z_k}
\begin{split}
\norm{(z_k,w_k)-(z^\ast,w^\ast)}^2 \leq & \norm{(z_0,w_0)-(z^\ast,w^\ast)}^2 \\
&\quad - \sum_{j=1}^k\rho_j(2-\rho_j)\gamma_j^2\norm{(Mu_j+Cv_j-d,\lambda(w_{j-1}-Mu_j))}^2.
\end{split}
\end{equation}
We rearrange terms in the equation above and notice that $\lambda(w_{j-1}-Mu_j)=x_j-y_j$, which yields
\begin{align}
\nonumber
\sum_{j=1}^k\rho_j(2-\rho_j)\gamma_j^2\norm{(Mu_j+Cv_j-d,x_j-y_j)}^2 & \leq \norm{(z_0,w_0)-(z^\ast,w^\ast)}^2 - \norm{(z_k,w_k)-(z^\ast,w^\ast)}^2\\
\label{eq:sum dist}
& \leq \norm{(z_0,w_0)-(z^\ast,w^\ast)}^2.
\end{align}
Taking $(z^\ast,w^\ast)$ to be the orthogonal projection of $(z_0,w_0)$ onto $S_e(\partial h_1,\partial h_2)$ in inequality \eqref{eq:sum dist}, we obtain
\begin{equation}
\label{eq:sum est dist}
\sum_{j=1}^k\rho_j(2-\rho_j)\gamma_j^2\norm{(Mu_j+Cv_j-d,x_j-y_j)}^2  \leq d_0^2.
\end{equation}
Now, for $i$ such that 
\begin{equation*}
i\in\arg\min_{j=1,\dots,k}\left(\norm{(Mu_j+Cv_j-d,x_j-y_j)}^2\right),
\end{equation*}
we use inequality \eqref{eq:sum est dist} and the fact that $\rho_j\in[1-\overline{\rho},1+\overline{\rho}]$ to conclude that
\begin{equation}
\label{eq:teo complexity point1}
\norm{Mu_i+Cv_i-d}^2+\norm{x_i-y_i}^2 \leq \frac{d_0^2}{(1-\overline{\rho})^2\sum\limits_{j=1}^k\gamma_j^2}.
\end{equation}

Next, we notice that Proposition \ref{Prop:conv}, together with the equality in \eqref{eq:gamma_k}, implies
\begin{align}
\label{eq:rew-gam}
\gamma_j = \frac{\phi_j(z_{j-1},w_{j-1})}{\norm{\nabla\phi_j}^2},\qquad\qquad\text{for }\,j=1,\dots,k,
\end{align}
where $\phi_j$ is the affine function given in \eqref{eq:phi_k} associated with $x_j$, $y_j$, $b_j$ and $a_j$ defined in \eqref{def abxy conv}. Moreover, combining equations \eqref{eq:phi_k}, \eqref{def abxy conv}, \eqref{gradient phi conv} and \eqref{eq:z-y} we have
\begin{equation*}
\begin{split}
\phi_j(z_{j-1},w_{j-1}) = &\, \lambda\norm{Cv_j-d+w_{j-1}}^2 + \lambda\inpr{d-Cv_j-Mu_j}{w_{j-1}-Mu_j}\\
= &\, \dfrac\lambda2\norm{Cv_j-d+w_{j-1}}^2 + \frac{\lambda}{2}\left(\norm{d-Cv_j-Mu_j}^2+\norm{w_{j-1}-Mu_{j}}^2\right).
\end{split}
\end{equation*}
Hence, we substitute the relation above into \eqref{eq:rew-gam} to obtain
\begin{equation}
\label{eq:rewritten gamma_j}
\begin{split}
\gamma_j & =  \frac{\lambda\norm{Cv_j-d+w_{j-1}}^2}{2\norm{\nabla\phi_j}^2} + \frac{\lambda\norm{d-Cv_j-Mu_j}^2+\lambda\norm{w_{j-1}-Mu_{j}}^2}{2\norm{\nabla\phi_j}^2}\\
& \geq \frac{\lambda\norm{d-Cv_j-Mu_j}^2+\lambda\norm{w_{j-1}-Mu_{j}}^2}{2\norm{\nabla\phi_j}^2}.
\end{split}
\end{equation}
Now, we use the following estimate 
\begin{align*}
\lambda\norm{d-Cv_j-Mu_j}^2+\lambda\norm{w_{j-1}-Mu_{j}}^2 & =\lambda\norm{d-Cv_j-Mu_j}^2+ \frac{1}{\lambda}\lambda^2\norm{w_{j-1}-Mu_{j}}^2\\
& \geq \tau (\norm{d-Cv_j-Mu_j}^2 + \lambda^2\norm{w_{j-1}-Mu_{j}}^2)\\
& = \tau \norm{\nabla\phi_j}^2,
\end{align*}
and the inequality in \eqref{eq:rewritten gamma_j} to deduce that
\begin{equation}
\label{eq:bound gamma_j}
\gamma_j \geq \frac{\tau}{2},\qquad\qquad\text{for }\,j=1,\dots,k.
\end{equation}
This last inequality, together with \eqref{eq:teo complexity point1}, implies
\begin{equation*}
\norm{Mu_i+Cv_i-d}^2+\norm{x_i-y_i}^2 \leq \frac{4d_0^2}{(1-\overline{\rho})^2\tau^2k},
\end{equation*}
from which the theorem follows.
\end{proof}

We now develop alternative complexity bounds for the PMM, which we call \emph{ergodic} complexity bounds. We define a sequence of ergodic iterates as weighted averages of the iterates and derive a convergence rate for the PMM, which as before, is obtained from estimates of the residuals for the KKT conditions associated with these ergodic sequences. 

The idea of considering averages of the iterates in the analysis of the convergence rate for methods for solving problem \eqref{eq:convex opt prob} has been already used in other works. For instance, in \cite{MonSva2013,HeYua2012} it was shown a worst-case $\mathcal{O}(1/k)$ convergence rate for the ADMM in the ergodic sense. 

The sequences of ergodic means $\{\overline{u}_k\}$, $\{\overline{v}_k\}$, $\{\overline{x}_k\}$ and $\{\overline{y}_k\}$ associated with $\{u_k\}$, $\{v_k\}$, $\{x_k\}$ and $\{y_k\}$, respectively, are defined as
\begin{equation}
\label{defi uv erg}
\begin{split}
&\overline{u}_k=\frac{1}{\Gamma_k}\sum_{j=1}^k\rho_j\gamma_ju_j,\qquad\qquad\overline{v}_k=\frac{1}{\Gamma_k}\sum_{j=1}^k\rho_j\gamma_jv_j,\\
&\overline{x}_k =\frac{1}{\Gamma_k} \sum_{j=1}^k\rho_j\gamma_jx_j,\qquad\qquad\overline{y}_k =\frac{1}{\Gamma_k} \sum_{j=1}^k\rho_j\gamma_jy_j,
\end{split}
\qquad\quad\text{where }\,\Gamma_k=\sum_{j=1}^k\rho_j\gamma_j.
\end{equation}

\begin{lemma}\label{lem:inclus erg}
For all integer $k\geq1$ define 
\begin{equation}
\label{eq:def epsilon ergodic}
\begin{split}
& \overline{\epsilon}_k^u = \frac{1}{\Gamma_k}\sum_{j=1}^k\rho_j\gamma_j\inpr{u_j-\overline{u}_k}{-M^\ast y_j},\qquad\qquad \overline{\epsilon}_k^v = \frac{1}{\Gamma_k}\sum_{j=1}^k\rho_j\gamma_j\inpr{v_j-\overline{v}_k}{-C^\ast x_j}.
\end{split}
\end{equation}
Then, \, $\overline{\epsilon}_k^v\geq0$,\, $\overline{\epsilon}_k^u\geq0$\, and 
\begin{equation}
\label{inclus fg erg}
0\in\partial_{\overline{\epsilon}_{k}^v}g(\overline{v}_k) + C^\ast\overline{x}_k,\qquad\qquad 0\in\partial_{\overline{\epsilon}_{k}^u}f(\overline{u}_k) + M^\ast\overline{y}_k.
\end{equation}
\end{lemma}

\begin{proof}
From inclusions in \eqref{eq:inclusions fg} we have
\begin{equation*}
-C^\ast x_k\in\partial g(v_k)\qquad\text{ and }\qquad -M^\ast y_k\in\partial f(u_k).
\end{equation*}
Thus, the assertion that $\overline{\epsilon}_k^v\geq0$ and the first inclusion in \eqref{inclus fg erg} are a direct consequence of the first inclusion in the equation above, the definitions of $\overline{x}_k$, $\overline{v}_k$ and $\overline{\epsilon}_k^v$, the fact that $C^\ast$ is a linear operator and Proposition \ref{Prop:subdifferential}(d). 

Similarly, the second inclusion in \eqref{inclus fg erg} and the fact that $\overline{\epsilon}_k^u\geq0$ follow from the definitions of $\overline{y}_k$, $\overline{u}_k$ and $\overline{\epsilon}_k^u$, linearity of the $M^\ast$ operator, the second inclusion in relation above and Proposition \ref{Prop:subdifferential}(d).
\end{proof}

According to Lemma \ref{lem:inclus erg}, if  $\norm{\overline{r}^p_k}=\norm{\overline{r}^d_k}=0$ and $\overline{\epsilon}_k^u=\overline{\epsilon}_k^v=0$, where $\overline{r}_k^p = M\overline{u}_k + C\overline{v}_k - d$ and $\overline{r}_k^d = \overline{x}_k - \overline{y}_k$; then it follows that $(\overline{u}_k,\overline{v}_k,\overline{x}_k)$ satisfies the KKT conditions and, consequently, it is a  saddle point of the Lagrangian function. 
Thus, we have computable residuals for the sequence of ergodic means, i.e. the residual vector $(\overline{r}_k^p,\overline{r}^d_k,\overline{\epsilon}_k^u,\overline{\epsilon}_k^v)$, and we can attempt to construct bounds on its size.

For this purpose, we first prove the following technical result.
It establishes an estimate for the quantity $\overline{\epsilon}_k^u+\overline{\epsilon}_k^v$.

\begin{lemma}
Let $\{u_{k}\}$, $\{v_k\}$, $\{z_k\}$, $\{w_k\}$, $\{\gamma_k\}$ and $\{\rho_k\}$ be the sequences generated by the PMM and $\{x_k\}$, $\{y_k\}$ be defined in \eqref{def abxy conv}. Define also the sequences of ergodic iterates $\{\overline{u}_k\}$, $\{\overline{v}_k\}$, $\{\overline{x}_{k}\}$, $\{\overline{y}_{k}\}$, $\{\overline{\epsilon}_{k}^u\}$ and $\{\overline{\epsilon}_{k}^v\}$ as in \eqref{defi uv erg} and \eqref{eq:def epsilon ergodic}. Then, for every integer $k\geq1$, we have
\begin{equation}
\label{compl_erg_ep}
\overline{\epsilon}_{k}^u + \overline{\epsilon}_{k}^v \leq \frac{1}{\Gamma_{k}}\left[\frac{1}{\Gamma_{k}}\sum_{j=1}^{k}\rho_{j}\gamma_{j}\left(\lambda^2\norm{Mu_j+Cv_j-d}^{2}+\norm{d-Cv_j-w_{j-1}}^2\right)+4d_0^2\right].
\end{equation}
\end{lemma}

\begin{proof}
We first show that 
\begin{equation}
\label{eq:rewriten epsilon}
\overline{\epsilon}_{k}^u + \overline{\epsilon}_{k}^v = -\frac{1}{\Gamma_k}\sum_{j=1}^k\rho_j\gamma_j\phi_j(\overline{y}_k,d-C\overline{v}_k).
\end{equation}
By the definitions of $\phi_j$, $b_j$ and $a_j$, we have
\begin{align}
\nonumber
\phi_j(\overline{y}_k,d-C\overline{v}_k) & = \inpr{\overline{y}_k-x_j}{C\overline{v}_k-Cv_j}+\inpr{\overline{y}_k-y_j}{d-C\overline{v}_k-Mu_j}\\
\label{eq:phi_j 1}
& = -\inpr{\overline{y}_k}{Cv_j}-\inpr{x_j}{C\overline{v}_k-Cv_j}+\inpr{\overline{y}_k-y_j}{d} +\inpr{y_j}{C\overline{v}_k}-\inpr{\overline{y}_k-y_j}{Mu_j}.
\end{align}
We use the definitions of $\overline{y}_k$, $\overline{v}_k$, $\Gamma_k$ and the fact that $C$ is a linear map, to obtain
\begin{equation*}
\frac{1}{\Gamma_k}\sum_{j=1}^k\rho_j\gamma_j\left(\inpr{\overline{y}_k-y_j}{d}-\inpr{\overline{y}_k}{Cv_j}+\inpr{y_j}{C\overline{v}_k}\right) = \inpr{\overline{y}_k-\overline{y}_k}{d} -\inpr{\overline{y}_k}{C\overline{v}_k} + \inpr{\overline{y}_k}{C\overline{v}_k} = 0.
\end{equation*}
Now, multiplying \eqref{eq:phi_j 1} by $\rho_j\gamma_j/\Gamma_k$, adding from $j=1$ to $k$ and combining with the relation above, we conclude that 
\begin{equation}
\label{eq:sum phi_k}
\frac{1}{\Gamma_k}\sum_{j=1}^k\rho_j\gamma_j\phi_j(\overline{y}_k,d-C\overline{v}_k) = - \frac{1}{\Gamma_k}\sum_{j=1}^k\rho_j\gamma_j\left(\inpr{x_j}{C\overline{v}_k-Cv_j}+\inpr{\overline{y}_k-y_j}{Mu_j}\right).
\end{equation}
Next, we observe that 
\begin{equation*}
\begin{split}
\overline{\epsilon}_{k}^u + \overline{\epsilon}_{k}^v & = \frac{1}{\Gamma_k}\sum_{j=1}^k\rho_j\gamma_j\left(\inpr{M\overline{u}_k-Mu_j}{y_j}+\inpr{C\overline{v}_k-Cv_j}{x_j}\right)\\
& = \frac{1}{\Gamma_k}\sum_{j=1}^k\rho_j\gamma_j\left(\inpr{Mu_j}{\overline{y}_k-y_j} + \inpr{C\overline{v}_k-Cv_j}{x_j}\right),
\end{split}
\end{equation*}
where the last equality above is a consequence of the definitions of $\overline{y}_k$ and $M\overline{u}_k$. We deduce formula \eqref{eq:rewriten epsilon} combining the equation above with \eqref{eq:sum phi_k}.

For an arbitrary $(z,w)\in\R^n\times\R^n$ and all integer $j\geq1$ we have
\begin{align}
\nonumber
\dfrac{1}{2}\norm{(z,w)-(z_{j-1},w_{j-1})}^2 = & \,\dfrac{1}{2}\norm{(z,w)-(z_j,w_j)}^2 + \inpr{(z,w)-(z_j,w_j)}{(z_j,w_j)-(z_{j-1},w_{j-1})}\\
\nonumber
& + \dfrac{1}{2}\norm{(z_j,w_j)-(z_{j-1},w_{j-1})}^2\\
\label{eq:eq 1}
= & \, \dfrac{1}{2}\norm{(z,w)-(z_j,w_j)}^2-\rho_j\gamma_j\inpr{(z,w)-(z_j,w_j)}{\nabla\phi_j} + \dfrac{1}{2}\rho_j^2\gamma_j^2\norm{\nabla\phi_j}^2,
\end{align}
where the second equality follows from the identity $(z_j,w_j)=(z_{j-1},w_{j-1})-\rho_j\gamma_j\nabla\phi_j$, which is a consequence of step 3 in the PMM, \eqref{gradient phi conv} and \eqref{eq:gradient ph_k}. Now, we notice that
\begin{equation*}
\begin{split}
\inpr{(z,w)-(z_j,w_j)}{\nabla\phi_j} & = \inpr{(z,w)-(y_j,d-Cv_j)}{\nabla\phi_j}+\inpr{(y_j,d-Cv_j)-(z_j,w_j)}{\nabla\phi_j}\\
& = \phi_j(z,w) - \phi_j(z_j,w_j)\\
& = \phi_j(z,w) - \phi_j((z_{j-1},w_{j-1})-\rho_j\phi_j\nabla\phi_j)\\
& = \phi_j(z,w) - \phi_j(z_{j-1},w_{j-1}) + \rho_j\phi_j\norm{\nabla\phi_j}^2,
\end{split}
\end{equation*}
where the second and forth equalities are due to \eqref{eq:phi_k}, \eqref{eq:gradient ph_k} and \eqref{def abxy conv}. Substituting the equation above into \eqref{eq:eq 1} yields
\begin{equation*}
\begin{split}
\dfrac{1}{2}\norm{(z,w)-(z_{j-1},w_{j-1})}^2  = & \,\dfrac{1}{2}\norm{(z,w)-(z_j,w_j)}^2 - \rho_j\gamma_j\phi_j(z,w)  + \rho_j\gamma_j\phi_j(z_{j-1},w_{j-1}) \\
& - \rho_j^2\gamma_j^2\norm{\nabla\phi_j}^2 + \dfrac{1}{2}\rho_j^2\gamma_j^2\norm{\nabla\phi_j}^2\\
 = & \, \dfrac{1}{2}\norm{(z,w)-(z_j,w_j)}^2 - \rho_j\gamma_j\phi_j(z,w) + \rho_j\gamma_j^2\norm{\nabla\phi_j}^2 - \dfrac{1}{2}\rho_j^2\gamma_j^2\norm{\nabla\phi_j}^2,
\end{split}
\end{equation*}
where formula \eqref{eq:gamma_k} is used for obtaining the last equality. Rearranging terms in the equation above and adding from $j=1$ to $k$, we obtain
\begin{equation*}
-\sum_{j=1}^k\rho_j\gamma_j\phi_j(z,w) = \dfrac{1}{2}\norm{(z,w)-(z_{0},w_{0})}^2 - \dfrac{1}{2}\norm{(z,w)-(z_k,w_k)}^2 - \sum_{j=1}^k\dfrac{1}{2}\rho_j(2-\rho_j)\gamma_j^2\norm{\nabla\phi_j}^2.
\end{equation*}
Consequently, we have
\begin{equation*}
-\sum_{j=1}^k\rho_j\gamma_j\phi_j(z,w) \leq \dfrac{1}{2}\norm{(z,w)-(z_{0},w_{0})}^2,\qquad\qquad\forall(z,w)\in\R^n\times\R^n.
\end{equation*}
Now we use inequality above with $(z,w)=(\overline{y}_k,d-C\overline{v}_k)$, and combine with \eqref{eq:rewriten epsilon}, to obtain
\begin{equation}
\label{eq:12}
\begin{split}
\overline{\epsilon}_k^u+\overline{\epsilon}_k^v \leq & \dfrac{1}{2\Gamma_k}\norm{(\overline{y}_k,d-C\overline{v}_k)-(z_{0},w_{0})}^2\\
\leq & \dfrac{1}{2\Gamma_k}\sum_{j=1}^k\dfrac{\rho_j\gamma_j}{\Gamma_k}\norm{(y_{j},d-Cv_j)-(z_0,w_0)}^2\\
\leq & \dfrac{1}{\Gamma_k}\sum_{j=1}^k\dfrac{\rho_j\gamma_j}{\Gamma_k}\left(\norm{(y_{j},d-Cv_j)-(z_{j-1},w_{j-1})}^2 + \norm{(z_{j-1},w_{j-1})-(z_0,w_0)}^2\right),
\end{split}
\end{equation}
where the second inequality above is due to the definitions of $\overline{y}_k$, $\overline{v}_k$, the fact that $C$ is a linear operator and the convexity of $\norm{\cdot}^2$. Further, the third inequality in equation above is obtained using the triangle inequality for norms.
Next, we notice that inequality \eqref{eq:est z_k} implies
\begin{equation*}
\norm{(z_j,w_j)-(z^\ast,w^\ast)} \leq \norm{(z^\ast,w^\ast)-(z_0,w_0)},
\end{equation*}
for all integers $j\geq0$ and all $(z^\ast,w^\ast)\in S_e(\partial h_1,\partial h_2)$. Taking $(z^\ast,w^\ast)$ to be the orthogonal projection of $(z_0,w_0)$ onto $S_e(\partial h_1,\partial h_2)$ in the relation above and using the triangle inequality, we deduce that
\begin{equation}
\label{eq:est z_kw_k}
\norm{(z_j,w_j)-(z_0,w_0)} \leq \norm{(z_j,w_j)-(z^\ast,w^\ast)} + \norm{(z^\ast,w^\ast)-(z_0,w_0)} \leq 2d_0.
\end{equation}
Combining \eqref{eq:est z_kw_k} with \eqref{eq:12} we have
\begin{equation*}
\overline{\epsilon}_{k}^u + \overline{\epsilon}_{k}^v \leq \frac{1}{\Gamma_k}\left[\frac{1}{\Gamma_k}\sum_{j=1}^k\rho_j\gamma_j\norm{(y_j,d-Cv_j)-(z_{j-1},w_{j-1})}^2 + 4d_0^2\right].
\end{equation*}
To end the proof we substitute the identity $y_j-z_{j-1}=\lambda(Mu_j+Cv_j-d)$, which follows from the definition of $y_j$ in \eqref{def abxy conv}, into the above inequality.
\end{proof}

The following theorem provides estimates for the quality of the measure of the ergodic means $\overline{u}_k$, $\overline{v}_k$, $\overline{x}_k$ and $\overline{y}_k$. More specifically, we show that the residuals associated with the ergodic sequences are $\mathcal{O}(1/k)$.
\begin{theorem}
Assume the hypotheses of Theorem \ref{Teo:complexity-point}. Consider also the sequences $\{\overline{u}_k\}$, $\{\overline{v}_k\}$, $\{\overline{x}_k\}$ and $\{\overline{y}_k\}$ given in \eqref{defi uv erg}, and $\{\overline{\epsilon}_k^u\}$, $\{\overline{\epsilon}_k^v\}$ defined in \eqref{eq:def epsilon ergodic}. Then, for all integer $k\geq1$, we have
\begin{equation}
\label{eq:inclusion fg ergodic}
0\in\partial_{\overline{\epsilon}_{k}^v}g(\overline{v}_k) + C^\ast\overline{x}_k,\qquad\qquad 0\in\partial_{\overline{\epsilon}_{k}^u}f(\overline{u}_k) + M^\ast\overline{y}_k,
\end{equation}
and
\begin{align}
\label{estima erg conv}
&\norm{M\overline{u}_k+C\overline{v}_k-d}\leq \frac{4d_0}{k(1-\overline{\rho})\tau},\qquad\qquad\quad \norm{\overline{x}_k-\overline{y}_k}\leq\frac{4d_0}{k(1-\overline{\rho})\tau},\\
\label{eq:est eps erg conv}
& \overline{\epsilon}^u_{k}+\overline{\epsilon}^v_{k}\leq\frac{8d_0^2\vartheta}{k(1-\overline{\rho})\tau};
\end{align}
where $\vartheta=\dfrac{1}{\tau^2(1-\overline{\rho})^2}+1$.
\end{theorem}

\begin{proof}
The inclusions in \eqref{eq:inclusion fg ergodic} were proven in Lemma \ref{lem:inclus erg}. 
To prove the estimates in \eqref{estima erg conv} we first observe that, since
\begin{equation*}
x_k - y_k =\lambda(w_{k-1}-Mu_k)\qquad\qquad\text{ for }\,k=1,2,\dots,
\end{equation*}
by the update rule in step 3 of the PMM we have
\begin{align}
\nonumber
(z_k,w_k) & = (z_{k-1},w_{k-1}) - \rho_k\gamma_k(d-Cv_k-Mu_k,x_k-y_k)\\
\nonumber
& = (z_0,w_0) - \sum_{j=1}^k\rho_j\gamma_j(d-Cv_j-Mu_j,x_j-y_j)\\
\label{eq:revritten erg residual}
& = (z_0,w_0) - \Gamma_k(d-C\overline{v}_k-M\overline{u}_k,\overline{x}_k-\overline{y}_k),
\end{align}
where the last equality above follows from the definitions of $\Gamma_k$, $\overline{v}_k$, $\overline{u}_k$, $\overline{x}_k$ and $\overline{y}_k$ in \eqref{defi uv erg}, and the fact that $M$ and $C$ are linear operators. Therefore, from \eqref{eq:revritten erg residual} we deduce that
\begin{equation*}
\norm{(d-C\overline{v}_k-M\overline{u}_k,\overline{x}_k-\overline{y}_k)} = \frac{1}{\Gamma_k}\norm{(z_0,w_0) - (z_k,w_k)},
\end{equation*}
and combining the identity above with estimate \eqref{eq:est z_kw_k} we obtain
\begin{equation}
\label{eq:bound ergodic residual 1}
\norm{(d-C\overline{v}_k-M\overline{u}_k,\overline{x}_k-\overline{y}_k)} \leq \frac{2d_0}{\Gamma_k}.
\end{equation}
Next, we notice that equation \eqref{eq:bound gamma_j} and the fact that $\rho_j\in[1-\overline{\rho},1+\overline{\rho}]$ imply
\begin{equation}
\label{eq:bound Gamma_k}
\Gamma_k=\sum_{j=1}^k\rho_j\gamma_j\geq\sum_{j=1}^k(1-\overline{\rho})\frac{\tau}{2} = (1-\overline{\rho})\frac{\tau}{2}k.
\end{equation}
The inequality above, together with \eqref{eq:bound ergodic residual 1}, yields
\begin{equation*}
\norm{(d-C\overline{v}_k-M\overline{u}_k,\overline{x}_k-\overline{y}_k)} \leq \frac{4d_0}{(1-\overline{\rho})\tau k},
\end{equation*}
from which the bounds in \eqref{estima erg conv} follow directly.

Now, using the equality in \eqref{eq:rewritten gamma_j} we have
\begin{equation*}
\gamma_j\geq \frac{\lambda\norm{Cv_j-d+w_{j-1}}^2}{2\norm{\nabla\phi_j}^2}+\frac{\lambda\norm{d-Cv_j-Mu_j}^2}{2\norm{\nabla\phi_j}^2},
\end{equation*}
and as a consequence we obtain
\begin{equation*}
\norm{\nabla\phi_j}^2\gamma_j \geq \frac{\tau}{2}\left(\norm{Cv_j-d+w_{j-1}}^2+\lambda^2\norm{d-Cv_j-Mu_j}^2\right),\qquad\text{for } j=1,\dots,k.
\end{equation*}
Multiplying the inequality above by $\rho_j\gamma_j2/\tau$, adding from $j=1$ to $k$ and using \eqref{compl_erg_ep}, we have
\begin{equation*}
\begin{split}
\overline{\epsilon}_k^u+\overline{\epsilon}_k^v & \leq \frac{1}{\Gamma_k}\left[\frac{2}{\tau\Gamma_k}\sum_{j=1}^k\rho_j\gamma_j^2\norm{\nabla\phi_j}^2+4d_0^2\right]\\
& = \frac{1}{\Gamma_k}\left[\frac{2}{\tau\Gamma_k}\sum_{j=1}^k\frac{1}{2-\rho
_j}\rho_j(2-\rho_j)\gamma_j^2\norm{\nabla\phi_j}^2+4d_0^2\right].
\end{split}
\end{equation*}
Finally, relation above, together with \eqref{eq:sum est dist} and the fact that $\rho_j\in[1-\overline{\rho},1+\overline{\rho}]$, yields
\begin{equation*}
\begin{split}
\overline{\epsilon}_k^u+\overline{\epsilon}_k^v \leq \frac{1}{\Gamma_k}\left[\frac{2}{\tau\Gamma_k(1-\overline{\rho})}d_0^2+4d_0^2\right].
\end{split}
\end{equation*}
The bound in \eqref{eq:est eps erg conv} is achieved using this last inequality and \eqref{eq:bound Gamma_k}.
\end{proof}

\section{Applications}
\label{sec:application}
In this section we discuss the specialization of the PMM to two common test problems. First, we consider the total variation model for image denoising (TV denoising). Then, we consider a compressed sensing problem for Magnetic Resonance Imaging. We also exhibit some preliminary numerical experiments to illustrate the performance of the PMM when solving these problems.

\subsection{TV denoising}
\label{subs:tv}
Total variation (TV) or ROF model is a common image model developed by Rudin, Osher and Fatemi \cite{rudin_tv} for the problem of removing noise from an image.
If $b\in\R^{m\times n}$ is an observed noisy image, the TV problem for image denoising estimates the unknown original image $u\in\R^{m\times n}$ by solving the minimization problem
\begin{equation}
\label{eq:tv problem}
\min_{u\in\R^{m\times n}}\zeta TV(u) + \frac{1}{2}\norm{u-b}^2_F,
\end{equation}
where $TV$ is the total variation norm defined as
\begin{equation}
\label{eq:tv-norm}
TV(u) = \norm{\nabla_1u}_1 + \norm{\nabla_2u}_1.
\end{equation}
Here $\nabla_1:\R^{m\times n}\to\R^{m\times n}$ and $\nabla_2:\R^{m\times n}\to\R^{m\times n}$ are the discrete forward gradients in the first and second direction, respectively, given by 
\begin{equation*}
\begin{aligned}
(\nabla_1 u)_{ij} = u_{i+1,j}-u_{i,j},\qquad
(\nabla_2u)_{ij} = u_{i,j+1}-u_{i,j},
\end{aligned}
\qquad i=1,\ldots,m,\,\,j=1,\ldots,n,\,\, u\in\R^{m\times n};
\end{equation*}
and we assume standard reflexive boundary conditions 
\begin{equation*}
u_{m+1,j} - u_{m,j} = 0,\quad j=1,\dots,n\qquad\text{and}\qquad u_{i,n+1} - u_{i,n} = 0,\quad i=1,\dots,m.
\end{equation*}
The regularization parameter $\zeta>0$ controls the tradeoff between fidelity to measurements and the smoothness term given by the total variation. 

To solve the TV problem using the PMM we first have to sate it in the form of a linearly constrained minimization problem \eqref{eq:convex opt prob}. If we define $\Omega:=\R^{m\times n}\times\R^{m\times n}$, and the linear map $\nabla:\R^{m\times n}\to\Omega$ by 
$$\nabla u=(\nabla_1u,\nabla_2u);$$
then, taking $v=\nabla u\in\Omega$, we have that \eqref{eq:tv problem} is equivalent to the optimization problem
\begin{equation}
\label{tv linear constrain}
\min_{(u,v)\in\R^{m\times n}\times\Omega}\set{\zeta\norm{v}_1+\frac{1}{2}\norm{u-b}^{2}_F\,:\,\nabla u-v=0}.
\end{equation}
Now, we solve \eqref{tv linear constrain} by applying the PMM with $f(u)=\dfrac12\norm{u-b}_F^2$, $g(v)=\zeta\norm{v}_1$, $M=\nabla$, $C=-I$ and $d=0$.

Given ${z}_{k-1},\,{w}_{k-1}\in\Omega$, the PMM requires the solution of problems,
\begin{align}
\label{subproblem l1}
&v_{k}=\arg\,\min_{v\in\Omega}\zeta\norm{v}_1-\inpr{z_{k-1}+\lambda w_{k-1}}{v}+\dfrac\lambda2\norm{v}^{2}_F,
\end{align}
and
\begin{align}
\label{subproblem l2}
&u_k = \arg\min_{u\in\R^{m\times n}}\dfrac12\norm{u-b}^2_F + \inpr{z_{k-1}-\lambda v_k}{\nabla u} + \dfrac\lambda2\norm{\nabla u}^2_F.
\end{align}
The optimality condition of problem \eqref{subproblem l1}, yields
\begin{equation*}
 0\in\zeta\partial \norm{\cdot}_1(v_k) -  (z_{k-1}+\lambda w_{k-1}) + \lambda v_k;
\end{equation*}
hence, 
\begin{equation*}
 v_k=\left(I+\dfrac\zeta\lambda\partial \norm{\cdot}_1\right)^{-1}\left(\dfrac1\lambda z_{k-1}+w_{k-1}\right).
\end{equation*}
Therefore, the solution of problem \eqref{subproblem l1} can be computed explicitly as
\begin{equation*}
v_k = \textbf{shrink}\left( \dfrac1\lambda z_{k-1} +  w_{k-1},\dfrac\zeta\lambda\right),
\end{equation*}
where the \textbf{shrink} operator is defined in \eqref{eq:shrink}. Deriving the optimality condition for problem \eqref{subproblem l2} we have that
\begin{equation*}
 0=u_k - b + \nabla^\ast(z_{k-1}-\lambda v_k) + \lambda\nabla^\ast\nabla u_k,
\end{equation*}
from which it follows that $u_k$ has to be the solution of the system of linear equations
\begin{equation*}
(I+\lambda\nabla^\ast\nabla) u_k = b-\nabla^\ast(z_{k-1}-\lambda v_k).
\end{equation*}

Thus, the PMM applied to problem \eqref{tv linear constrain} produces the iteration:
\begin{align}
\label{formula v^k}
 &v_k = \textbf{shrink}\left(\dfrac1\lambda  z_{k-1} + w_{k-1},\dfrac\zeta\lambda\right),\\
 \label{u problem}
 &(I+\lambda\nabla^\ast\nabla) u_k = b-\nabla^\ast(z_{k-1}-\lambda v_k),\\
 &\gamma_k = \frac{\lambda\norm{w_{k-1}-v_k}^2+\lambda\inpr{v_k-\nabla u_k}{w_{k-1}-\nabla u_k}}{\norm{\nabla u_k-v_k}^2+\lambda^2\norm{\nabla u_k-w_{k-1}}^2},\\
 & z_{k} = z_{k-1} + \rho_k\gamma_k(\nabla u_k - v_k),\\
 & w_{k} = w_{k-1} - \rho_k\gamma_k\lambda(w_{k-1}-\nabla u_k).
\end{align}

We used three images to test the PMM in our experiments: the first was ``Lena'' image of size $512\times512$, the second was ``Baboon'' image of size $512\times512$, and the third was ``Man'' image of size $768\times768$, see Figure \ref{fig:test images}. All images were contaminated with Gaussian noise using the Matlab function ``\texttt{imnoise}'' with variance $\sigma=0.02$ and $\sigma=0.06$. The PMM was implemented in Matlab code and it was chosen $\lambda=1$ in all tests, since we have found that choosing this valued for $\lambda$ was effective for all the experiments. Images were denoised with $\zeta=20$ and $\zeta=50$.

\begin{figure}[h]
\centering
\subfigure{\includegraphics[height=48.0mm,width=48.0mm]{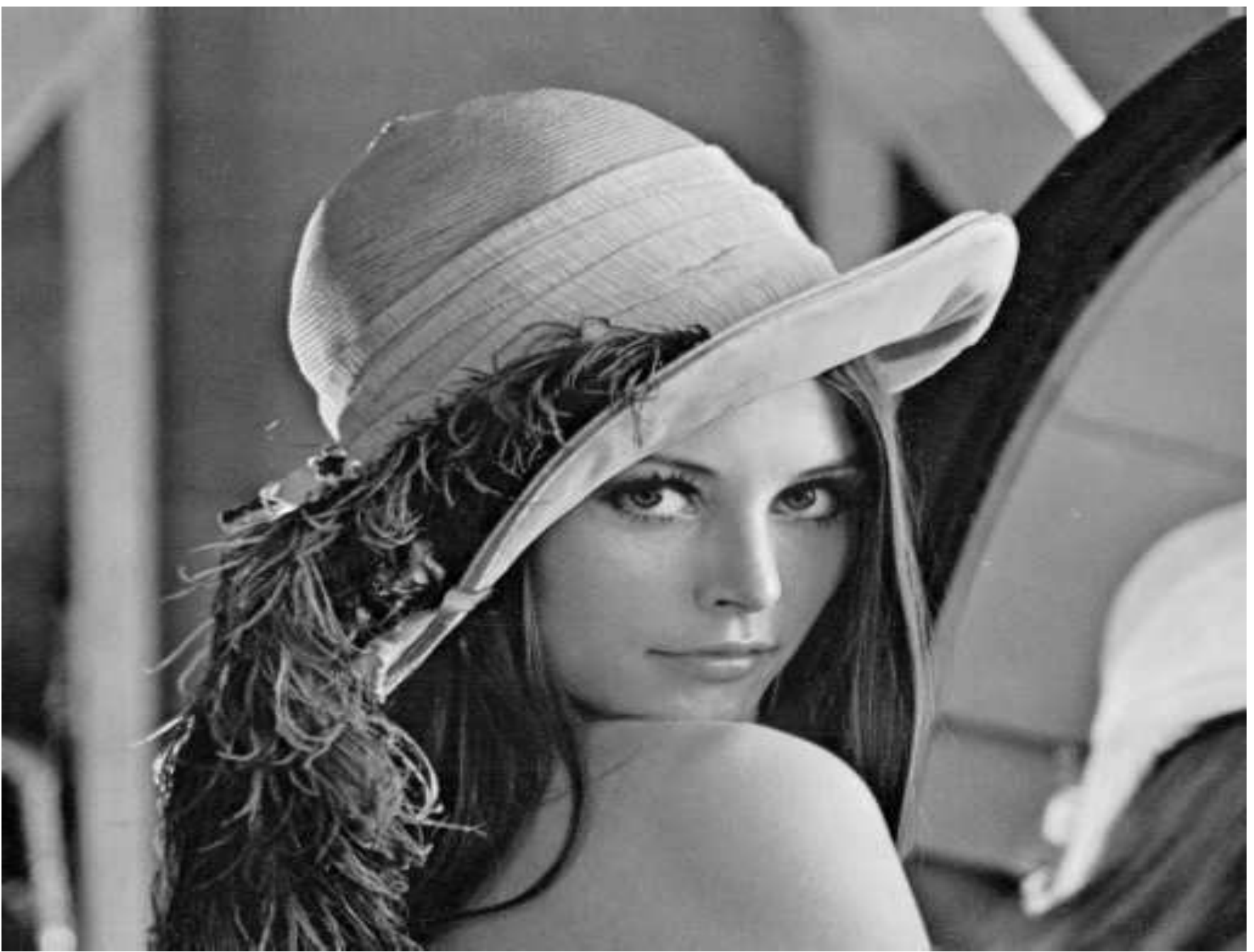}}
\subfigure{\includegraphics[height=48.0mm,width=48.0mm]{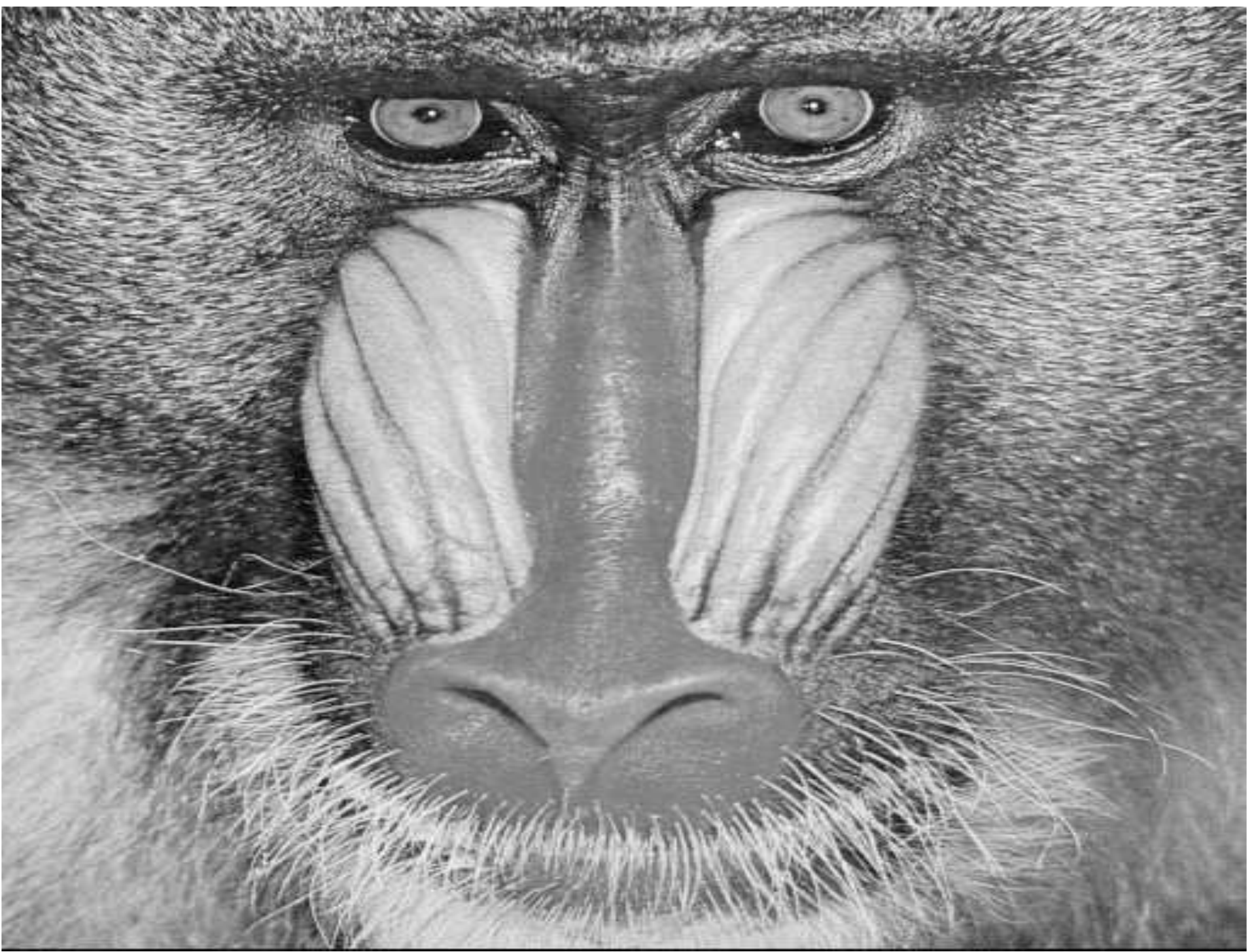}}
\subfigure{\includegraphics[height=48.0mm,width=48.0mm]{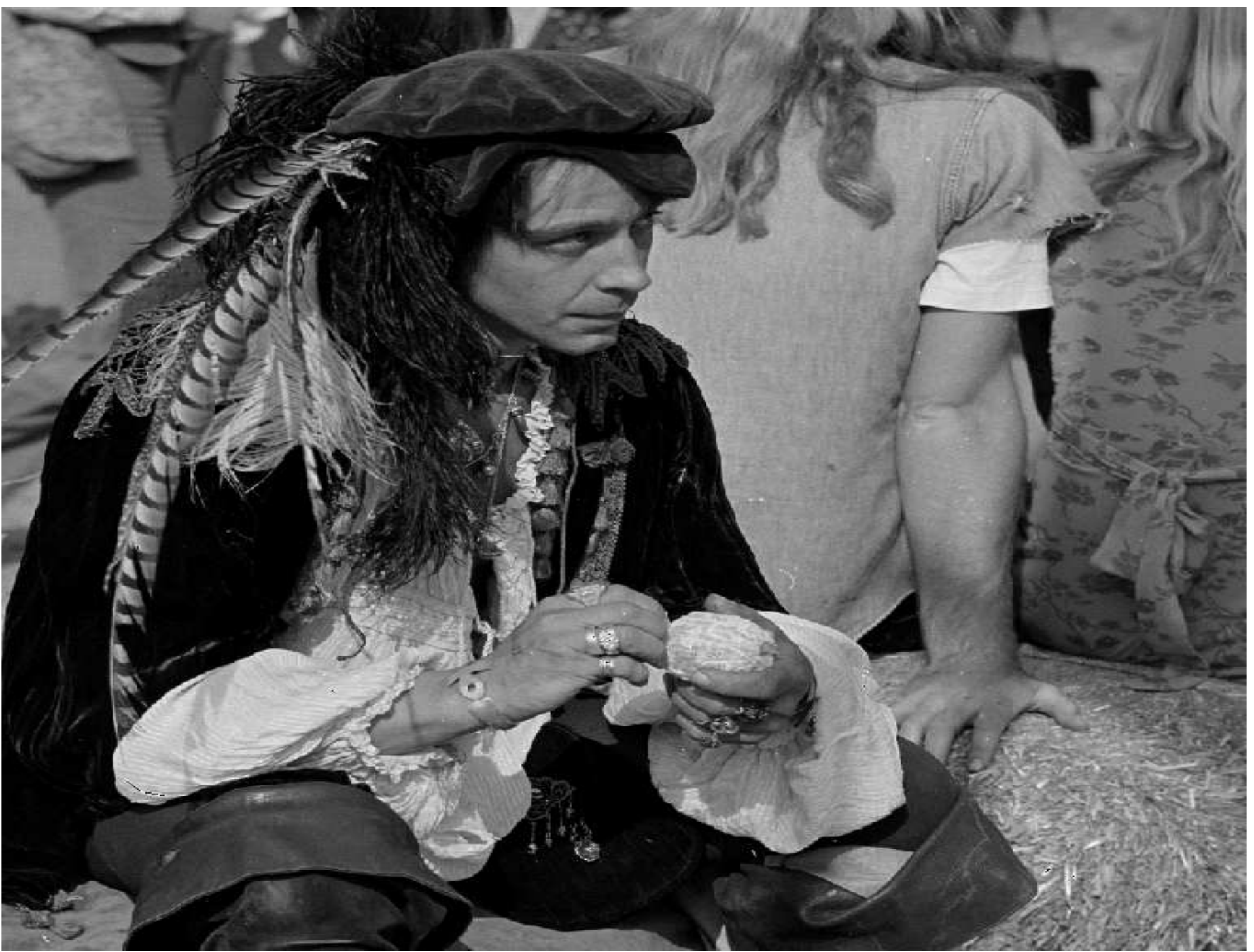}}
\caption{Test images: Lena (left), Baboon (center) and Man (right).}
\label{fig:test images}
\end{figure}

As a way to provide a reference, we also report the results obtained with ADMM, which is actually equivalent to the Split Bregman (SB) 
method \cite{goldstein_bregman,esser2009} for TV regularized problems. For a fair comparison, we implemented the 
generalized ADMM \cite{eck_ber_dr} with over and under relaxation factors, see also \cite{eckstein2012augmented}. In the 
numerical tests we used $\rho_k=1$ or $\rho_k=1.5$ for all integer $k\geq1$, in both methods. In Figure \ref{fig:denoising results} we present some denoising results. It shows the noise contaminated images and the reconstructed images with the PMM.
 As in \cite{goldstein_bregman} iterations were terminated when condition $\norm{u_k-u_{k-1}}/\norm{u_k}\leq10^{-3}$ was met; since this stopping criterion is satisfied faster than the stopping condition given by the KKT residuals, while yielding good denoised images.

Additionally, in Figure \ref{fig:residuals lena} we report the primal and dual residuals for the KKT optimality conditions for problem \eqref{tv linear constrain} for both methods, in some specific tests. The primal and dual residuals for the PMM were defined in section \ref{sec:complexity}. For the ADMM the primal residual is also defined as $\nabla u_k-v_k$, i.e. it is the residual for the equality constraint at iteration $k$. The dual residual for the ADMM is defined as the residual for the dual feasibility condition (see equation \eqref{Kuhn Tucker} and the comments below). Since the exact solution of the problems are known we also plotted in Figure \ref{fig:residuals lena} the  error $\norm{u_k-u^\ast}$ vs iteration, where $u^\ast$ is the exact solution. In these experiments both methods were stopped at iteration $50$. It can be observed in Figure \ref{fig:residuals lena} that the speed of the PMM and ADMM measured by the residuals curves are very similar; however the residuals for the PMM decay faster, and this difference is more evident in the dual residual curve.

In Table \ref{tab:comparison pmm admm} we present a more detailed comparison between the methods. It reports the iteration counts and total time, in seconds, required for the PMM and ADMM in the experiments. We observe that in the tests the PMM executed fewer iterations than ADMM, and the PMM was generally faster. We also observe that both methods accelerate when $\rho=1.5$. 

\begin{figure}[h]
\centering
\subfigure[$\sigma=0.02$]{\includegraphics[height=48mm,width=48mm]{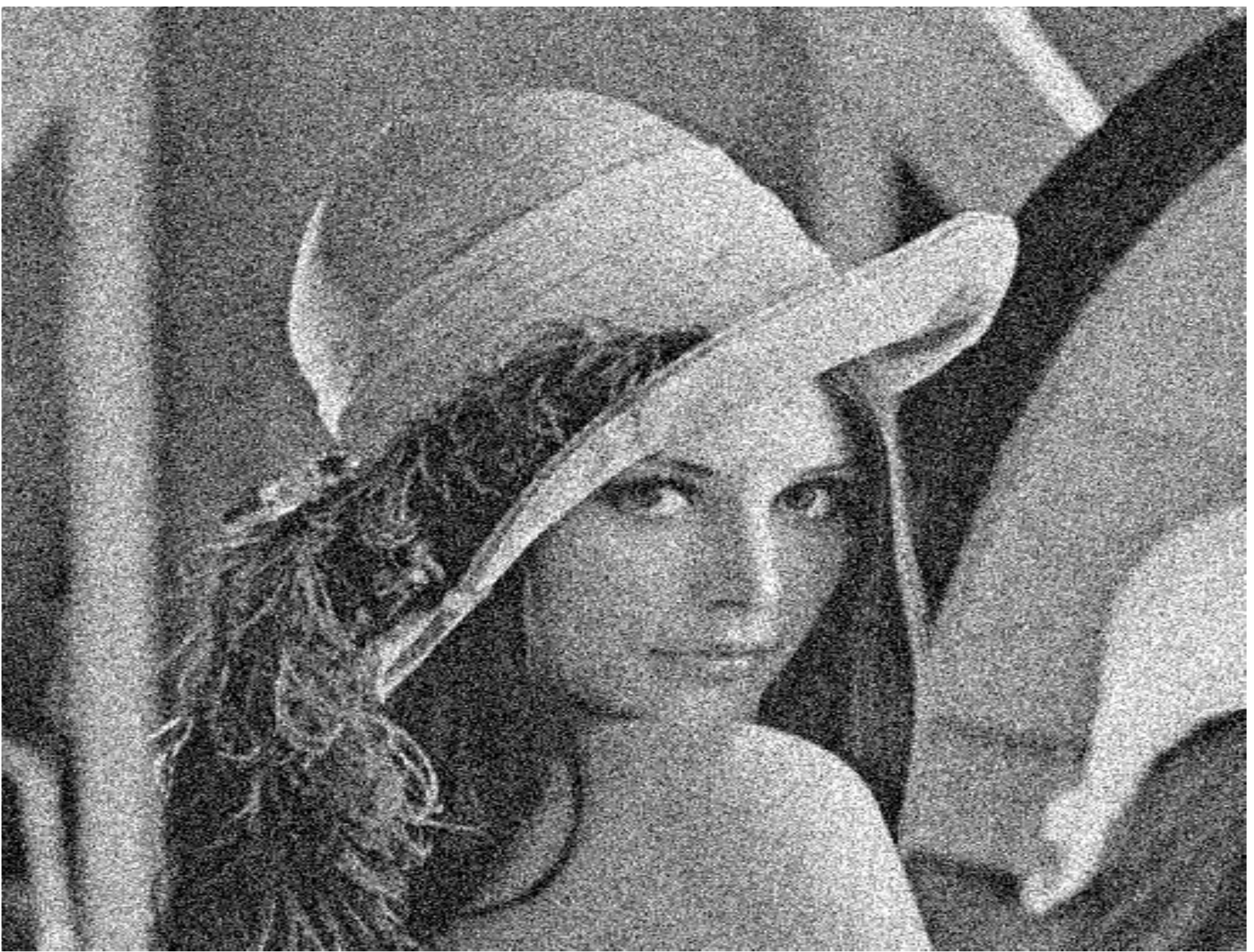}}
\subfigure[12/14, $\rho=1.5$]{\includegraphics[height=48mm,width=48mm]{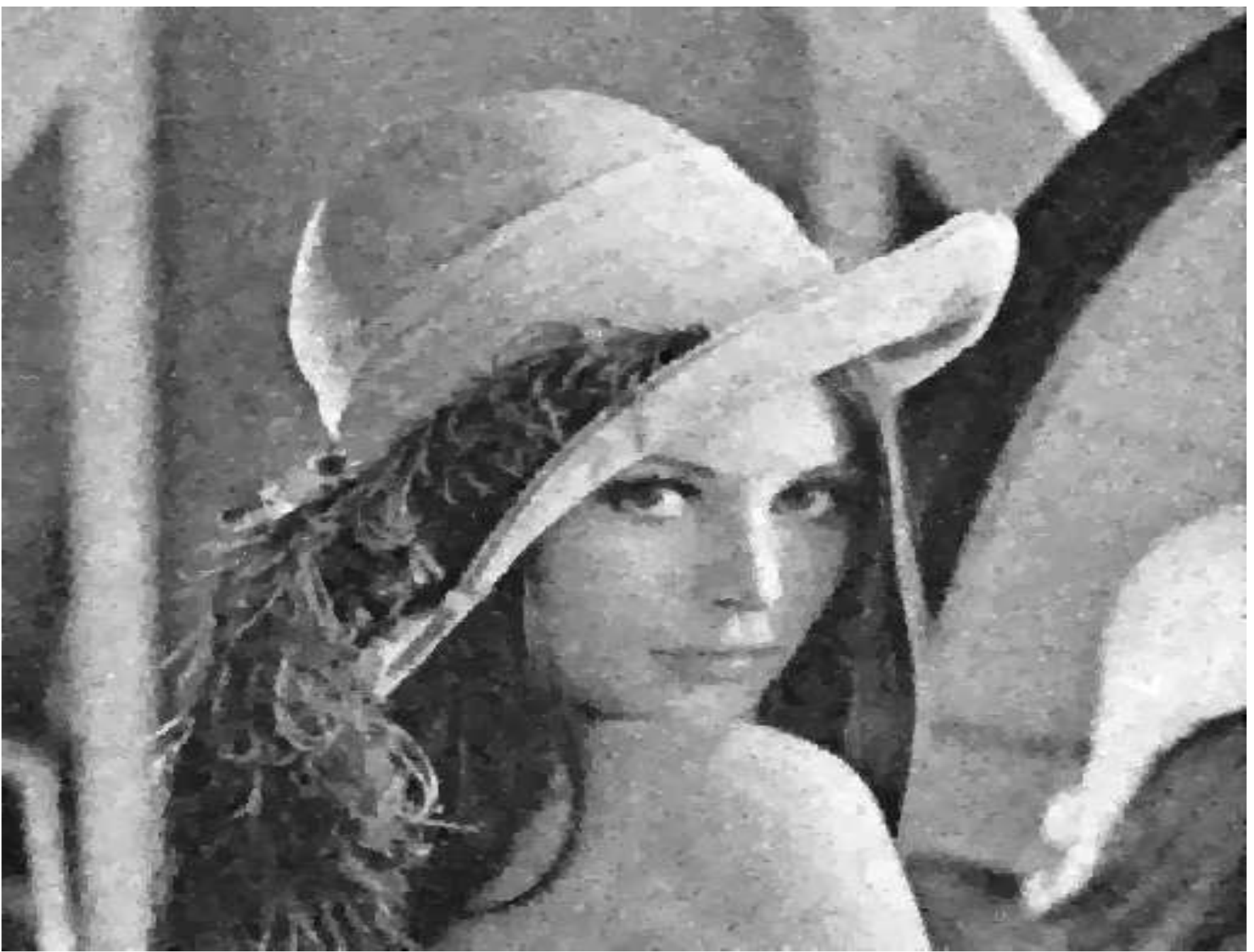}}
\subfigure[20/21, $\rho=1$]{\includegraphics[height=48mm,width=48mm]{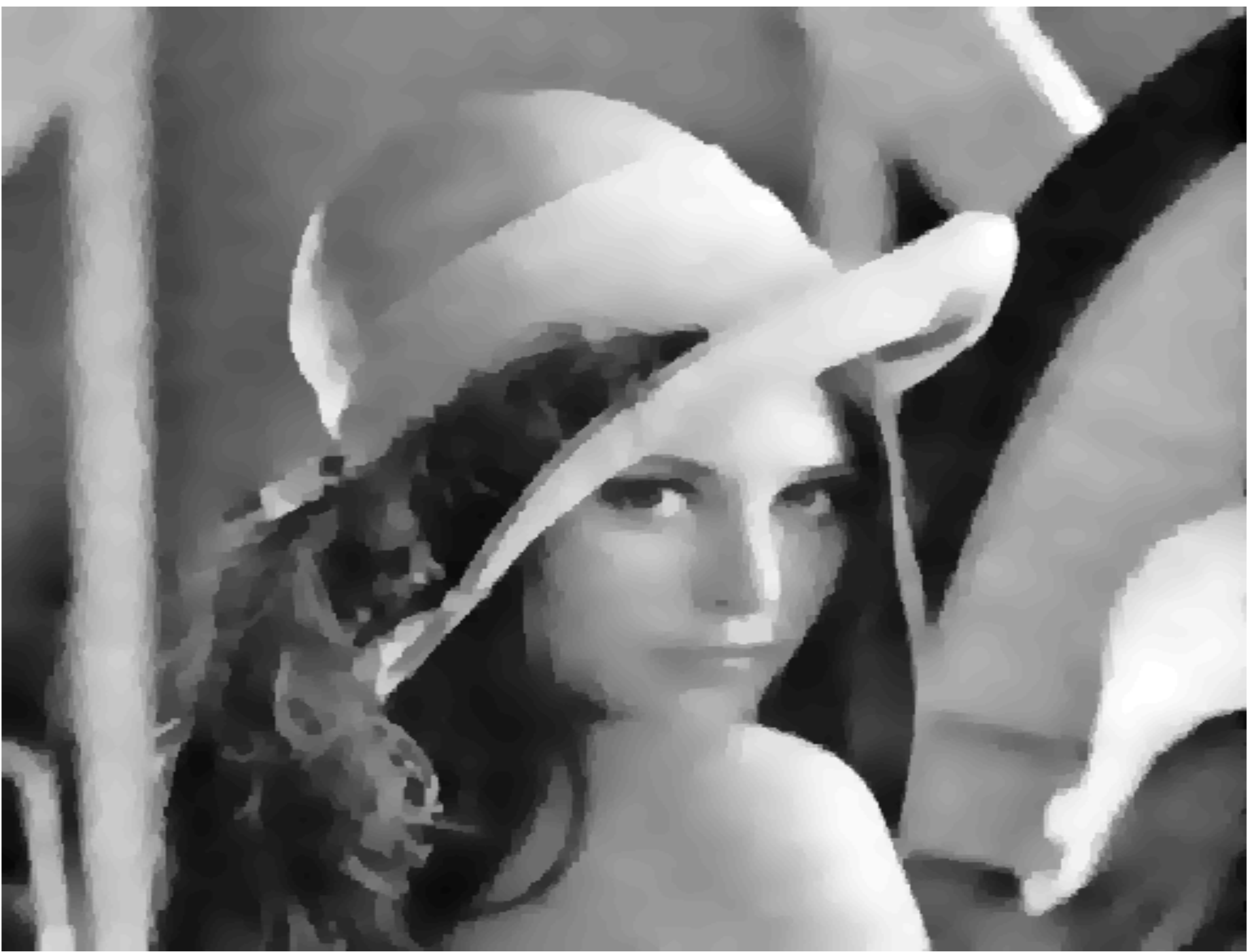}}
\subfigure[$\sigma=0.02$]{\includegraphics[height=48mm,width=48mm]{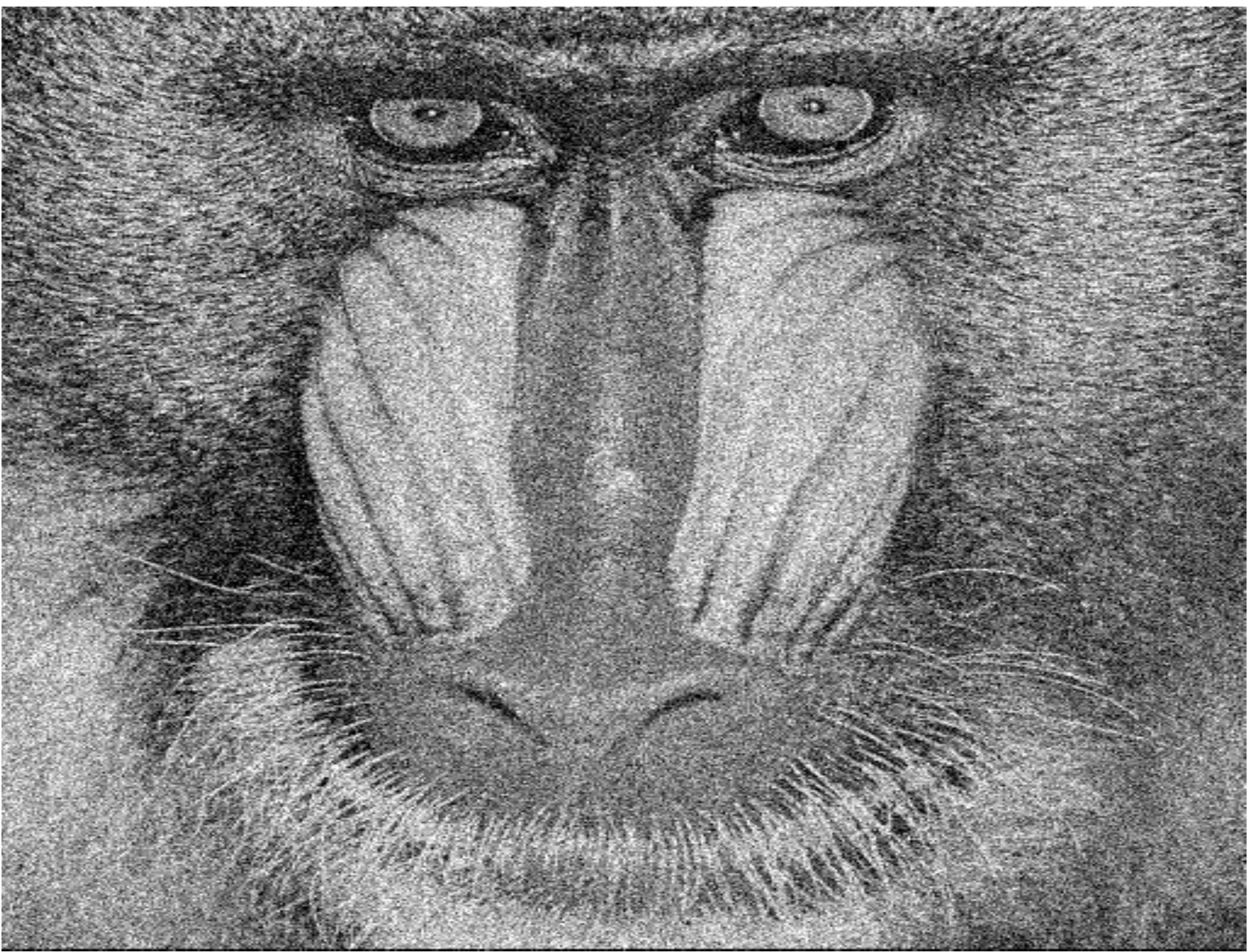}}
\subfigure[13/18, $\rho=1$]{\includegraphics[height=48mm,width=48mm]{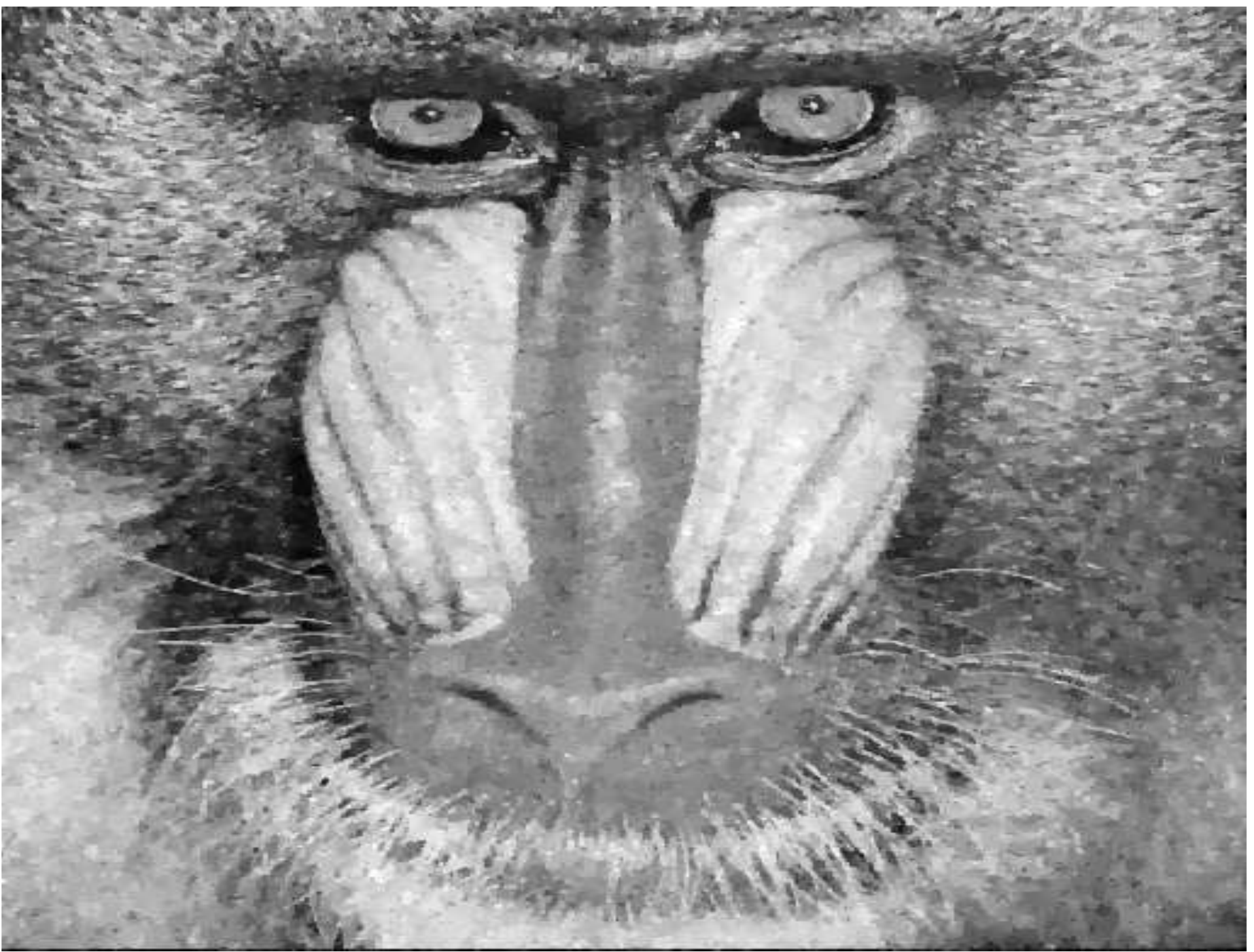}}
\subfigure[20/21, $\rho=1$]{\includegraphics[height=48mm,width=48mm]{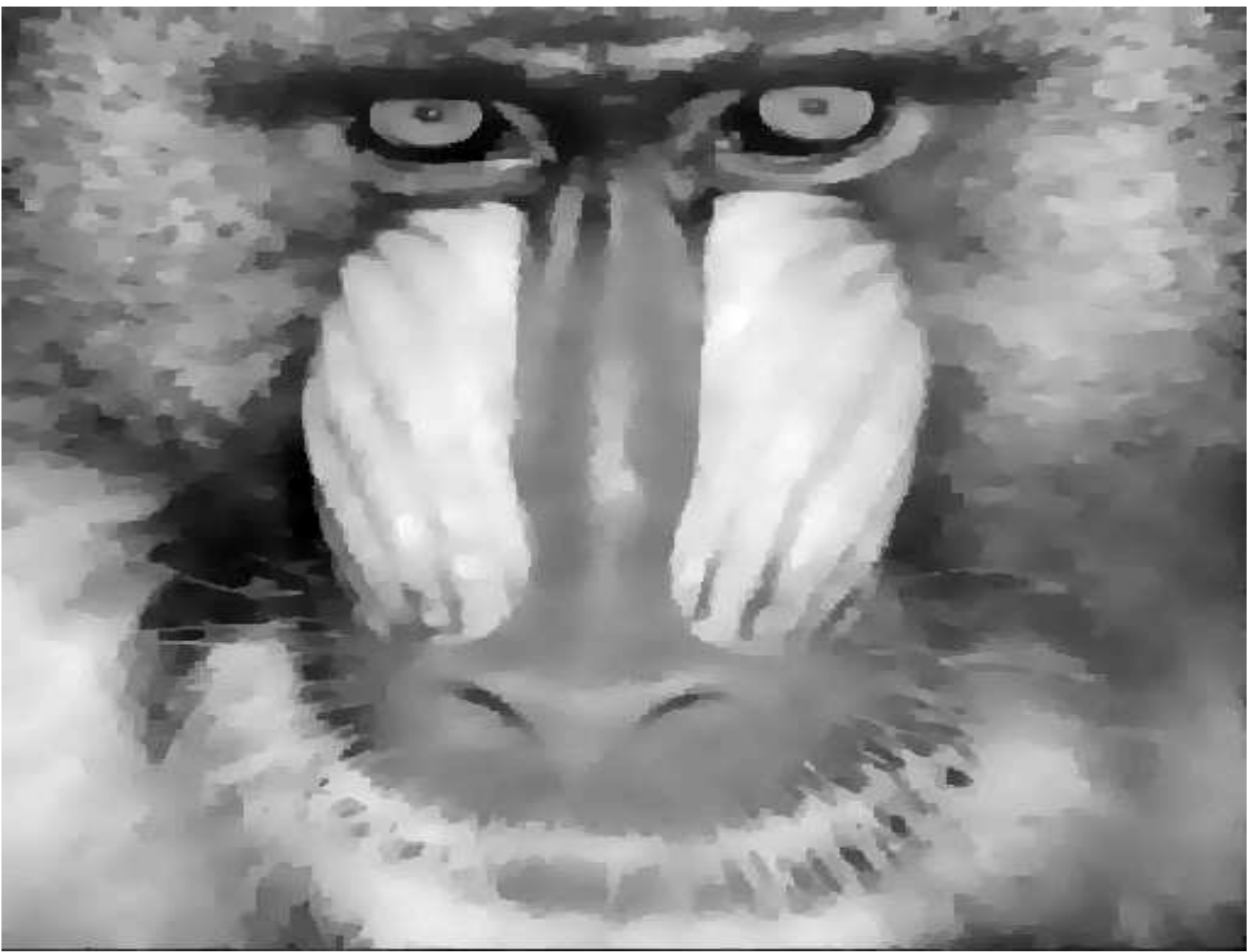}}
\subfigure[$\sigma=0.06$]{\includegraphics[height=48mm,width=48mm]{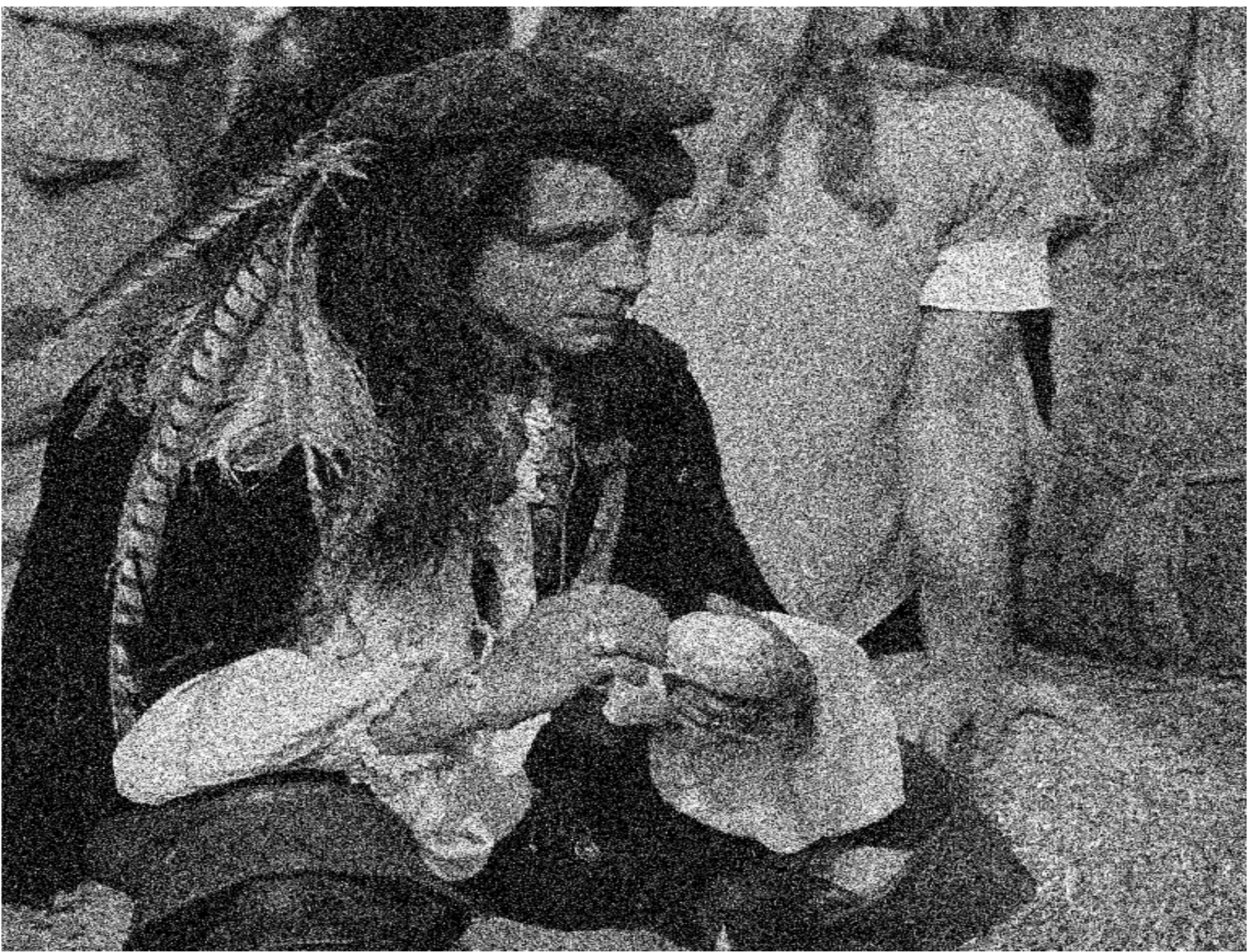}}
\subfigure[13/16, $\rho=1.5$]{\includegraphics[height=48mm,width=48mm]{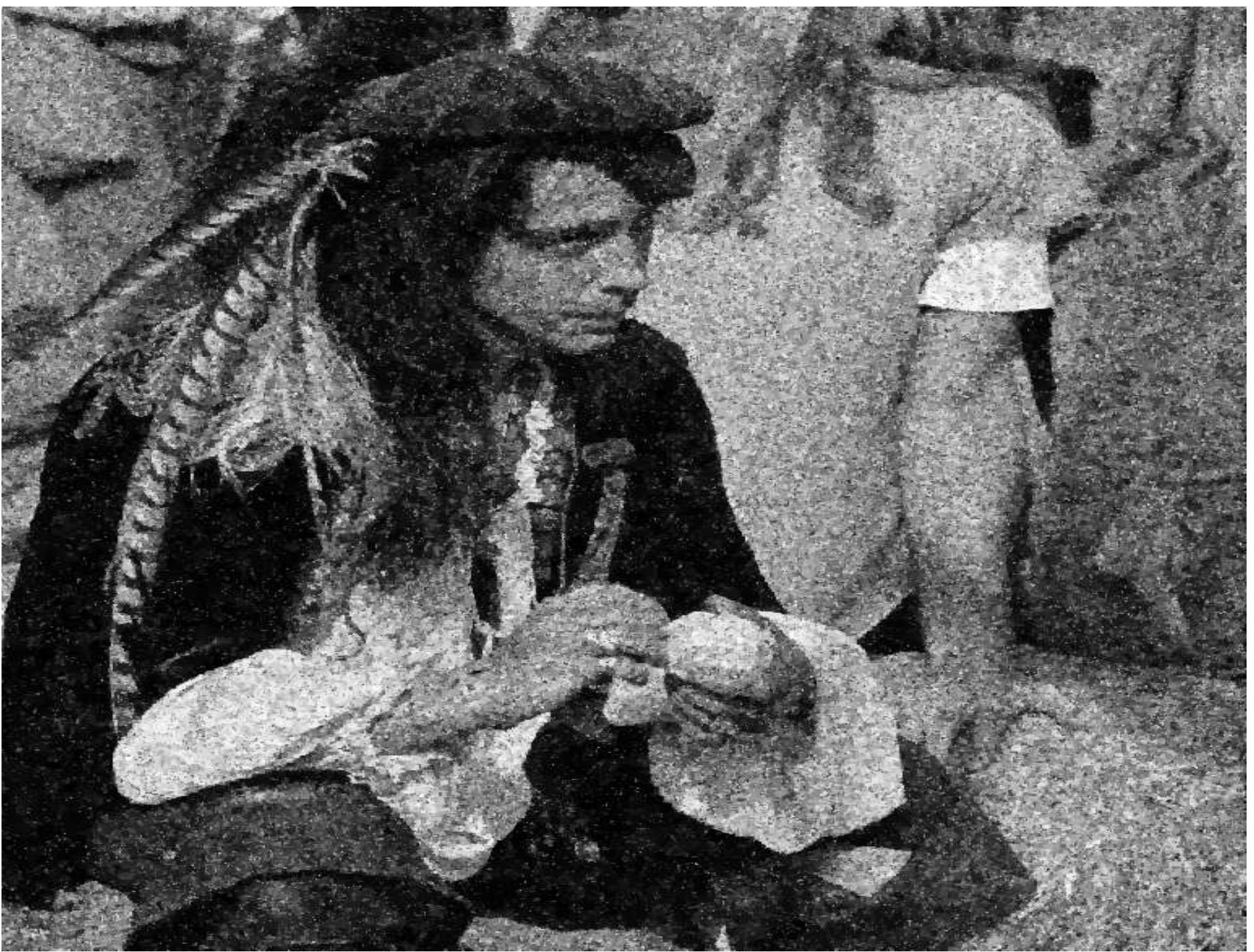}}
\subfigure[19/21, $\rho=1.5$]{\includegraphics[height=48mm,width=48mm]{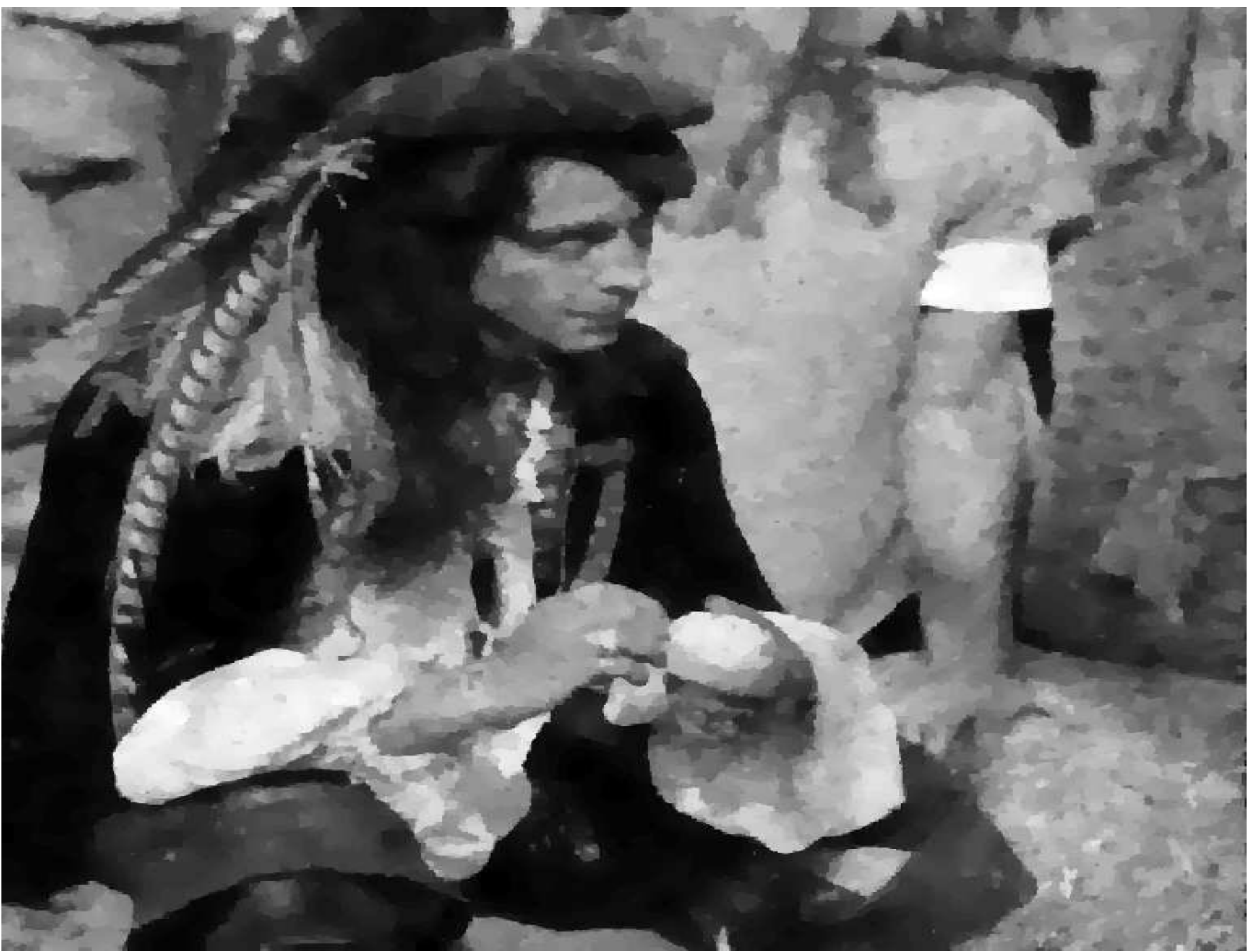}}
\caption{Denoising results. Images (a), (d) and (g) are contaminated with Gaussian noise, the value of variance ($\sigma$) is reported below each image. Images were denoised with $\zeta=20$ (center) and $\zeta=50$ (right). The value of $\rho$ and the number of iterations required to satisfy the stopping criterion, for both the PMM/ADMM, are listed below each image.}
\label{fig:denoising results}
\end{figure}

\begin{figure}[h!]
\centering 
\subfigure{\includegraphics[height=55mm,width=70mm]{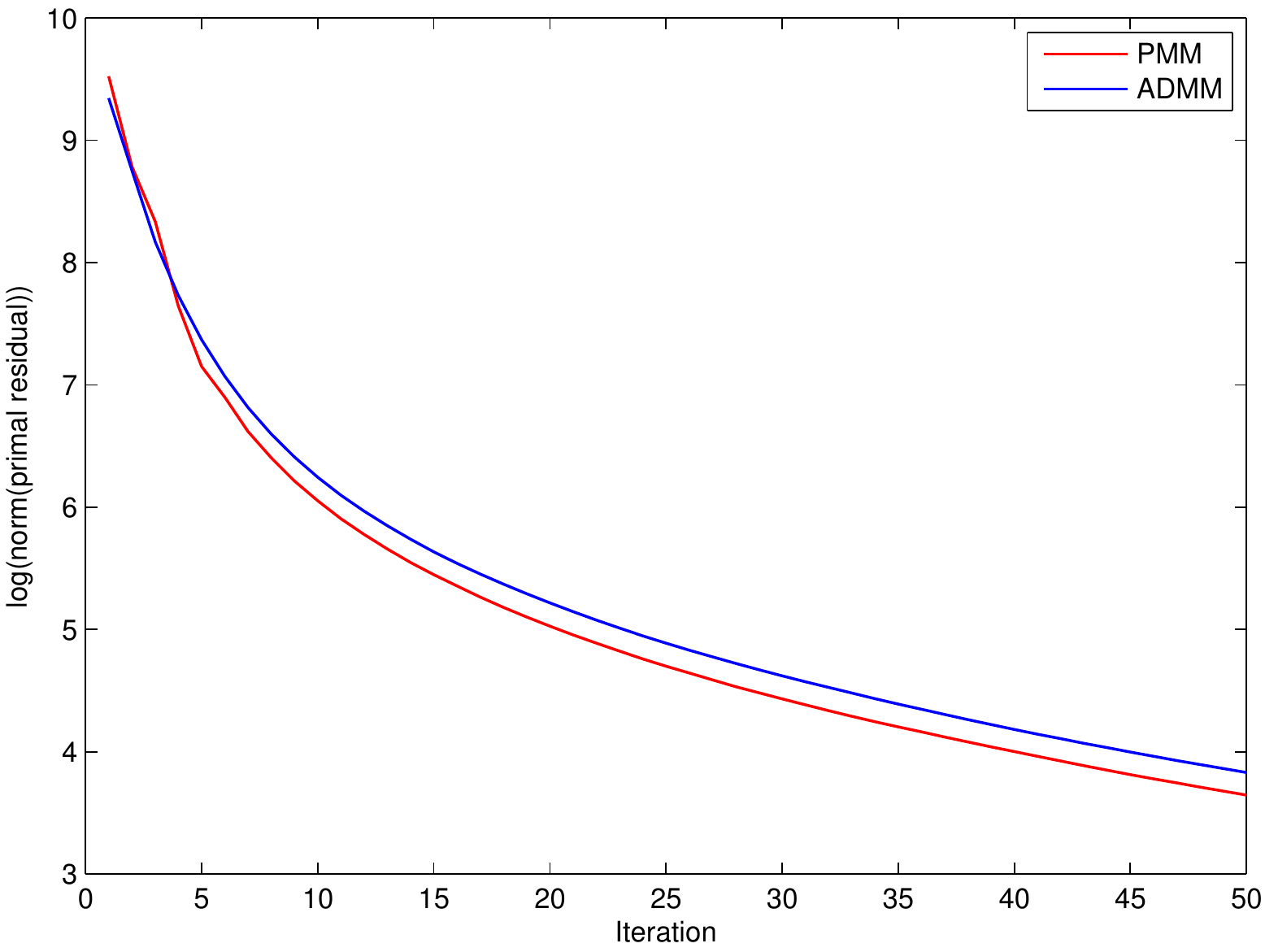}}
\subfigure{\includegraphics[height=55mm,width=70mm]{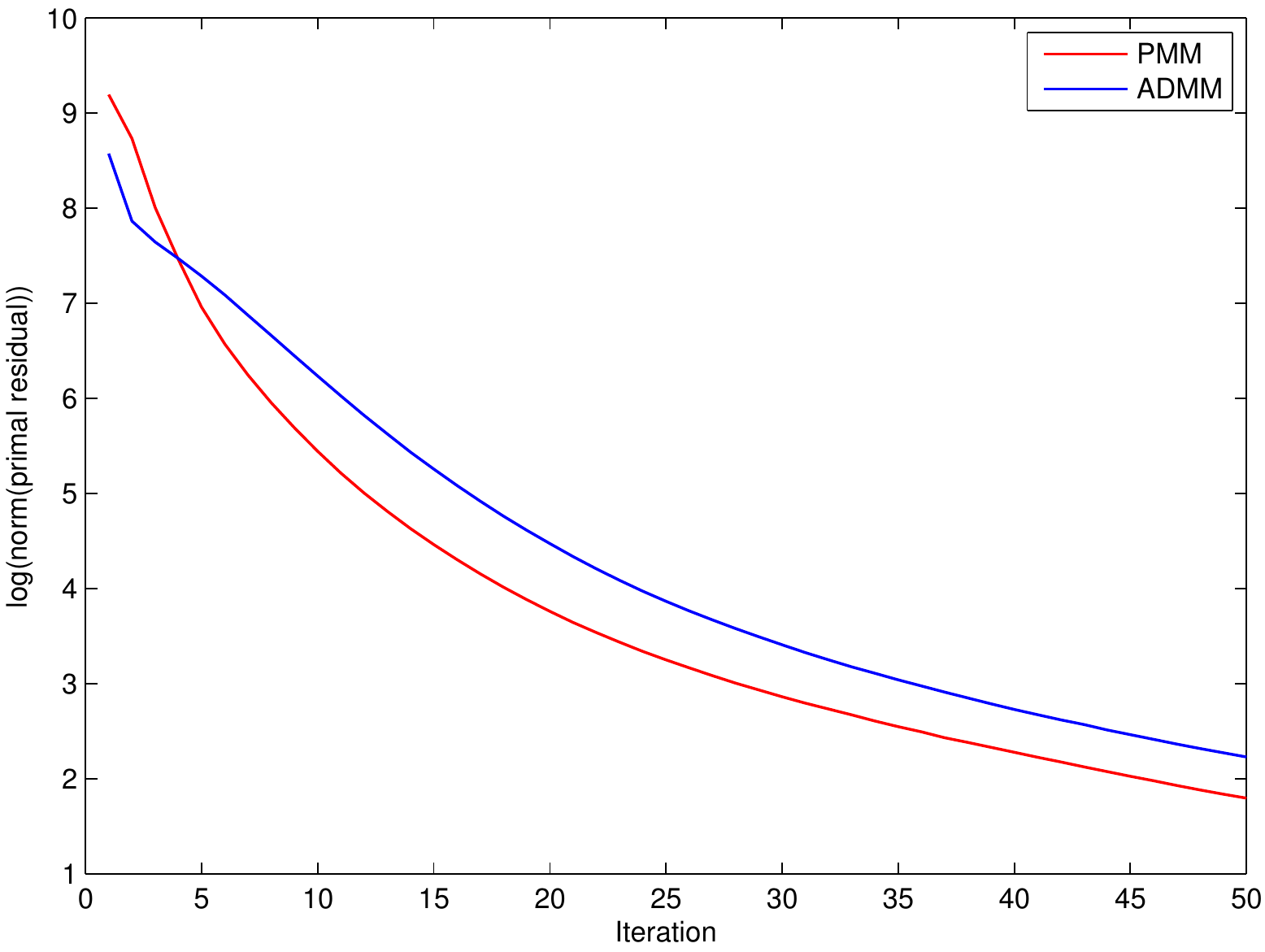}}
\subfigure{\includegraphics[height=55mm,width=70mm]{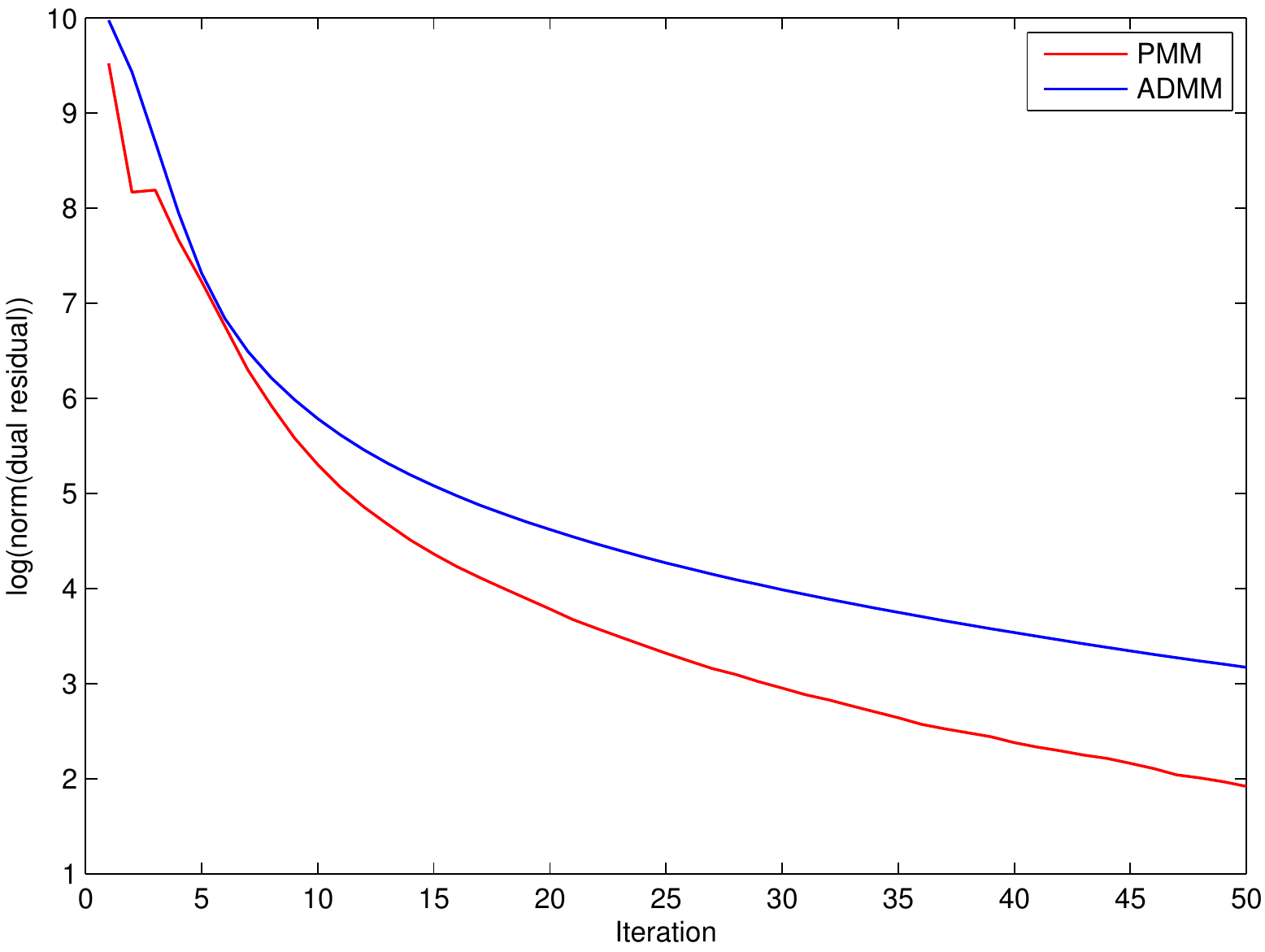}}
\subfigure{\includegraphics[height=55mm,width=70mm]{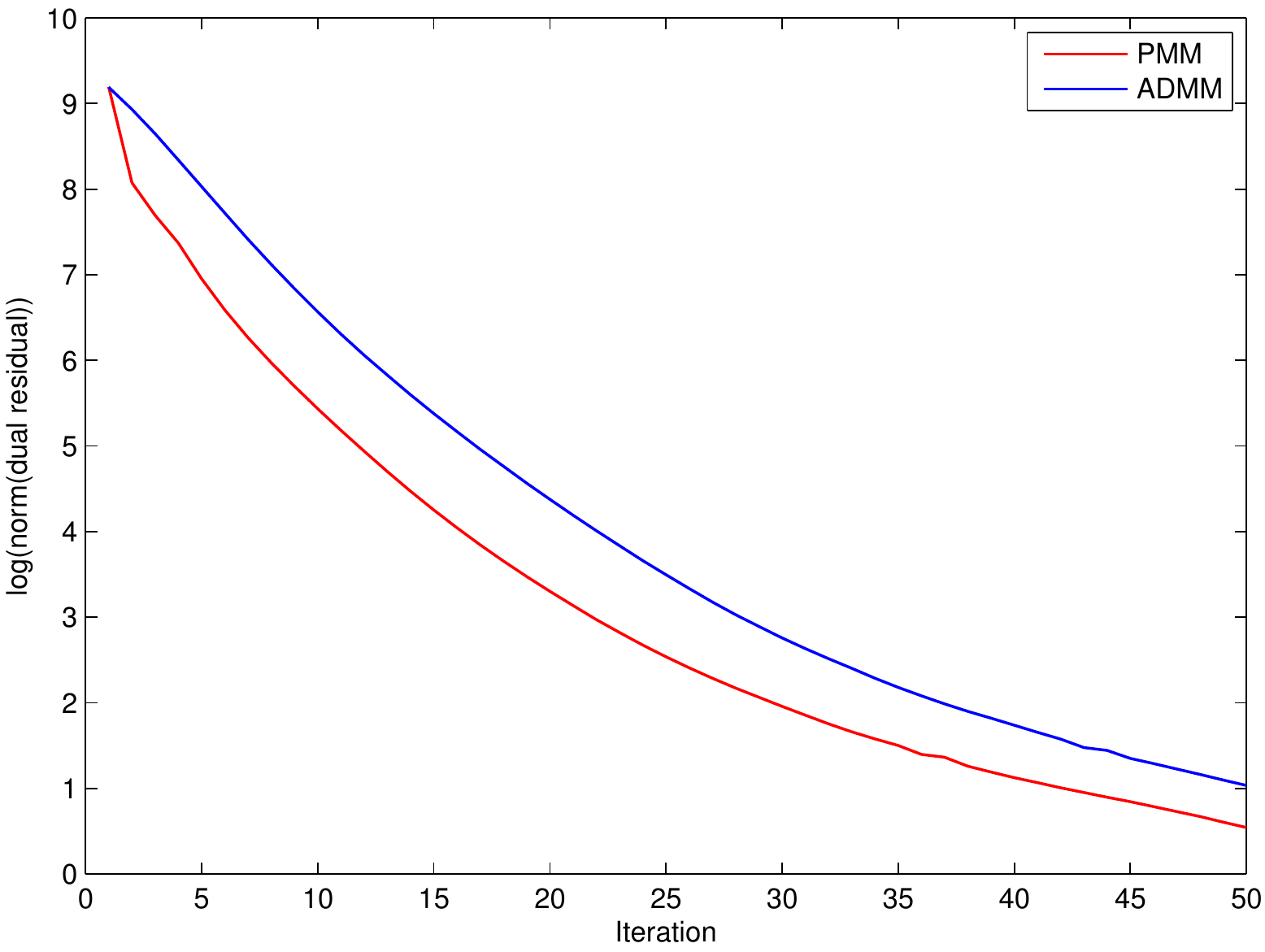}}
\subfigure{\includegraphics[height=55mm,width=70mm]{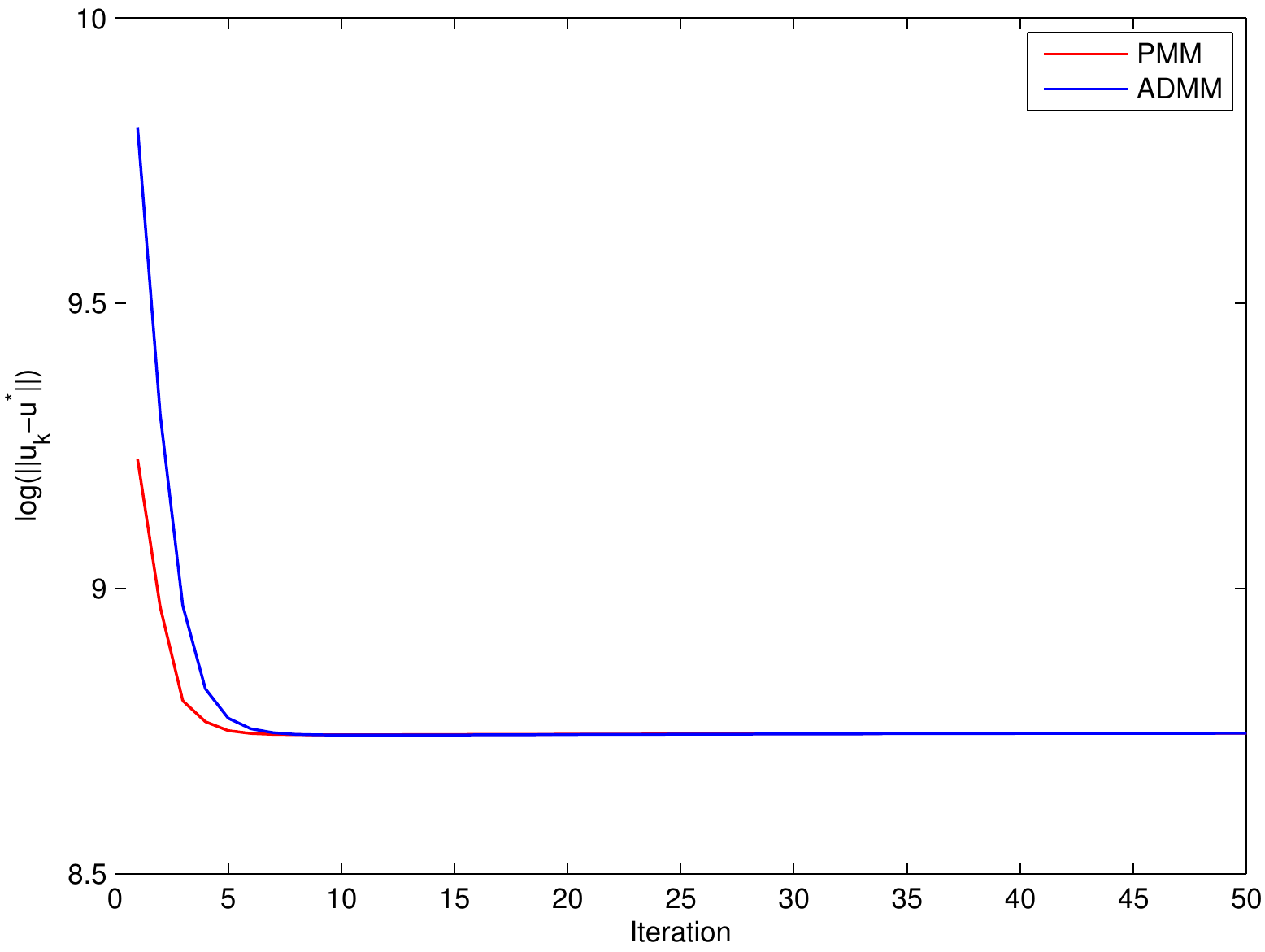}}
\subfigure{\includegraphics[height=55mm,width=70mm]{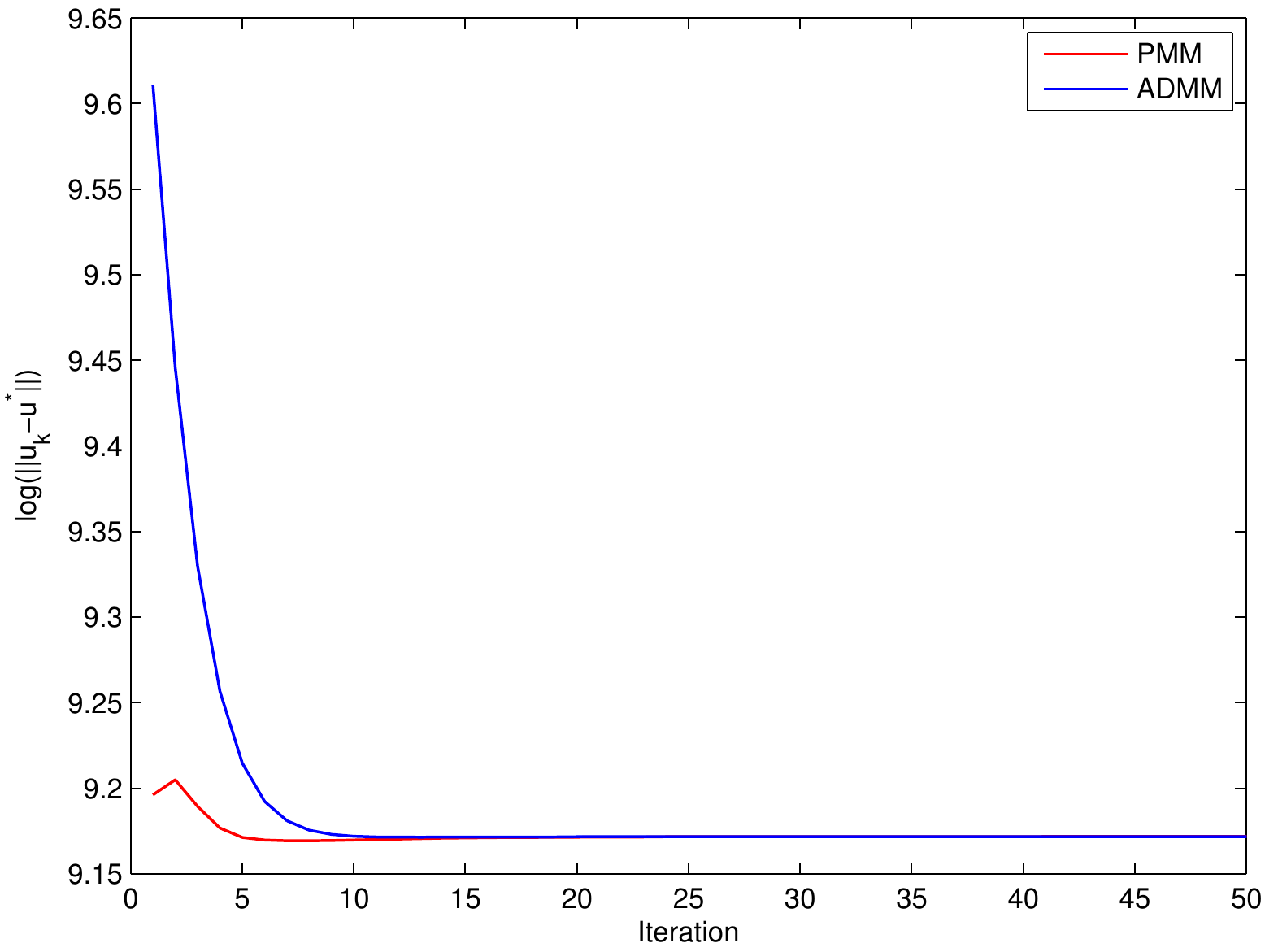}}
\caption{Residual curves of the PMM and ADMM for the TV denoising problems. (top) Primal error $\norm{\nabla u_k - v_k}$ vs iteration number $k$. (center) Dual error $\norm{x_k-y_k}$ vs iteration number $k$. (bottom) Error $\norm{u_k-u^\ast}$ vs iteration number $k$ ($u^\ast$ is the exact solution). (left) Convergence results are for the tested image Lena with $\sigma=0.06$, $\zeta=50$ and $\rho=1.5$. (right) Convergence results are for the tested image Baboon with $\sigma=0.02$, $\zeta=20$ and $\rho=1$.}
\label{fig:residuals lena}
\end{figure}

\begin{table}[h]
 \begin{center}
   \begin{tabular}{cccc|cc}
   \hline
   Image & $\zeta$ & $\rho$& $\sigma$ & PMM & ADMM\\
   \hline
   Lena & $20$  & $1$ & $0.02$ & $13 (2.789)$ & $17 (3.510)$\\
      Lena & $20$  & $1.5$ & $0.02$ & $12 (1.803)$ & $14 (2.050)$\\
          Lena & $50$  & $1$ & $0.06$ & $19 (2.855)$ & $21 (3.050)$\\
          Lena & $50$  & $1.5$ & $0.06$ & $17 (2.642)$ & $18 (2.656)$\\
          Baboon & $50$  &  $1$ & $0.02$ & $20 (2.862)$ & $21 (2.811)$\\
          Baboon & $50$   &  $1.5$ & $0.02$ & $19 (2.920)$ & $19 (2.789)$\\
          Baboon & $20$  &  $1$ & $0.06$ & $15 (2.311)$ & $21 (2.928)$\\
          Baboon & $20$   &  $1.5$ & $0.06$ & $13 (2.025)$ & $15 (2.336)$\\
        Man & $50$   &  $1$ & $0.02$ & $24 (7.537)$ & $24 (7.364)$\\
          Man & $50$   &  $1.5$ & $0.02$ & $21 (6.700)$ & $22 (6.597)$\\
          Man & $20$   &  $1$ & $0.06$ & $16 (5.322)$ & $21 (6.540)$\\         
          Man & $20$   &  $1.5$ & $0.06$ & $13 (4.395)$ & $16 (5.143)$\\    
          Man & $50$   &  $1$ & $0.06$ & $21 (6.625)$ & $24 (7.371)$\\       
          \hline\\    
\end{tabular}
 \end{center}
 \caption{Iterations and computation times (seconds) in parenthesis required for the TV problem.}
 \label{tab:comparison pmm admm}
\end{table}

The operation of highest computational cost within each iteration of the PMM, and ADMM, for the TV problem, consists in solving problem \eqref{u problem}. In our tests we solved this step for both algorithms using the  Conjugate Gradient (CG) method with tolerance $10^{-5}$. This strategy consistently yielded convergence in fewer iterations when using the PMM. Table \ref{tab:cg comparison} presents the total number of iteration executed by the CG method in each algorithm for some specific experiments. In the tests presented both methods were stopped at iteration 20. 

\begin{table}[h]
\begin{center}
\begin{tabular}{cccc|cc}
\hline
Image & $\zeta$ & $\rho$& $\sigma$ & PMM & ADMM\\
\hline
Lena & $20$  & $1$ & $0.02$ & $108$ & $117$\\
Lena & $20$ & $1.5$ & $0.06$ & $101$ & $110$\\
Baboon & $20$   &  $1.5$ & $0.02$ & $102$ & $110$\\
Baboon & $50$   &  $1$ & $0.02$ & $121$ & $122$\\
Man & $20$   &  $1.5$ & $0.06$ & $104$ & $112$\\    
Man & $50$   &  $1$ & $0.06$ & $124$ & $126$\\
\hline\\
\end{tabular}
\end{center}
\caption{Total number of iteration of CG method. Tests were stopped at iteration 20.}
\label{tab:cg comparison}
\end{table}

However, the authors of \cite{goldstein_bregman} observed that the ADMM (SB method) attained optimal efficiency executing, at each iteration of the algorithm, just a single iteration of an iterative method to solve problem \eqref{u problem}. This inexact minimization can be justified by the convergence theory for the generalized ADMM developed by Eckstein and Bertsekas in \cite{eck_ber_dr}, see also \cite{esser2009}. 

In \cite{eck_sv_09}, Eckstein and Svaiter generalized the projective-splitting algorithm for the sum of $N\geq2$ maximal monotone operators, and they introduced a relative error criterion for approximately evaluating the proximal mappings.
This framework suggests that the PMM can also admit inexact minimization for the subproblems. Indeed, as Figure \ref{fig:denoising one cg-it} below shows, the PMM also yields good denoised images performing a single iteration of the CG method at each step of the algorithm.

\begin{figure}[h]
\centering
\subfigure[$\sigma=0.06$]{\includegraphics[height=48mm,width=48mm]{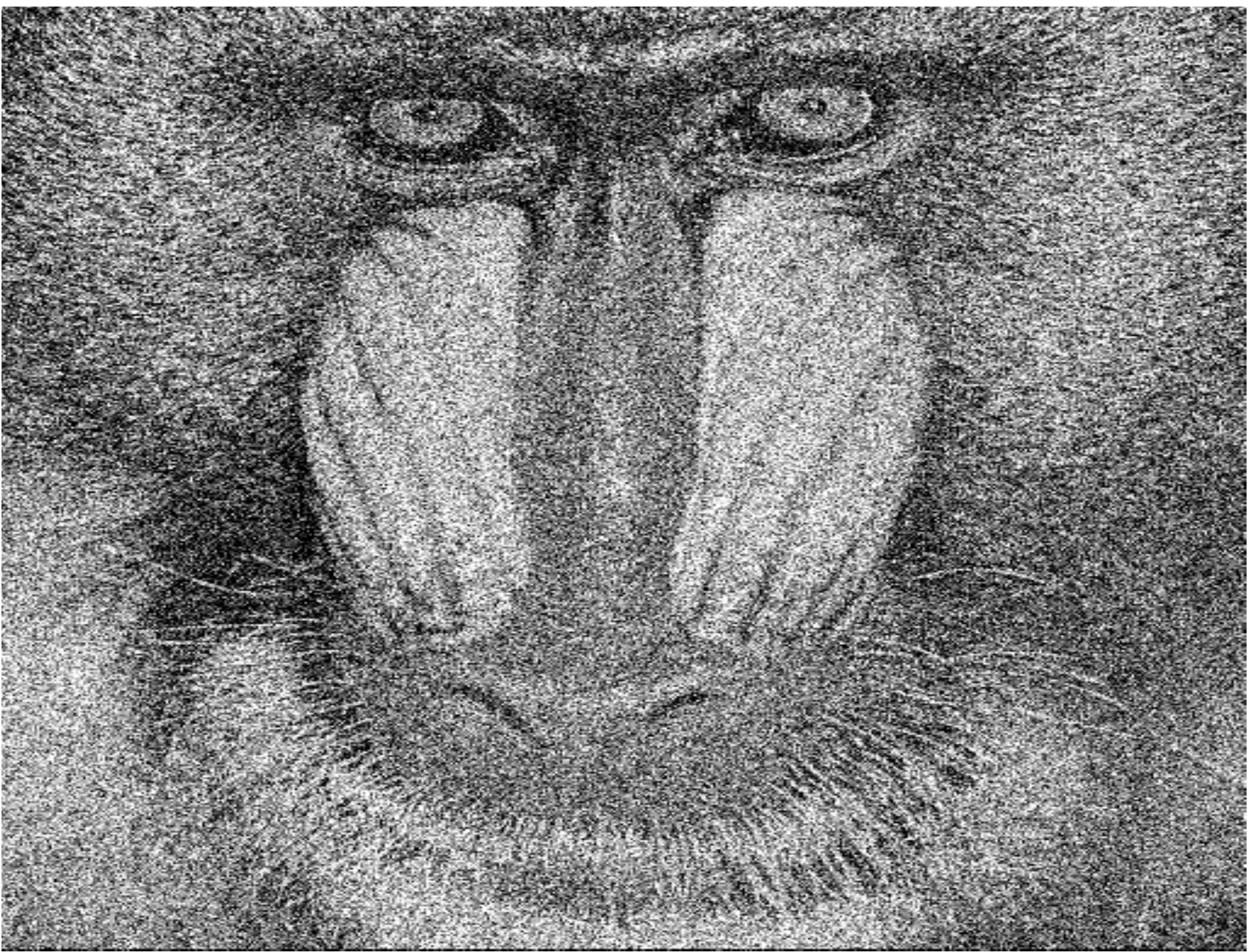}}
\subfigure[20(1.594), $\zeta=50$, $\rho=1$]{\includegraphics[height=48mm,width=48mm]{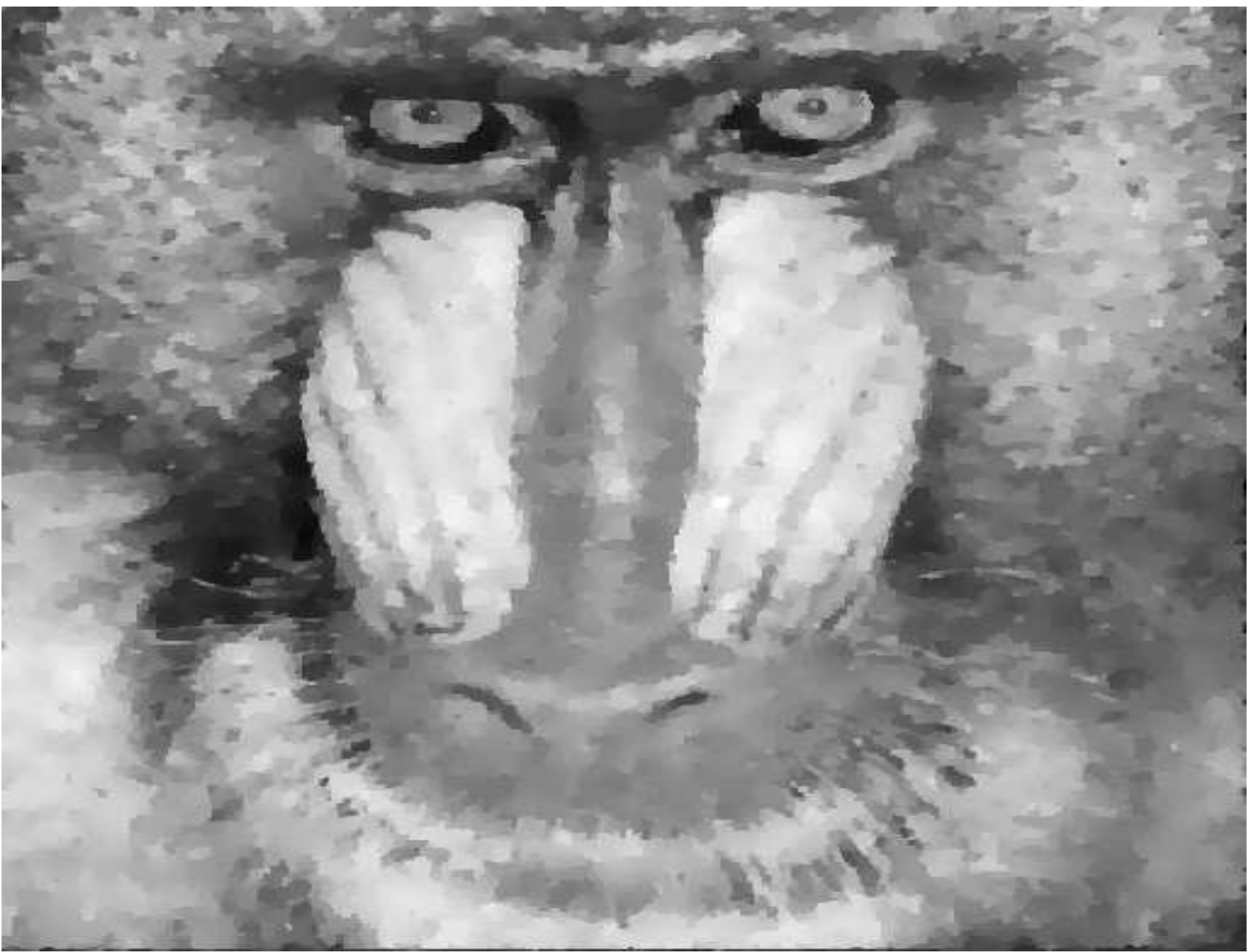}}
\subfigure[22(1.557), $\zeta=50$, $\rho=1$]{\includegraphics[height=48mm,width=48mm]{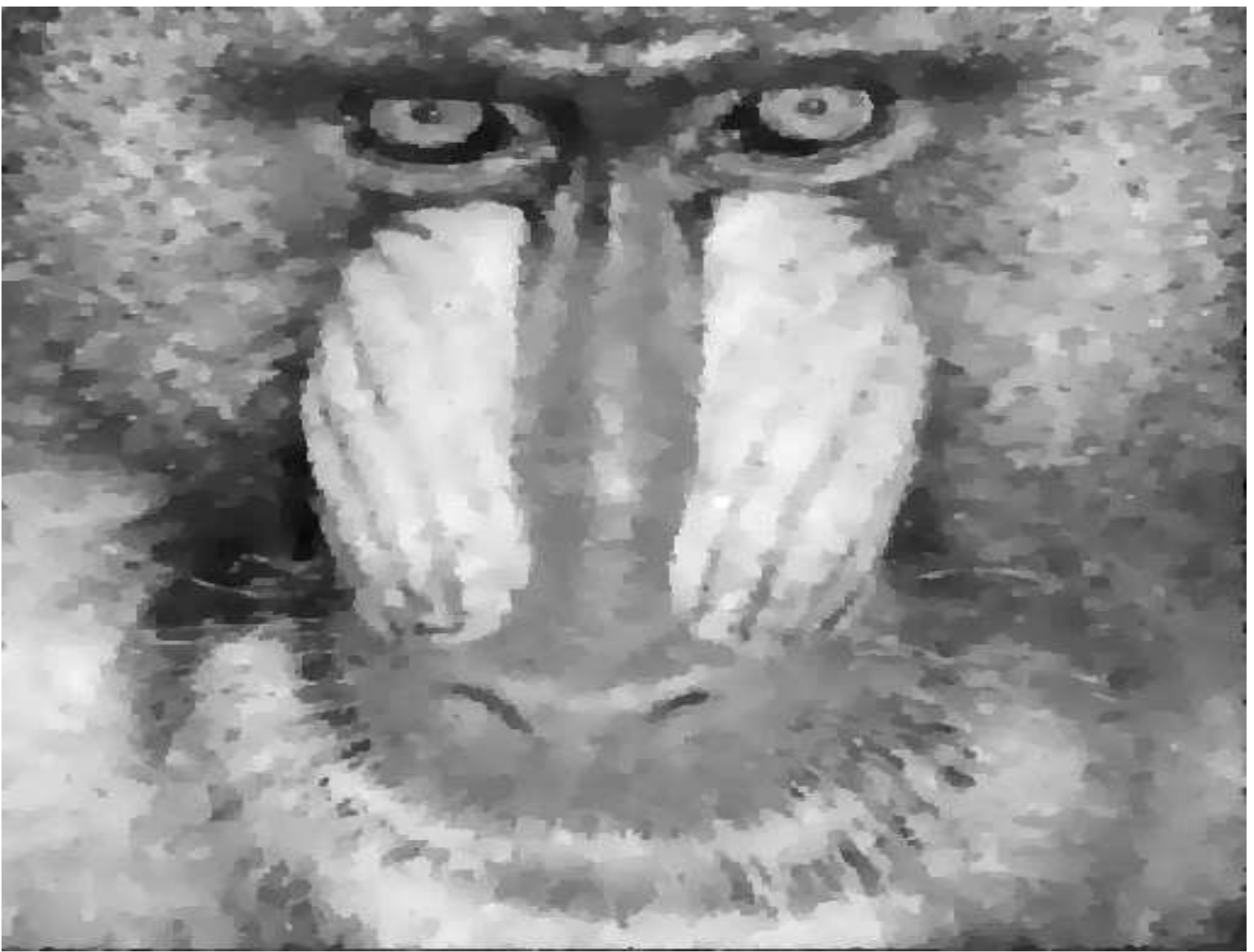}}
\subfigure[$\sigma=0.02$]{\includegraphics[height=48mm,width=48mm]{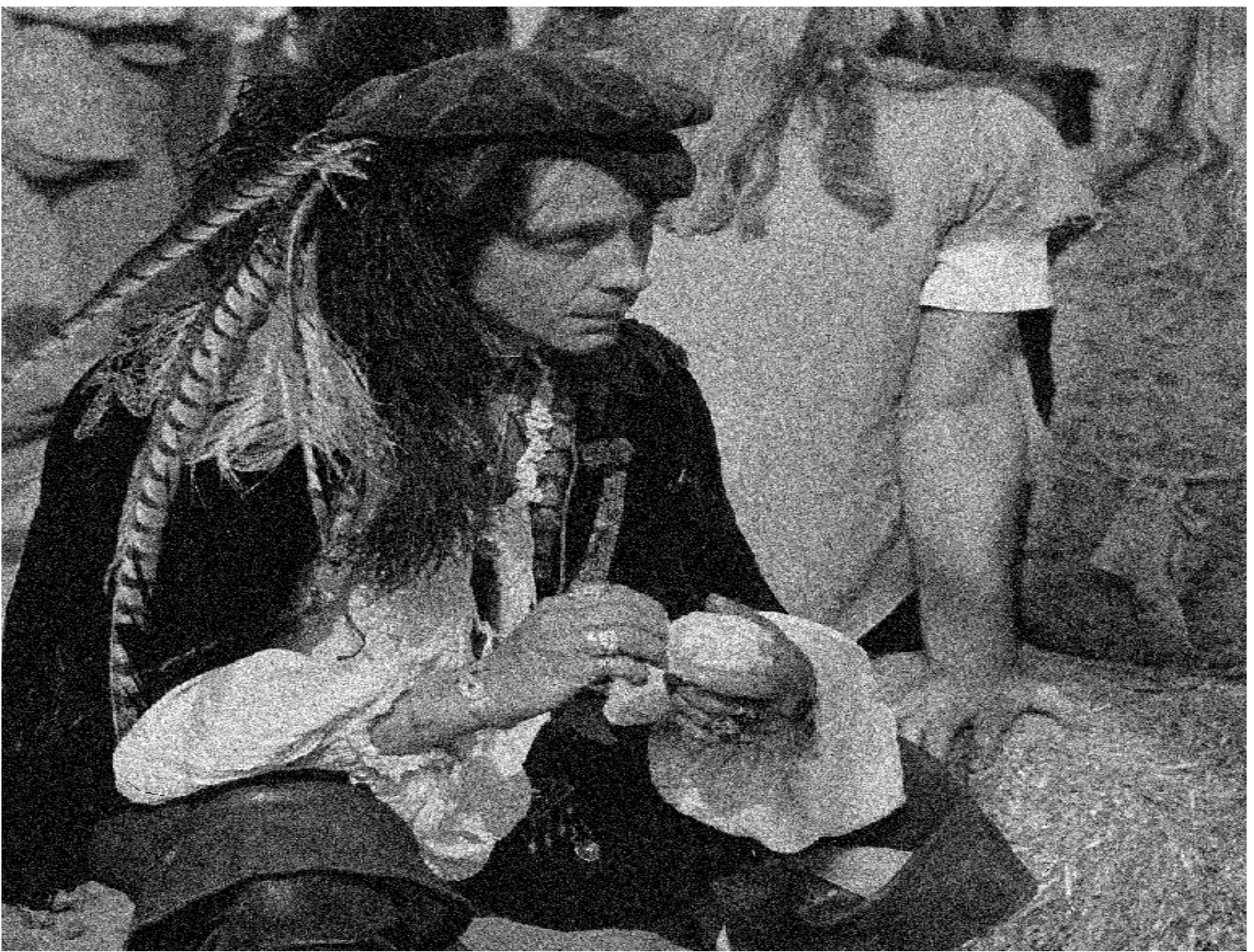}}
\subfigure[14(2.436), $\zeta=20$, $\rho=1.5$]{\includegraphics[height=48mm,width=48mm]{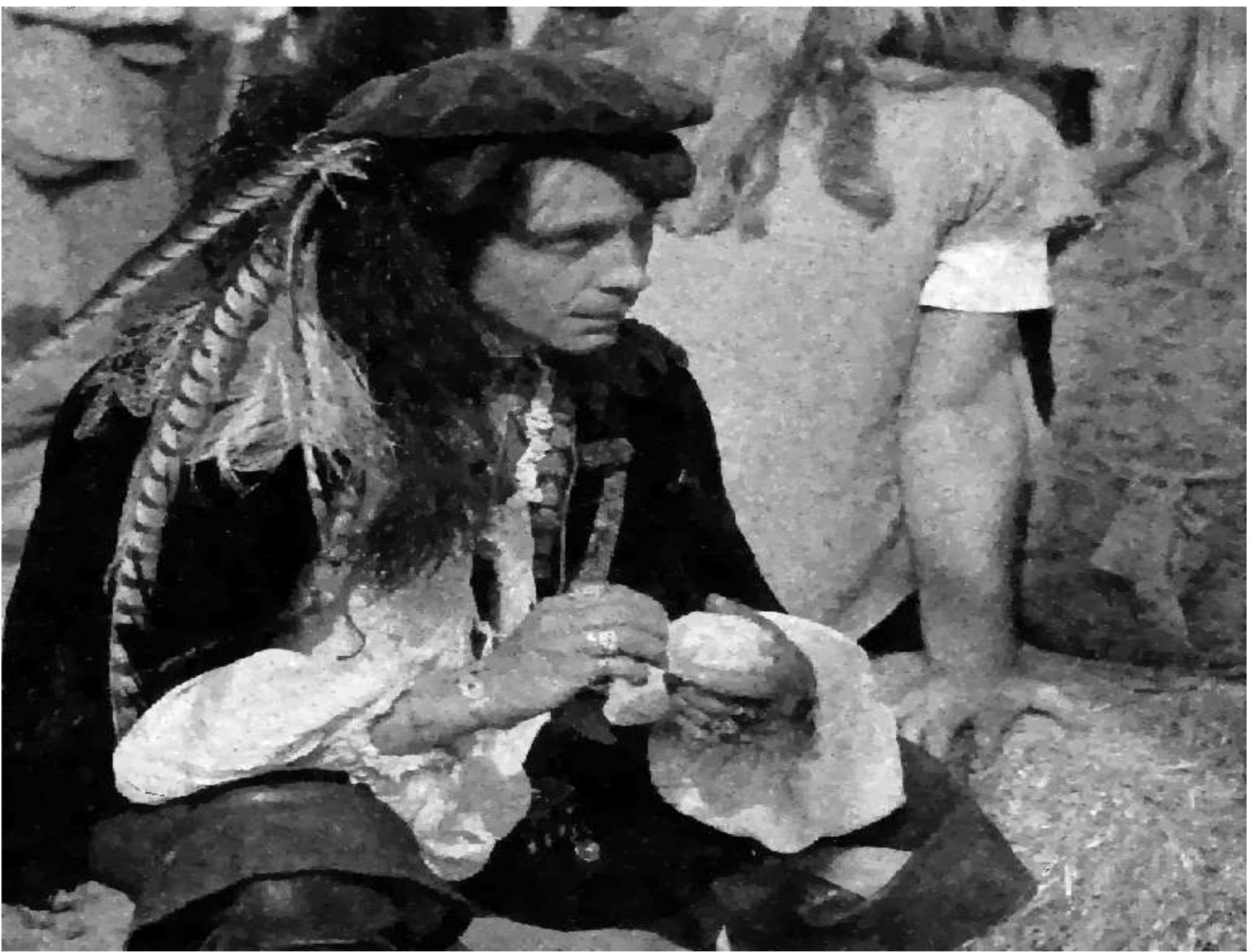}}
\subfigure[16(2.514), $\zeta=20$, $\rho=1.5$]{\includegraphics[height=48mm,width=48mm]{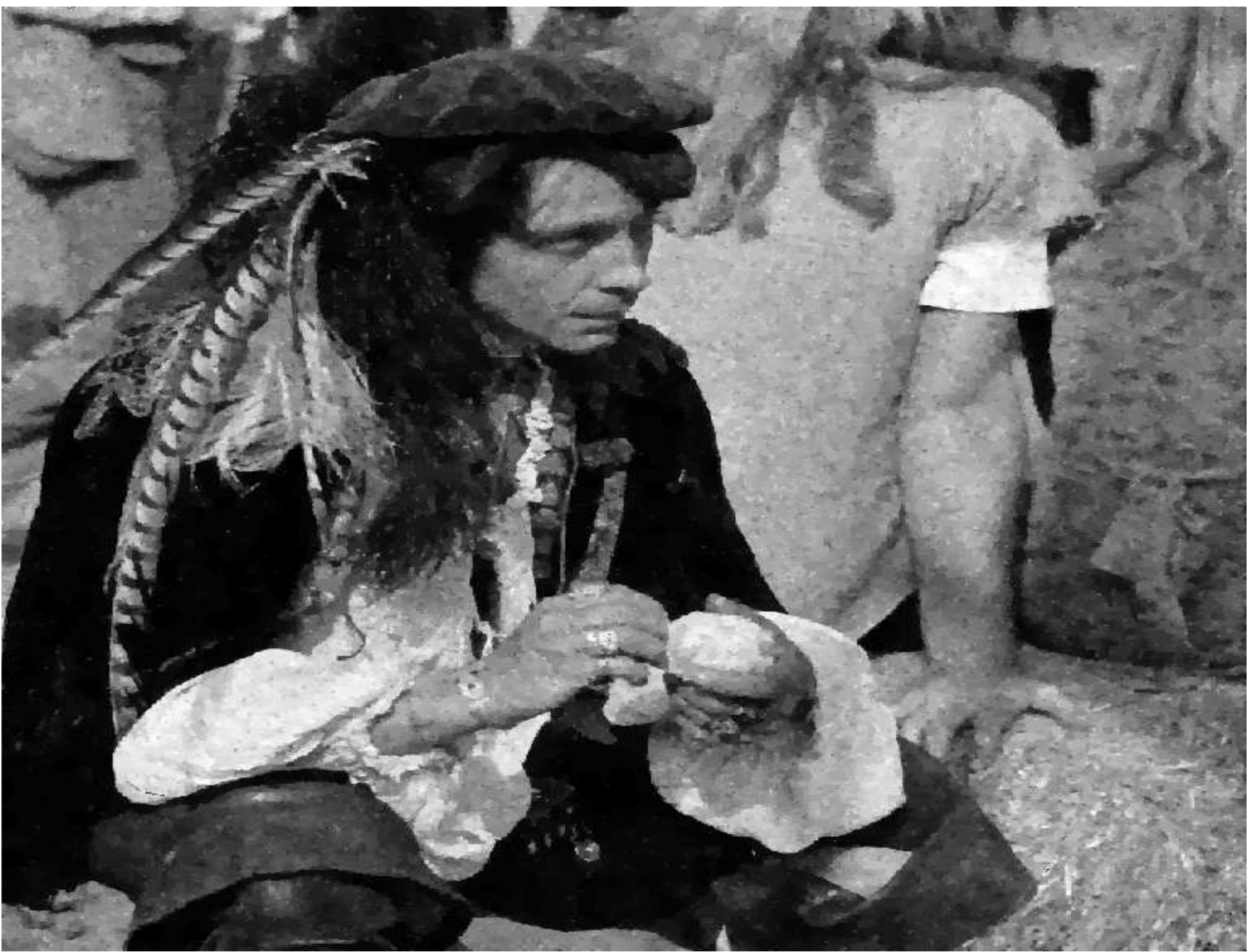}}
\caption{Denoising with one iteration of CG method per iteration. (left) Noisy images, the value of variance is reported below each image. (center) Images denoised with PMM. (right) Images denoised with ADMM. The number of iterations, the total time in seconds (in parenthesis); as well as the used values of $\zeta$ and $\rho$ are displayed below each image.}
\label{fig:denoising one cg-it}
\end{figure}

\subsection{Compressed sensing}
\label{subs:cs}
In many areas of applied mathematics and computer science it is often desirable to reconstruct a signal from small amount 
of data. Compressed sensing is a signal processing technique that allow the reconstruction of signals and images from 
small number of measurements, provided that they have a sparse representation. This technique has gained considerable 
attention in the signal processing community since the works of Cand\`es, Romberg and Tao 
\cite{CanRomTer2006}, and of Donoho \cite{Donoho2006}, and it has had a significant impact in several 
applications, for example in imaging, video and medical imaging.

For testing the PMM we consider a particular application of compressed sensing in Magnetic Resonance Imaging 
(MRI), which is an essential medical imaging tool. MRI is based on the reconstruction of an image from a subset of measurements in the Fourier domain. This imaging 
problem can be modeled by the optimization problem
\begin{equation}
\label{eq:cs problem}
\min_u \,\, TV(u) + \dfrac{\zeta}{2}\norm{RFu-b}_F^2,
\end{equation}
where $TV$ is the total variation norm \eqref{eq:tv-norm}, $F$ is the Discrete Fourier Transform, $R$ is a diagonal matrix, $b$ is the known Fourier data and $u$ is the unknown image that we wish to reconstruct. 

The matrix $R$ has a $1$ along the diagonal at entries corresponding to the Fourier coefficients that were measured, and $0$ for the unknown coefficients.  The second term in \eqref{eq:cs problem} induces the Fourier transform of the reconstructed image to be close to the measured data, while the TV term in the minimization enforces ``smoothness'' of the image. The parameter $\zeta>0$ provides a tradeoff between the fidelity term and the smoothness term.

Problem \eqref{eq:cs problem} can be posed as a linearly constrained minimization problem \eqref{eq:convex opt prob} in much the same manner as was done for the TV problem in the previous subsection. Therefore, to apply the PMM to \eqref{eq:cs problem} we take $f(u) = \dfrac{\zeta}{2}\norm{RFu-b}^2_F$, $g(v)=\norm{v}_1$, $M = \nabla$, $C=-I$ and $d=0$. The resulting minimization problems are 
\begin{align}
\label{eq:v-subproblem-cs}
& v_k = \arg \min_v \,\, \norm{v}_1 - \inpr{z_{k-1} + \lambda w_{k-1}}{v} + \dfrac\lambda2\norm{v}^2_F,
\end{align}
and
\begin{align}
\label{eq:u-subproblem-cs}
& u_k = \arg \min_u \,\, \dfrac{\zeta}{2}\norm{RFu-b}^2_F + \inpr{z_{k-1} - \lambda v_k}{\nabla u} + \dfrac\lambda2\norm{\nabla u}^2_F.
\end{align}

Problem \eqref{eq:v-subproblem-cs} can be solved explicitly using the \textbf{shrink} operator \eqref{eq:shrink}. Indeed, by the optimality conditions for this problem we have
\begin{equation*}
v_k = \textbf{shrink}\left(\dfrac1\lambda z_{k-1} + w_{k-1},\dfrac{1}{\lambda}\right).
\end{equation*}

The optimality condition for the minimization problem \eqref{eq:u-subproblem-cs} is
\begin{equation*}
0 = \zeta  F^T R^T (RFu_k - b) + \nabla^\ast(z_{k-1} - \lambda v_k) + \lambda\nabla^\ast\nabla u_k, 
\end{equation*}
or equivalently 
\begin{equation*}
(\zeta F^T R^T RF + \lambda\nabla^\ast\nabla)u_k = \zeta F^T R^T b - \nabla^\ast(z_{k-1} - \lambda v_k).
\end{equation*}
Thus, we obtain $u_k$, the solution of the system above, by
\begin{equation*}
u_k = F^T( \zeta R^T R + \lambda F\nabla^\ast\nabla F^T)^{-1}F(\zeta F^T R^T b - \nabla^\ast(z_{k-1} - \lambda v_k)).
\end{equation*}

We tested the PMM on two synthetic phantom. The first is the digital Shepp-Logan phantom with dimensions $256\times256$, which was created with the Matlab function ``\texttt{phantom}". For the compressed sensing problem of reconstructing this image we measured at random $25\%$ of the Fourier coefficients. The second experiment was done with a CS-Phantom of size $512\times512$, which was taken from the mathworks web site. For this image we used $50\%$ of the Fourier coefficients. As stopping condition for these problems was used the criterion given by the residuals for the KKT conditions. More specifically, the PMM and ADMM were stopped when both, the primal and dual residual, associated with each method was less than a prefixed tolerance. 
 Figure \ref{fig:recovered phantoms} shows the test images and their reconstructions using the PMM.
 
\begin{figure}[h]
\centering
\subfigure{\includegraphics[height=60mm,width=60mm]{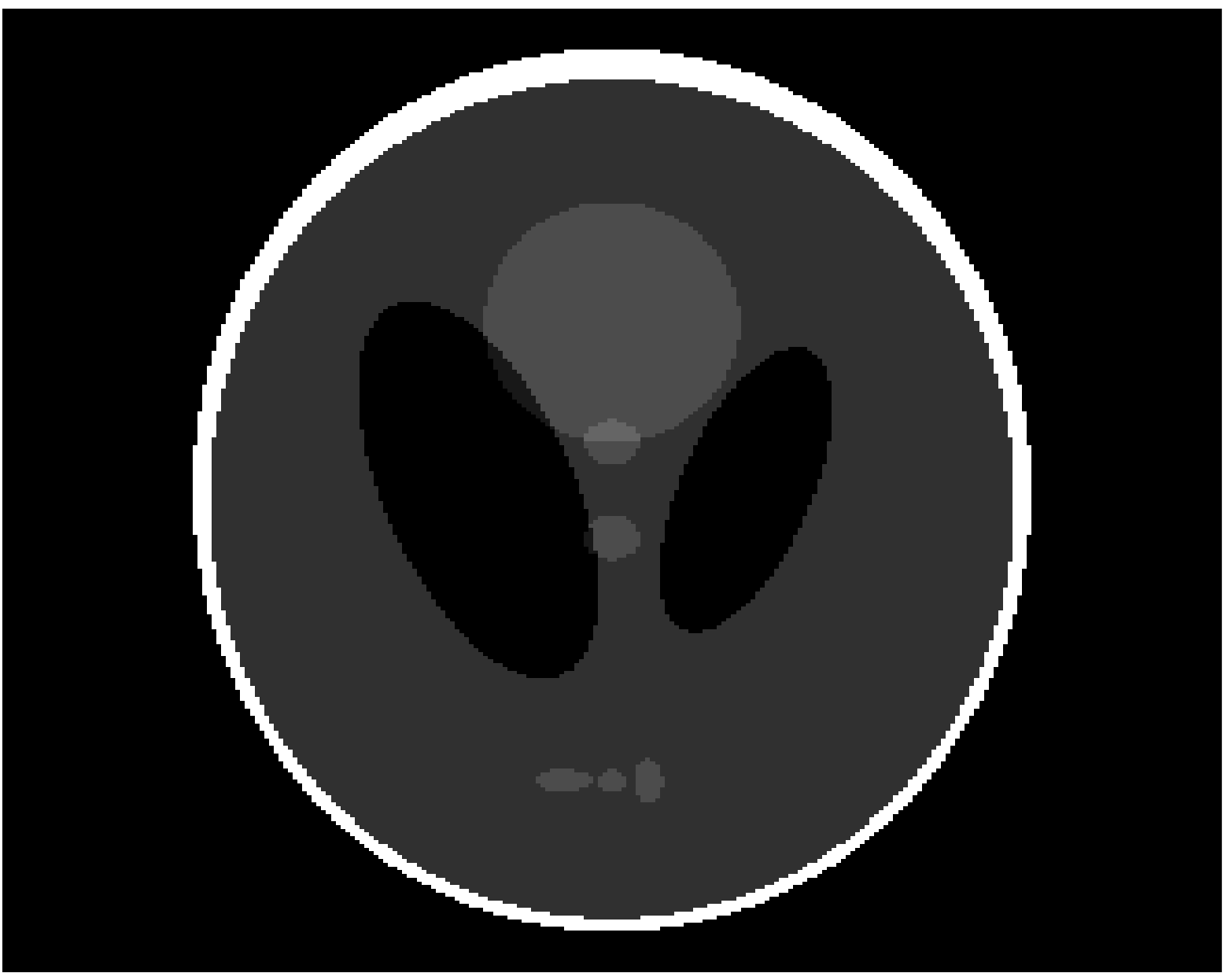}}
\subfigure{\includegraphics[height=60mm,width=60mm]{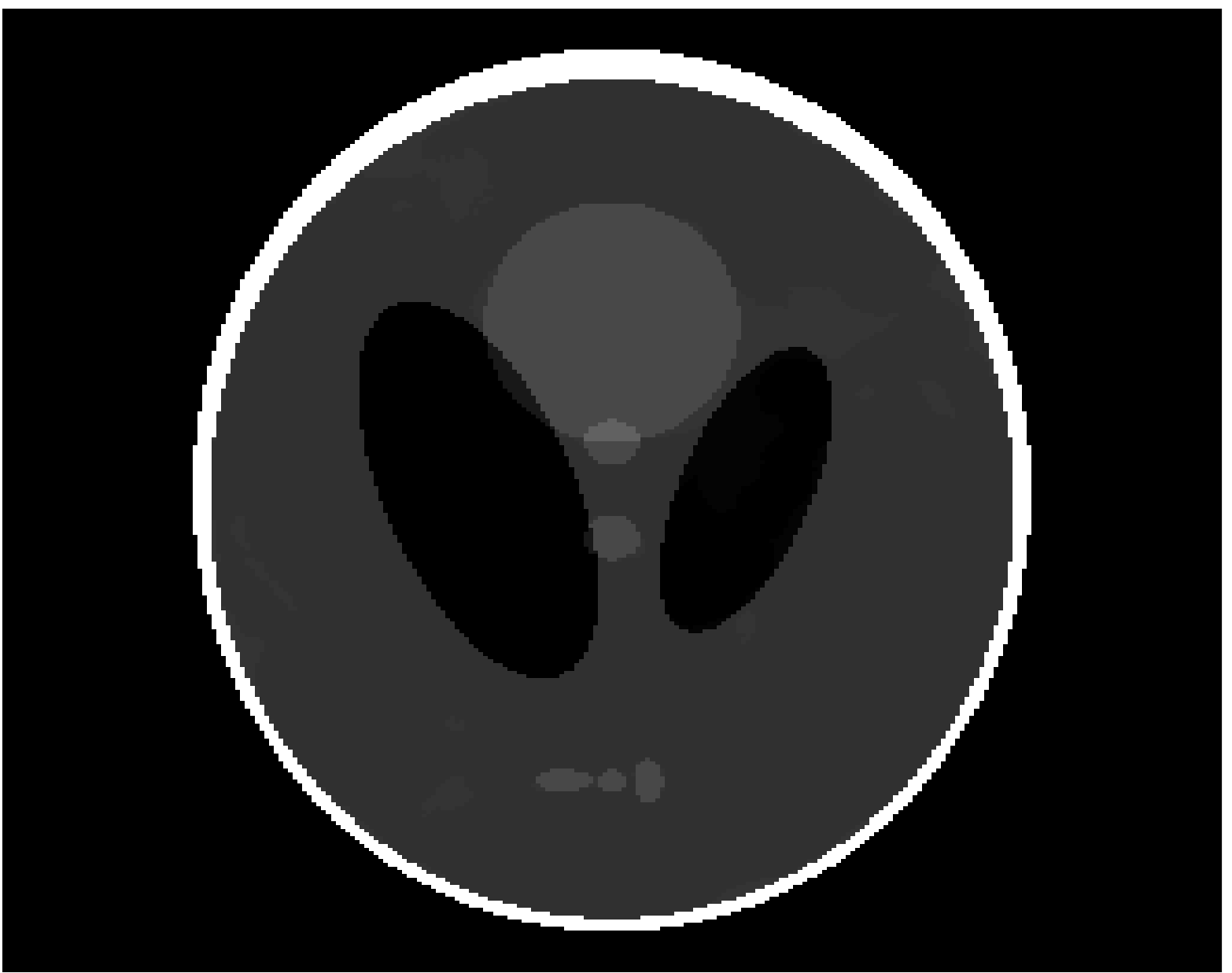}}
\subfigure{\includegraphics[height=60mm,width=60mm]{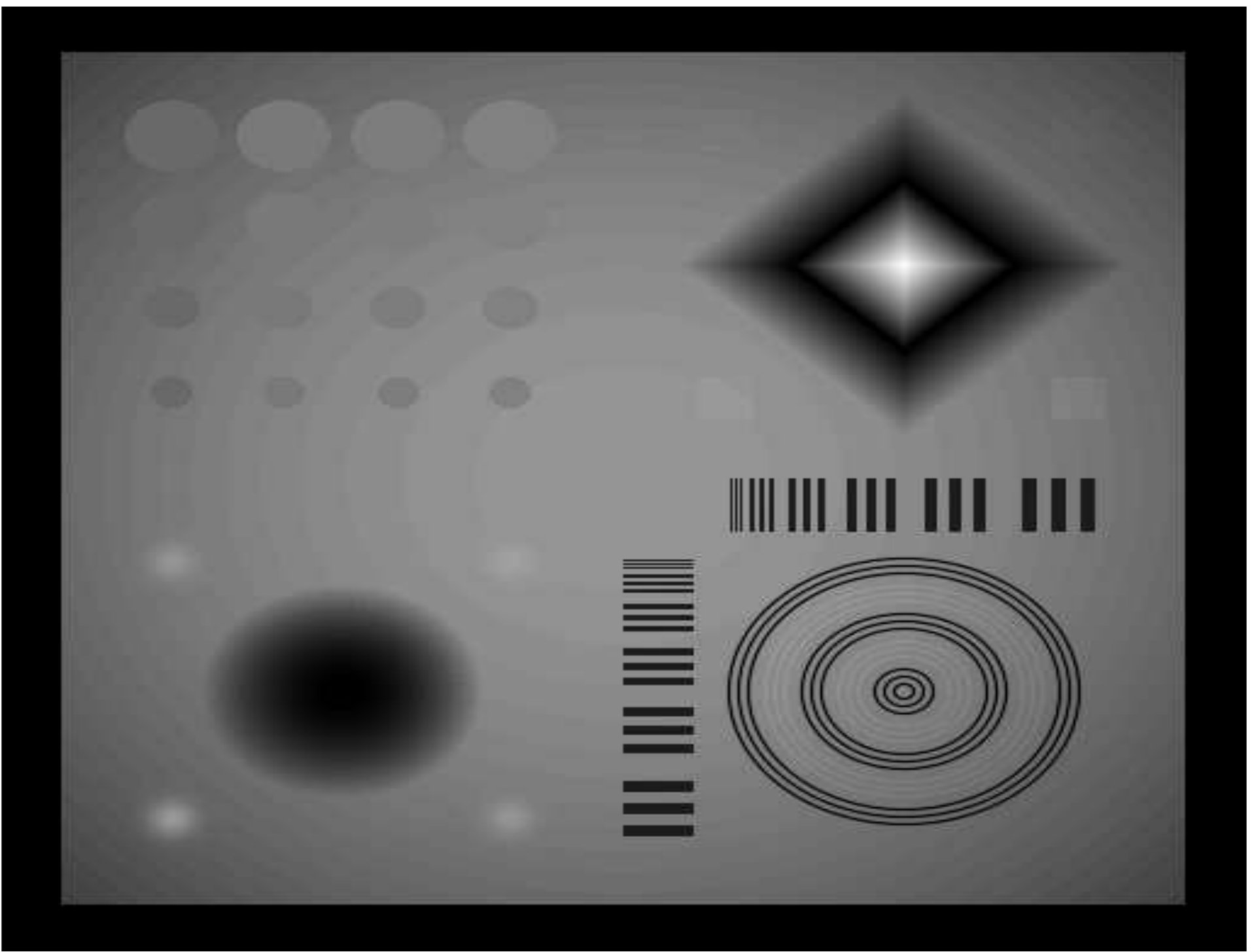}}
\subfigure{\includegraphics[height=60mm,width=60mm]{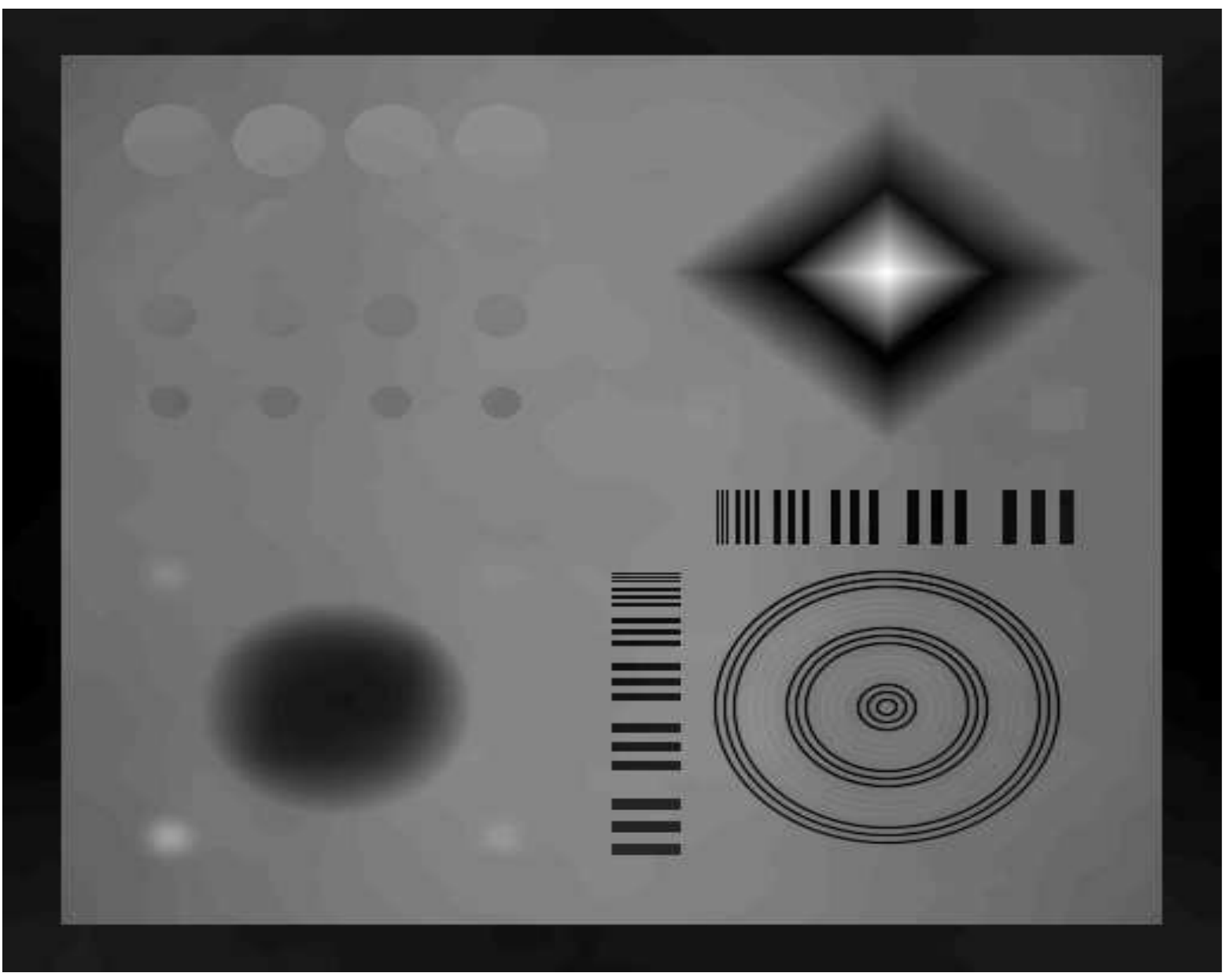}}
\caption{(left) Images used in compressed sensing tests. (right) The images reconstructed with the PMM. The Shepp-Logan phantom (top) was recovered with $25\%$ sampling and the CS-Phantom (bottom) with $50\%$.}
\label{fig:recovered phantoms}
\end{figure} 

For all the experiments we used $\zeta=500$, $\rho=1.5$ and $\lambda=1$, since we found that these choices were effective for both methods. 

The performance of the PMM and ADMM can be seen in Figure \ref{fig:residuals cs}, which reports the residuals curves for both methods, as were the error $\norm{u^k-u^\ast}$, where $u^\ast$ is the exact solution. Observe that the primal curves for both methods are very similar along all iterations. 
However, the decay for the dual residual curve for the PMM is much faster than the dual residual for the ADMM.

\begin{figure}[h]
\centering
\subfigure{\includegraphics[height=55mm,width=70mm]{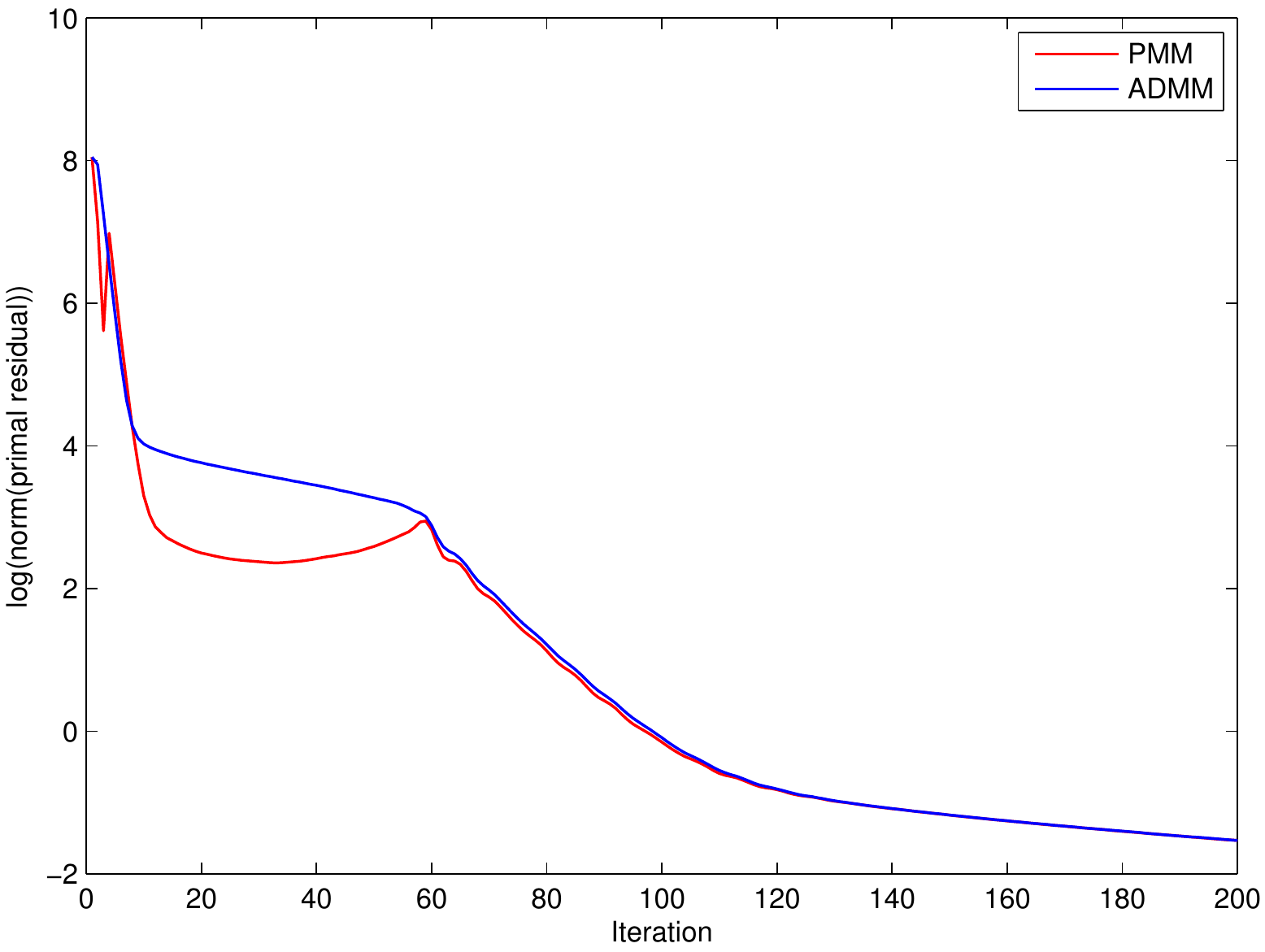}}
\subfigure{\includegraphics[height=55mm,width=70mm]{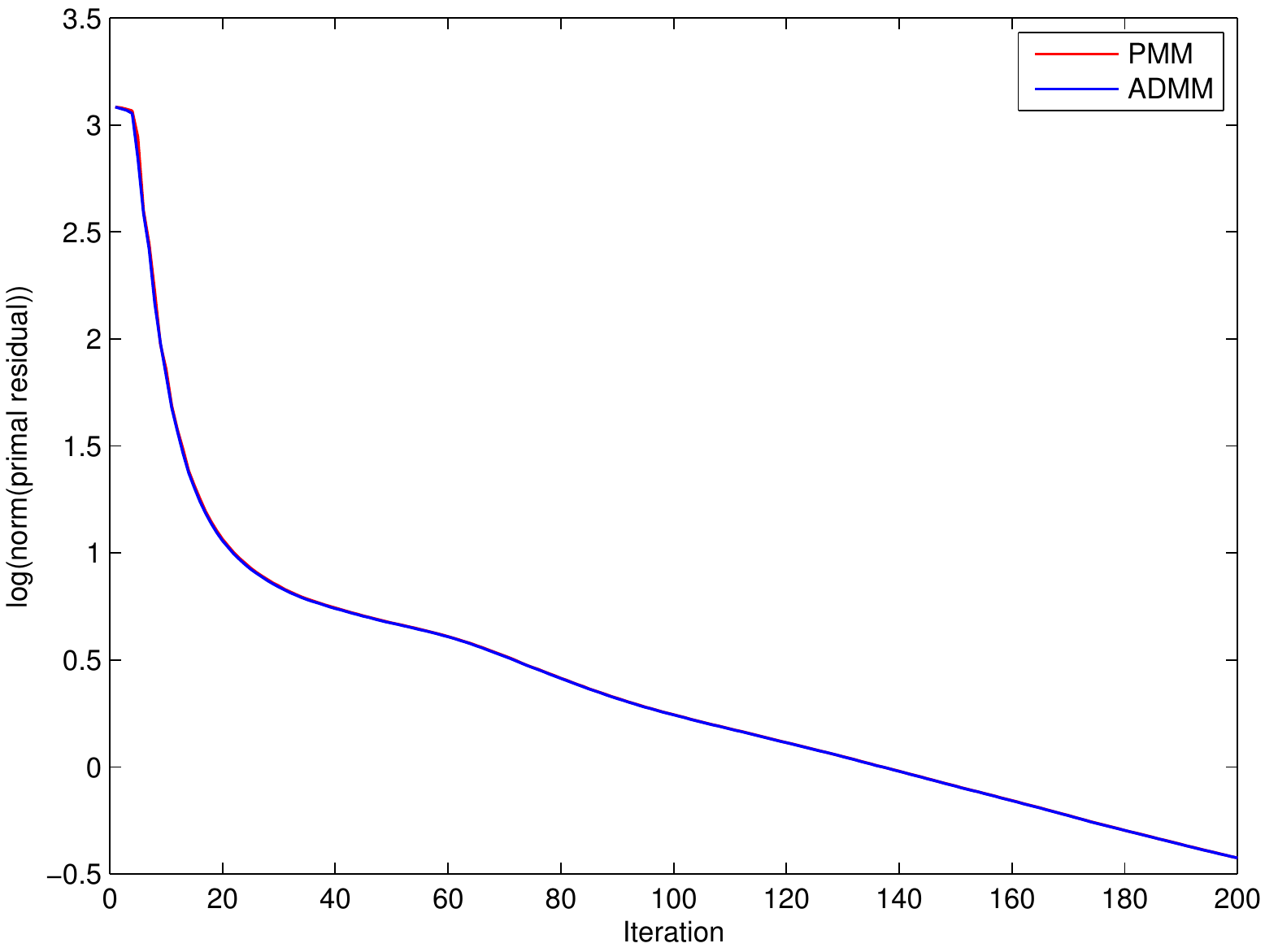}}
\subfigure{\includegraphics[height=55mm,width=70mm]{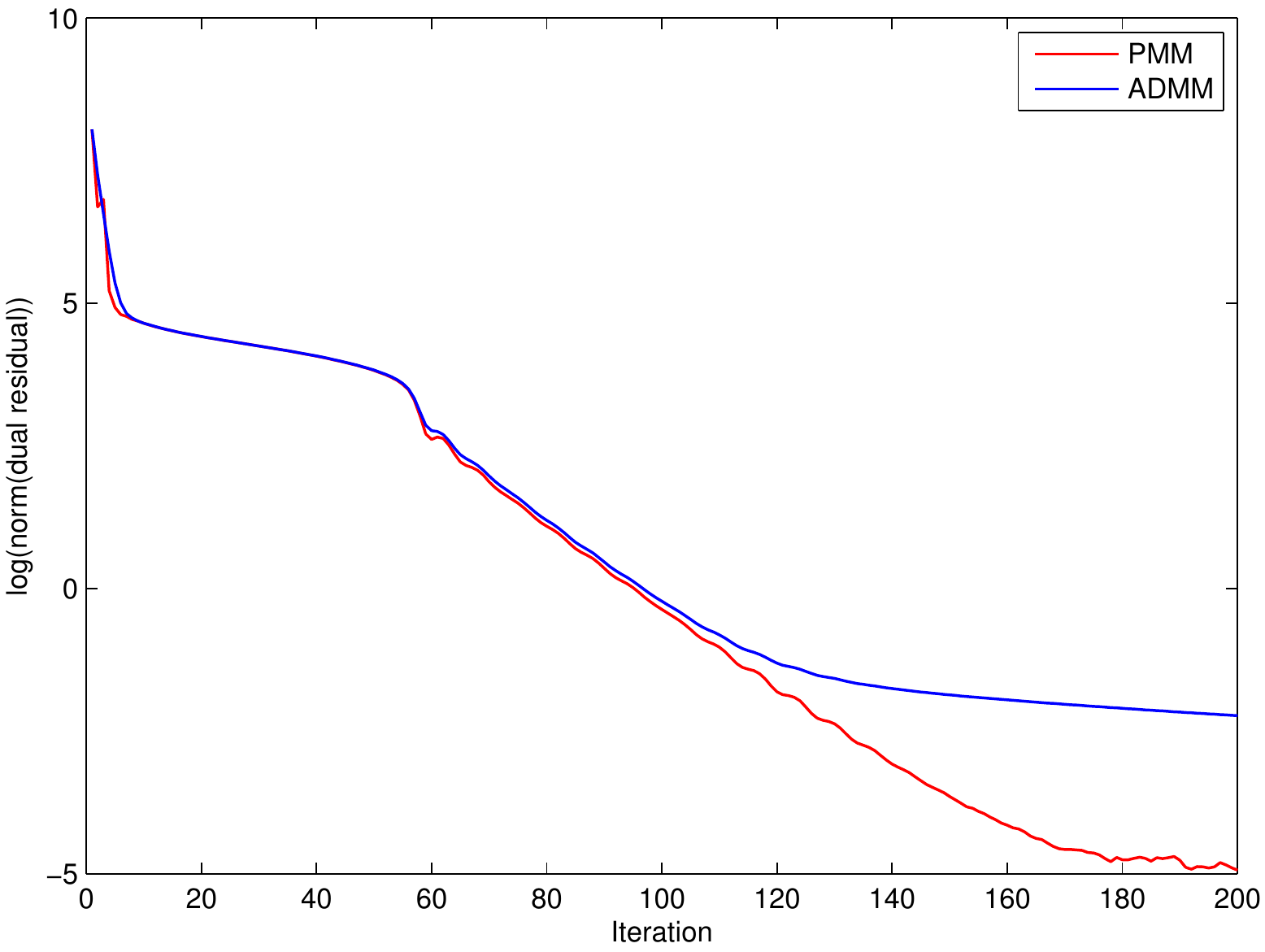}}
\subfigure{\includegraphics[height=55mm,width=70mm]{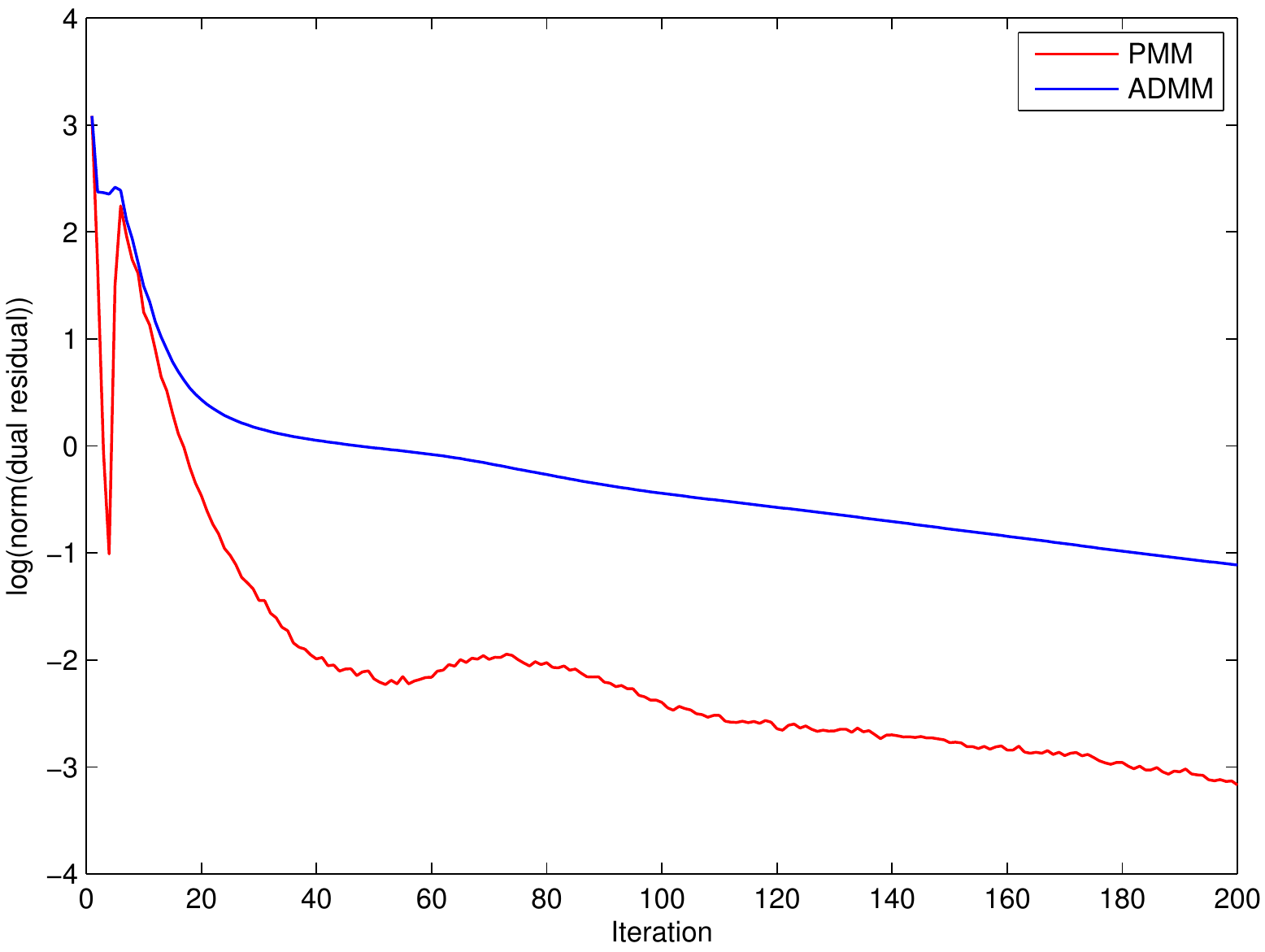}}
\subfigure{\includegraphics[height=55mm,width=70mm]{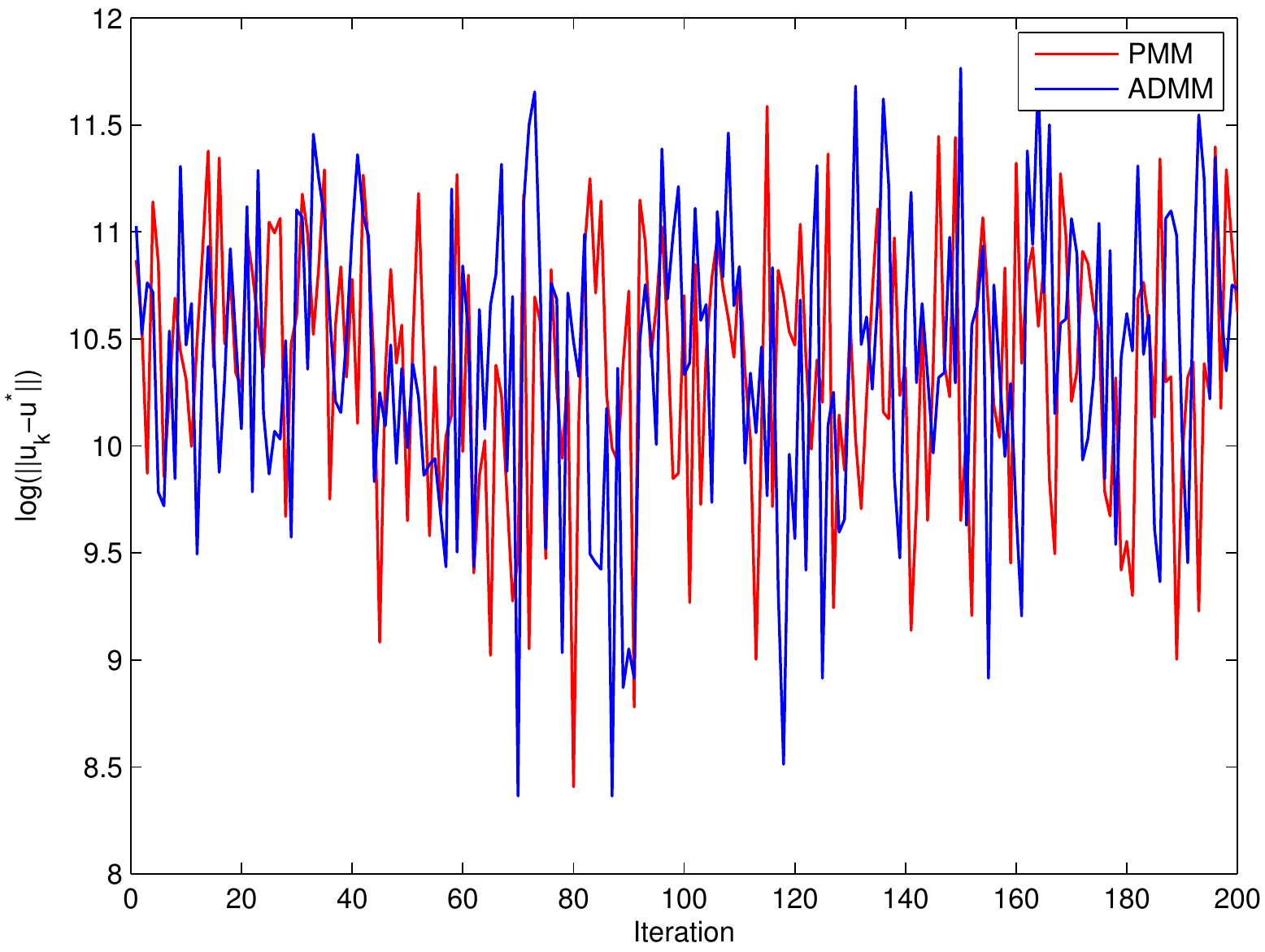}}
\subfigure{\includegraphics[height=55mm,width=70mm]{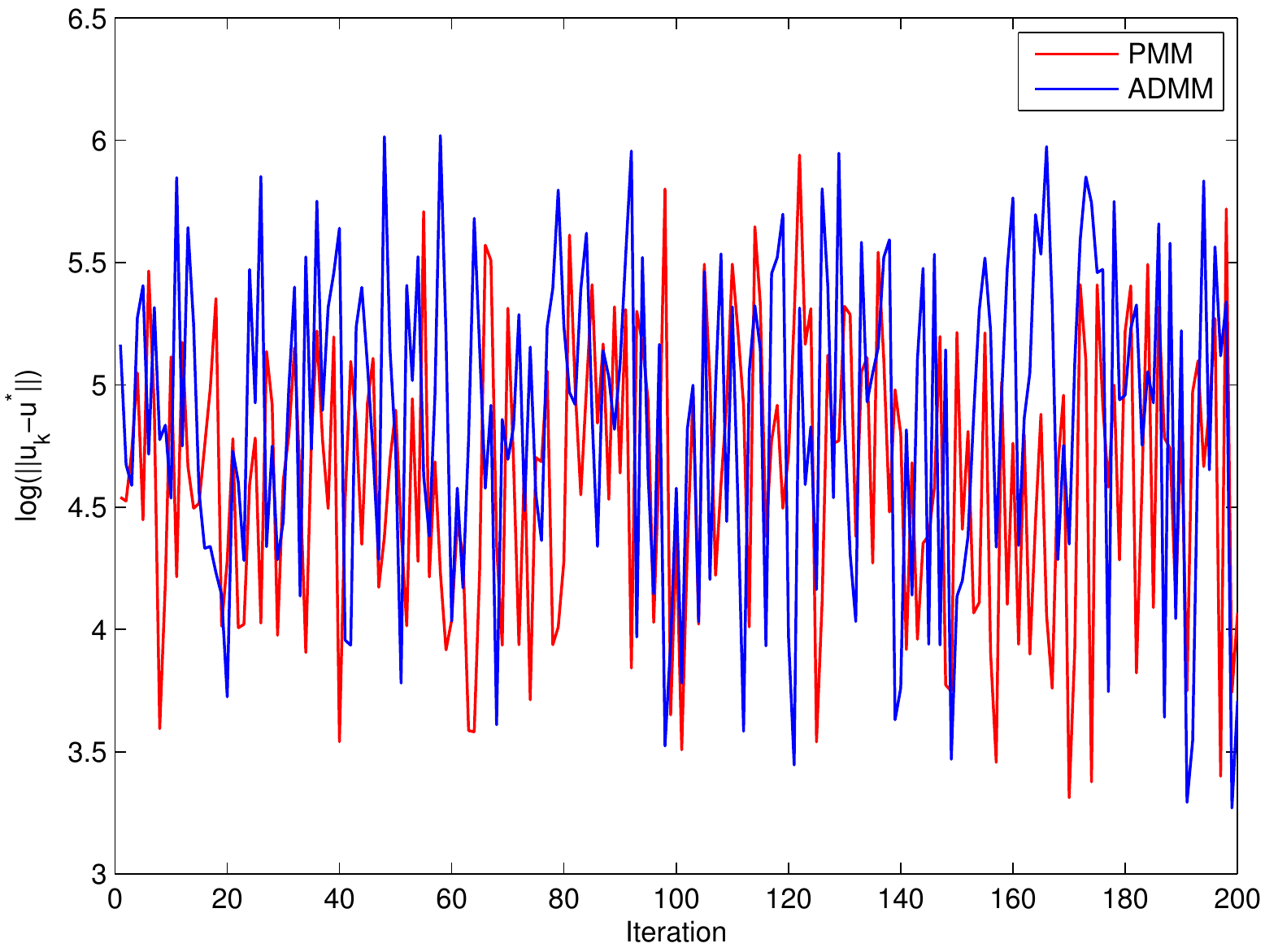}}
\caption{Residuals curves of the PMM and ADMM for the compressed sensing problems. (top) Primal error $\norm{\nabla u_k - v_k}$ vs iteration number. (center) Dual error $\norm{x_k - y_k}$ vs iteration number. (bottom) Error $\norm{u_k-u^\ast}$ vs iteration number ($u^\ast$ is the exact solution). (left) Convergence results are for the Shepp-Logan phantom. (right) Convergence results are for the CS-Phantom.}
\label{fig:residuals cs}
\end{figure}

\subsection{The dual residual}
It was observed in our numerical experiments that, despite the overall rate of decrease for the PMM and the ADMM are very similar, the dual variable in the PMM sequence is smaller than the ADMM dual variable. This could be an advantage for the PMM, and motivates us to study the performance of the method using a stopping criterion based on the dual residual.

In this subsection we present some preliminary computational results 
considering a termination condition that only uses information from the dual residual sequences.
We use as test problems the TV \eqref{eq:tv problem} and CS \eqref{eq:cs problem} problems discussed in the previous subsections. The algorithms were run until condition 
\begin{equation}
\label{eq:stop-dual}
\norm{d_k}/(m*n)\leq10^{-6}
\end{equation}
was satisfied, where $d_k$ is the corresponding dual residual of the sequence at iteration $k$, and $m$ and $n$ are the dimensions of the images. In all the experiments we fixed $\lambda=1$.

Table \ref{tab:dual-stop-crit} presents the number of iterations and time in seconds required for the PMM and ADMM to solve the problems in the experiments. We observe that the performances of the PMM and ADMM using criterion 
\eqref{eq:stop-dual} are very similar in processing time and number of iterations when $\rho=1$. However, for $\rho>1$ the PMM is generally much faster than ADMM. We also notice that the PMM accelerates for $\rho>1$, when compared to the $\rho=1$ case, which does not always occur for the ADMM.

Figures \ref{fig:lena-dual-res} and \ref{fig:logan-dual-res} show the image reconstruction results for some tests. It can be observed in Figure \ref{fig:lena-dual-res} that for the TV problem both methods recover good images 
using \eqref{eq:stop-dual}. This is not surprising since the stopping criterion used in subsection \ref{subs:tv} is more flexible than \eqref{eq:stop-dual}, and the restoration results were satisfactory (see subsection \ref{subs:tv}). 
It turns out that for the CS problem, although the termination condition considered in subsection \ref{subs:cs} is more restrictive than \eqref{eq:stop-dual}, the PMM and ADMM can also reconstruct images with good quality using this last stopping criterion, 
as can be seen in Figure \ref{fig:logan-dual-res}.

\begin{table}[h]
\begin{center}
\begin{tabular}{l|lr|lr}
& \multicolumn{2}{c|}{PMM} & \multicolumn{2}{c}{ADMM}\\
\hline
\multicolumn{1}{c|}{Problem} & $\#$ It & time(s) & $\#$ It & time(s)\\
\hline
TV(Man, $\zeta=20$, $\rho=1$, $\sigma=0.03$) & 98 & 154.708&114 &161.859\\
TV(Man, $\zeta=20$, $\rho=1.8$, $\sigma=0.03$)&71&56.753 &79&55.086\\
TV(Lena, $\zeta=40$, $\rho=1$, $\sigma=0.04$)&248&74.617 &289&74.603\\
TV(Lena, $\zeta=40$, $\rho=1.5$, $\sigma=0.04$)&184&60.477 &418&89.560\\
TV(Baboon, $\zeta=20$, $\rho=1$, $\sigma=0.01$)&137&45.185 &148&45.251\\
TV(Baboon, $\zeta=20$, $\rho=1.3$, $\sigma=0.01$)&101&34.016 &170&44.787\\
CS(Shepp-Logan, $\zeta=500$, $\rho=0.8$, $25\%$)&193&28.623&273&46.412\\
CS(Shepp-Logan, $\zeta=500$, $\rho=1$, $25\%$)&160&23.508 &160&27.055\\
CS(Shepp-Logan, $\zeta=500$, $\rho=1.3$, $25\%$)&140&21.524 &229&36.013\\
CS(Shepp-Logan, $\zeta=500$, $\rho=1.6$, $25\%$)&138&16.998 &338&45.221\\
\hline
\end{tabular}
\end{center}
\caption{Performance results using stooping criterion \eqref{eq:stop-dual}.}
\label{tab:dual-stop-crit}
\end{table}

\clearpage
\begin{figure}[h]
\centering
\subfigure{\includegraphics[height=42mm,width=42mm]{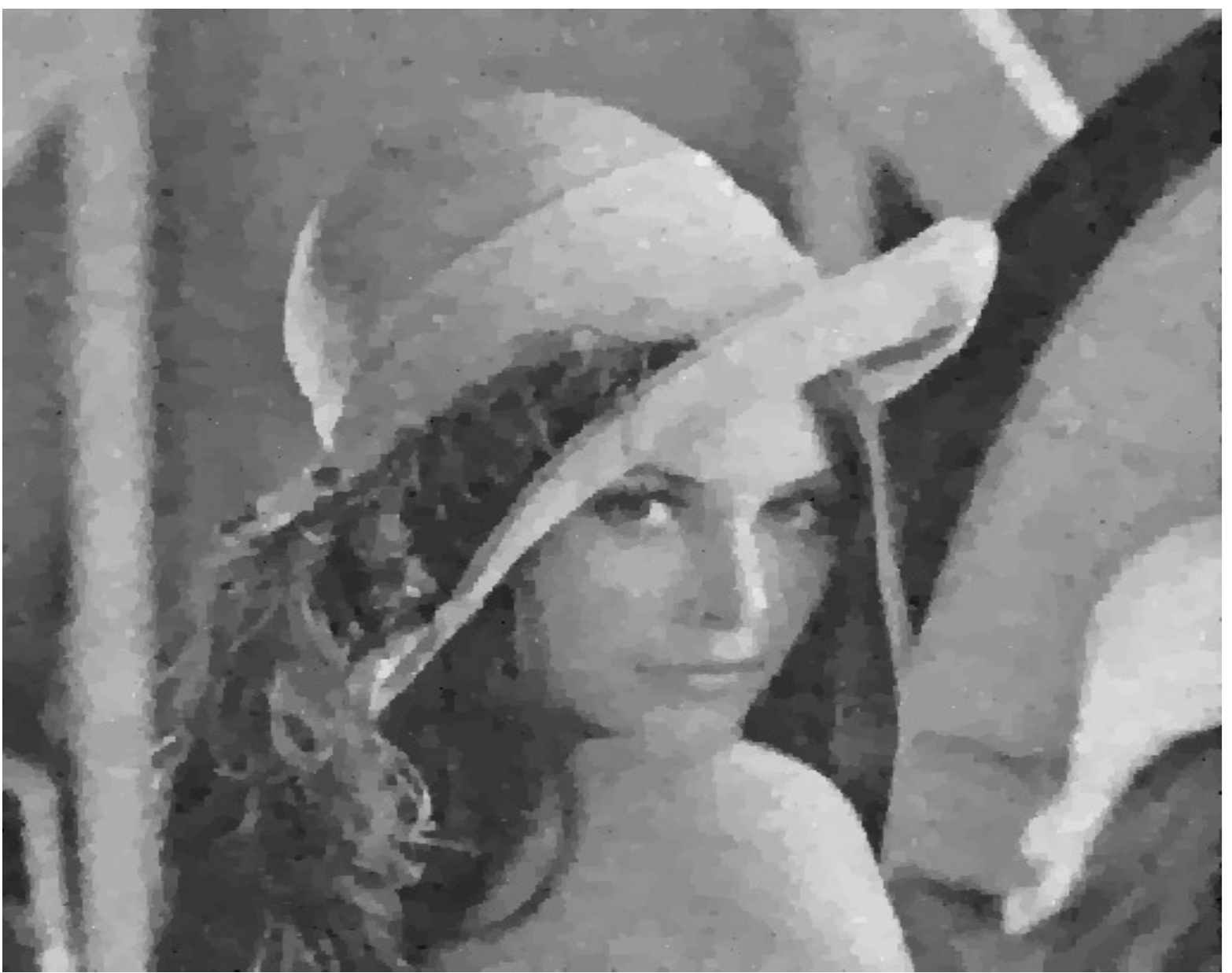}}
\subfigure{\includegraphics[height=42mm,width=42mm]{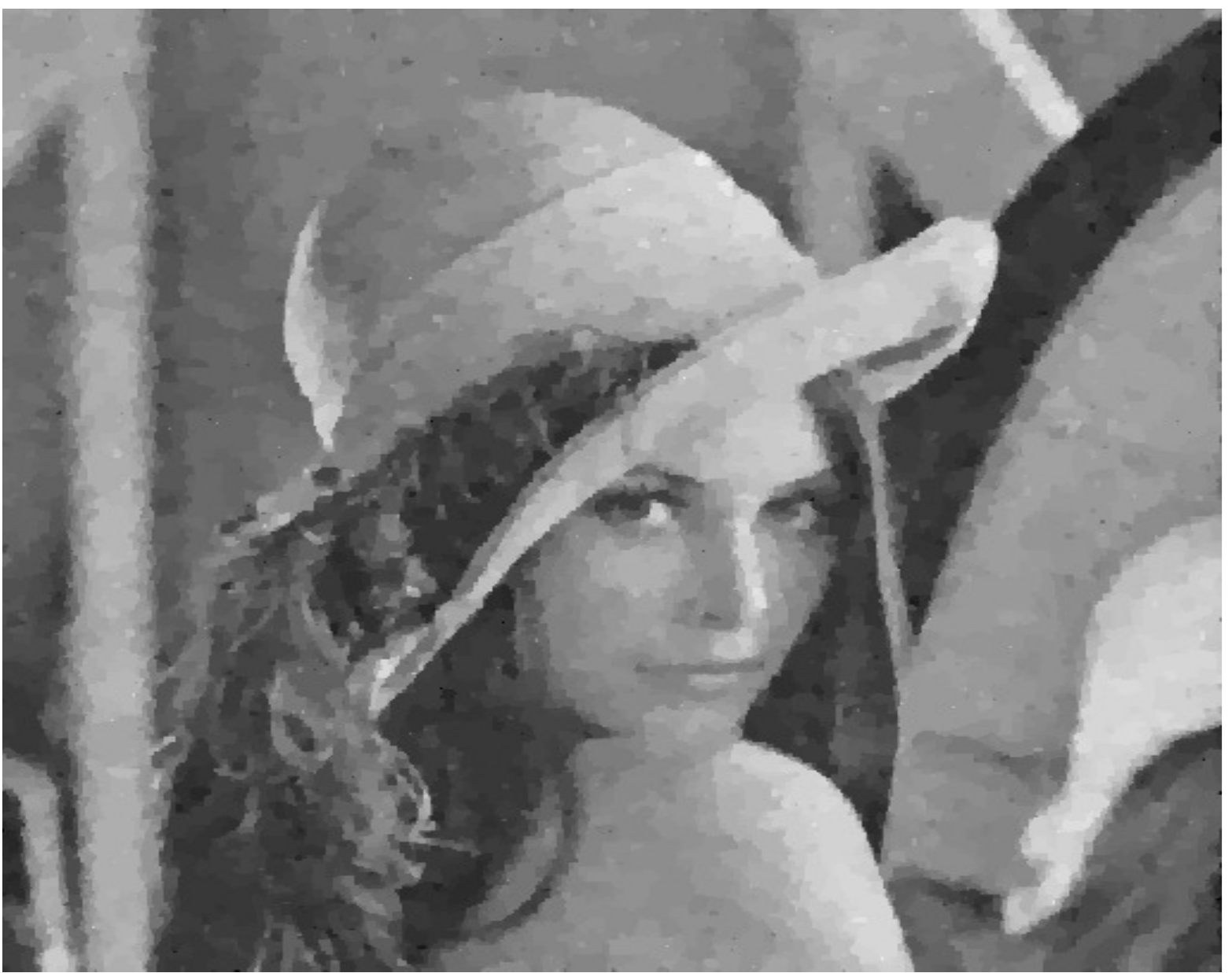}}
\caption{TV problem for the test image Lena, which was contaminated with Gaussian noise with variance $\sigma=0.04$. (left) Image denoised with PMM. (right) Image denoised with ADMM. The image was denoised using $\zeta=40$ and $\rho=1$.}
\label{fig:lena-dual-res}
\end{figure}
\begin{figure}[h]
\centering
\subfigure{\includegraphics[height=42mm,width=42mm]{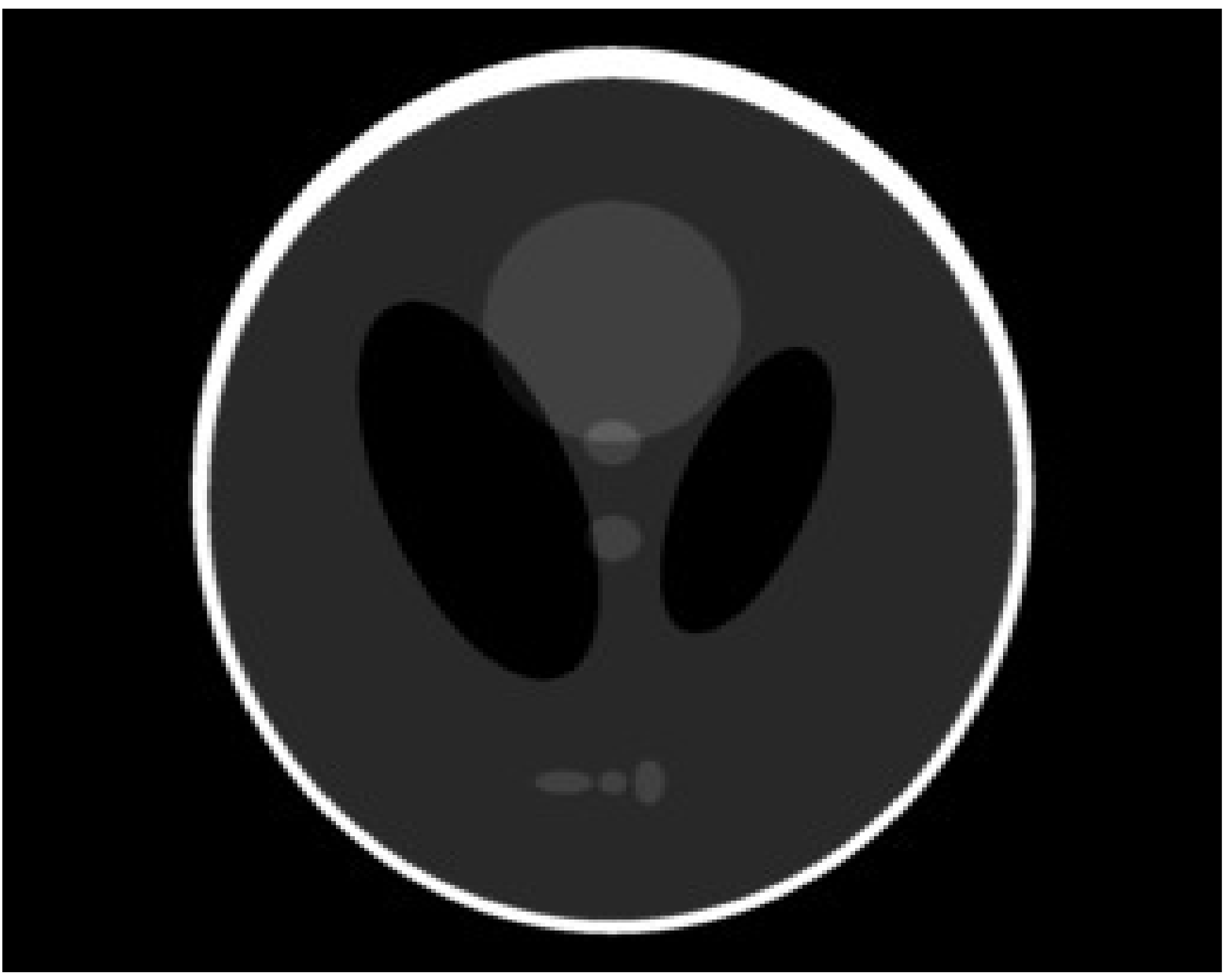}}
\subfigure{\includegraphics[height=42mm,width=42mm]{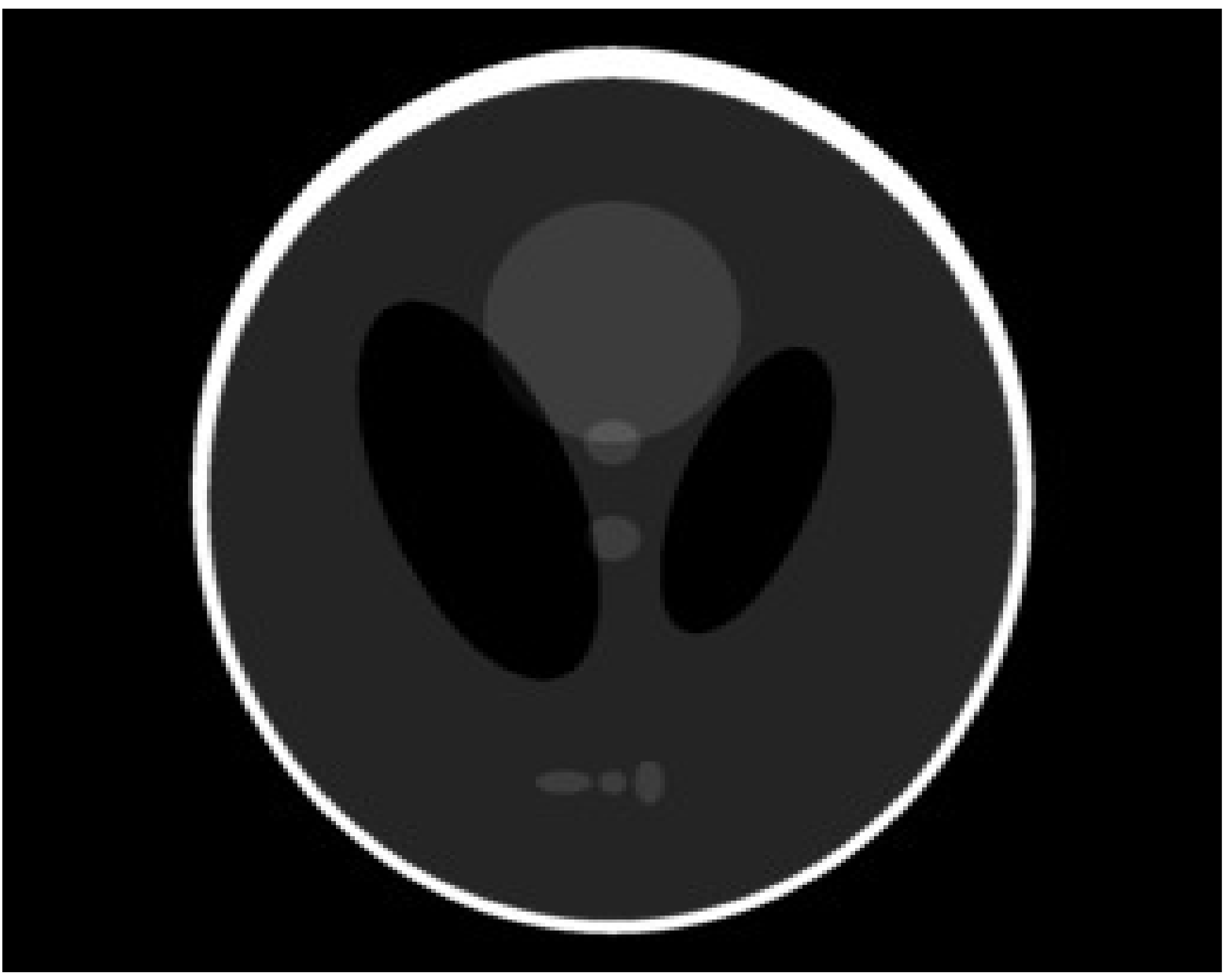}}
\caption{Compressed sensing problem for the test image Shepp-Logan phantom with $25\%$ sampling. (left) Image recovered with the PMM. (right) Image recovered with the ADMM. In the experiments were used $\zeta=500$ and $\rho=1.3$.}
\label{fig:logan-dual-res}
\end{figure}

\section*{Acknowledgements}
The author would like to thank Carlos Antonio Galeano Ríos and Mauricio Romero Sicre for the many helpful suggestions on this paper, which have improved the exposition considerably.

 \bibliographystyle{acm}
\bibliography{projective_method_of_multipliers}
\end{document}